\documentclass[a4paper,11pt]{amsart}%
\usepackage{amsmath}%
\usepackage{amssymb}%
\usepackage{amscd}%
\usepackage{mathrsfs}%
\usepackage[all]{xy}%
\usepackage{enumerate}%

\theoremstyle{plain}%
    \newtheorem{thm}{Theorem}[section]%
    \newtheorem{prop}[thm]{Proposition}%
    \newtheorem{lem}[thm]{Lemma}%
	\newtheorem{cor}[thm]{Corollary}%
    \newtheorem{con}[thm]{Conjecture}%

\theoremstyle{definition}%
    \newtheorem{defn}[thm]{Definition}%

\theoremstyle{remark}%
    \newtheorem{rem}[thm]{Remark}%
    \newtheorem*{ack}{Acknowledgements}%

\def\endpiece{xxx}%						
\def\makeop[#1]{\xmakeop#1,xxx,}%					Ex. \makeop[Hom,Spec]
\def\mkop#1{\expandafter\def\csname #1\endcsname{{\mathrm{#1}}}} % 
\def\xmakeop#1,{\def\temp{#1}\ifx\temp\endpiece\else\mkop{#1}\expandafter\xmakeop\fi}%

\def\N{{\mathbb{N}}}%
\def\Z{{\mathbb{Z}}}%
\def\Q{{\mathbb{Q}}}%
\def\F{{\mathbb{F}}}%
\def\R{{\mathbb{R}}}%
\def\C{{\mathbb{C}}}%
\def\A{{\mathbb{A}}}%
\def\p{$p$\nobreakdash}%
\def\q{$q$\nobreakdash}%
\def\pK{\widetilde{K}}%
\def\id{\mathrm{id}}%
\def\iw#1{\Lambda (#1)}%
\def\riw#1{\Omega (#1)}%
\def\piw#1{\widetilde{\Lambda} (#1)}%
\def\conj#1{\mathrm{Conj}(#1)}%
\def\mod#1{\mathrm{mod}\, #1}%
\def\pmod#1{\widetilde{\mathrm{mod}}\, #1}%
\def\ind#1#2#3{\mathrm{Ind}^{#1}_{#2}(#3)}%
\def\cyc{\mathrm{cyc}}%
\def\Spec{\mathrm{Spec}\,}%

\makeop[Tr,Nr,Ver,Ker,Aut,det,sgn,End,ev,tr]%
\makeop[GL,SL,E,M,K,SK,Gal,id,inn,aug,can,Hom]%

\begin{document}

\setcounter{tocdepth}{1} %%%%%%%%%%%%%
\numberwithin{equation}{section}%%%%%%%%%%%%%%%

%%%%%%%%%%%%%%%%%%%%%%%%%%%%%%%%%%%%%%%%%%%%%%%%%%%%%%%%%
%%----Title----------------------------------------------

\title[Inductive construction of $p$-adic zeta functions]{Inductive
construction of the $p$-adic zeta functions for non-commutative
$p$-extensions of totally real fields with exponent $p$}
\author{Takashi Hara}
\thanks{The author was supported by Japan Society for the Promotion of Science 
(JSPS), Grant-in-Aid for JSPS fellows ($21\cdot 7079$) 
and partially supported by JSPS Core-to-Core Program 18005,
``New Developments of Arithmetic Geometry, Motive, Galois Theory, and
Their Practical Applications.''}
\address{Graduate School of Mathematical Sciences, The University of
Tokyo, 8-1 Komaba 3-chome, Meguro-ku, Tokyo, 153-8914, Japan}
\email{thara@ms.u-tokyo.ac.jp}
\date{\today}

\begin{abstract}
We construct the \p-adic zeta function for
a one-dimensional (as a \p-adic Lie extension) 
non-commutative \p-extension $F_{\infty}$ 
of a totally real number field $F$ such that the finite part
 of its Galois group $G$ is a \p-group with exponent~$p$. We first 
 calculate the Whitehead groups of the Iwasawa algebra $\iw{G}$ and 
 its canonical Ore localisation $\iw{G}_S$ 
 by using Oliver-Taylor's theory upon integral logarithms. 
 This calculation reduces the existence of the non-commutative 
 \p-adic zeta function
 to certain congruence conditions among abelian \p-adic zeta pseudomeasures. 
 Then we finally verify these
 congruences by using Deligne-Ribet's theory and certain inductive technique.
 As an application we shall prove a special case of (the \p-part of) 
 the non-commutative equivariant Tamagawa number conjecture for critical 
 Tate motives. The main results of this paper give generalisation of 
 those of the preceding paper of the author.
\end{abstract}

\maketitle
\setcounter{section}{-1}

%%%%%%%%%%%%%%%%%%%%%%%%%%%%%%%%%%%%%%%%%%%%%%%%%%%%%%%%%
%
%
\section{Introduction} \label{sc:intro}
%
%
%%%%%%%%%%%%%%%%%%%%%%%%%%%%%%%%%%%%%%%%%%%%%%%%%%%%%%%%%

One of the most important topics in non-commutative Iwasawa theory is to 
construct the \p-adic zeta function and to verify the main conjecture, 
as well as in classical theory. 
Up to the present, there have been several
successful examples upon this topic for
\p-adic Lie extensions of totally real number fields: 
the results of J\"urgen Ritter and Alfred Weiss \cite{RW7}, 
Kazuya Kato \cite{Kato2}, Mahesh
Kakde \cite{Kakde1} and so on. In this article, we shall 
construct different type
of example for certain non-commutative \p-extensions 
of totally real number fields.

Let $p$ be a positive odd prime number and $F$ a totally real number field. 
Let $F_{\infty}$ be a totally real \p-adic Lie extension of $F$
which contains the cyclotomic $\Z_p$-extension $F_{\cyc}$ of $F$, and 
assume that all but finitely many primes of $F$ ramify in $F_{\infty}$. 
For a moment we admit Iwasawa's $\mu=0$
conjecture to simplify conditions (see Section \ref{ssc:review-iwa} 
($F_{\infty}$-3) for more general $\mu=0$ condition). 
The aim of this article is to prove the following theorem 
under these conditions:

\begin{thm}[$=$Theorem~$\ref{thm:maintheorem}$] \label{thm:mt}
Let $G$ denote the Galois group of $F_{\infty}/F$.  
 Then for $F_{\infty}/F$ 
 the \p-adic zeta function $\xi_{F_{\infty}/F}$ exists and 
 the Iwasawa main conjecture is true  
 if $G$ is isomorphic to
 the direct product of a finite \p-group $G^f$ with exponent~$p$ and 
 the abelian \p-adic Lie group $\Gamma$ isomorphic
 to $\Z_p$.
\end{thm}
 
We shall review the characterisation of the \p-adic zeta function 
and the precise statement of the non-commutative Iwasawa 
main conjecture in Section~\ref{ssc:review-iwa}.
In the preceding paper \cite{H}, we constructed the \p-adic zeta
function and verified the main conjecture when the Galois group 
$\Gal(F_{\infty}/F)$ 
is isomorphic to the pro-$p$ group
\begin{align*}
\begin{pmatrix} 1 & \F_p & \F_p & \F_p \\ 0 & 1 & \F_p & \F_p \\
 0 & 0 & 1 & \F_p \\ 0 & 0 & 0 & 1 \end{pmatrix} \times \Gamma
\end{align*}
and $p$ is not equal to either $2$ or $3$. Theorem~\ref{thm:mt} 
generalises this result. 

Philosophically the Iwasawa main conjecture 
is closely related to the special values of $L$-functions (as implied by
many people including Kazuya Kato, Annette Huber-Klawitter, Guido Kings, 
David Burns, Mathias Flach,......); hence 
our main theorem (Theorem~\ref{thm:mt}) should also suggest
validity of conjectures upon these values in some sense even
in non-commutative coefficient cases. 
In fact we may verify a special case of (the \p-part of) the equivariant 
Tamagawa number conjecture 
for critical Tate motives with non-commutative coefficient 
combining Theorem~\ref{thm:mt} with descent theory 
established by David Burns and Otmar Venjakob \cite{BurVen}.

\begin{cor}[=Corollary~$\ref{cor:ETNC}$] \label{cor:etnc}
Let $F_{\infty}$ be a \p-adic Lie extension of a totally real number
 field $F$ as in Theorem~$\ref{thm:mt}$ and $F'$ an arbitrary
 finite Galois subextension of $F_{\infty}/F$. 
 Then for an arbitrary natural
 number $r$ divisible by $p-1$, the \p-part of the equivariant 
 Tamagawa number conjecture for $\Q(1-r)_{F'/F}$ is 
 true $($here $\Q(1-r)_{F'/F}=h^0(\Spec F')(1-r)$ denotes 
 the $(1-r)$-fold Tate motive regarded as defined over $F$$)$.
\end{cor} 

This may also be regarded as an analogue of the 
cohomological Lichtenbaum conjecture (in special cases), 
which was proven by Barry Mazur
and Andrew Wiles \cite{MW, Wiles} when $F'$ is 
the same field as $F$---the Bloch-Kato conjecture case--- 
as the direct consequence of the main conjecture 
(for commutative cases) which they verified.
Later we shall give a brief review upon
the formulation of the equivariant Tamagawa number conjecture
for Tate motives (see Section~\ref{ssc:review-tama}).  

Now let us summarise the main idea to prove Theorem~\ref{thm:mt}. 
Consider the family $\mathfrak{F}_B$ of 
all pairs $(U,V)$ 
such that $U$ is an open subgroup of $G$ 
containing $\Gamma$ and 
$V$ is the commutator subgroup of $U$. By classical induction theorem 
of Richard Brauer \cite[Th\'eor\`em~22]{Serre1}, 
an arbitrary Artin representation of $G$
is isomorphic to a $\Z$-linear combination of representations induced by  
characters of abelian groups $U/V$ (as a virtual representation) 
where each $(U,V)$ is in $\mathfrak{F}_B$. Let $\theta_{U,V}$
and $\theta_{S,U,V}$ denote the composite maps 
\begin{align*}
 K_1(\iw{G}) \xrightarrow{\text{norm}} K_1(\iw{U})
 \xrightarrow{\text{canonical}} \iw{U/V}^{\times}, \\
K_1(\iw{G}_S) \xrightarrow{\text{norm}} K_1(\iw{U}_S)
 \xrightarrow{\text{canonical}} \iw{U/V}_S^{\times}
\end{align*}
for each $(U,V)$ in $\mathfrak{F}_B$ where $\iw{G}_S$ (resp.\
$\iw{U}_S$, $\iw{U/V}_S$) is 
{\em the canonical Ore localisation} of the Iwasawa algebra $\iw{G}$ 
(resp.\ $\iw{U}$, $\iw{U/V}$) introduced 
in \cite[Section~2]{CFKSV} (see also Section~1 in this article). Set
$\theta=(\theta_{U,V})_{(U,V)\in \mathfrak{F}_B}$ and  
$\theta_S=(\theta_{S,U,V})_{(U,V)\in \mathfrak{F}_B}$. 

Let $F_U$ (resp.\ $F_V$) be the maximal subfield of $F_{\infty}$ fixed
by $U$ (resp.\ $V$). Then {\em the \p-adic zeta function} exists 
for each {\em abelian} extension $F_V/F_U$; 
Pierre Deligne and Kenneth Alan Ribet first constructed it 
\cite{DR}, and by using their results 
Jean-Pierre Serre reconstructed it as a unique element 
$\xi_{U,V}$ in the total quotient ring of $\iw{U/V}$ 
which satisfies certain {\em interpolation formulae} \cite{Serre2}. 
Now suppose that there
exists an element $\xi$ in $K_1(\iw{G}_S)$ which satisfies the equation
\begin{align} \label{eq:brauer}
\theta_S(\xi)=(\xi_{U,V})_{(U,V)\in \mathfrak{F}_B}.
\end{align} 
Then we may verify by Brauer induction that $\xi$ satisfies 
an interpolation formula which characterises $\xi$ 
as {\em the \p-adic zeta 
function for $F_{\infty}/F$}. 
This observation motivates us to prove that $(\xi_{U,V})_{(U,V)\in
\mathfrak{F}_B}$ is contained in the image of $\theta_S$. It 
seems, however, to be difficult in general 
to characterise the image of the theta map $\theta_S$ completely
for the localised Iwasawa algebra $\iw{G}_S$. 
Therefore we shall first determine the image of 
the theta map $\theta$ for the (integral) Iwasawa algebra $\iw{G}$, 
and then construct an element $\xi$ satisfying (\ref{eq:brauer}) 
by using this calculation and certain diagram chasing. 
The strategy which we introduced here was first proposed by
David Burns (and hence 
we call this method {\em Burns' technique} in this article). 
We shall discuss its details in Section~\ref{sc:burns}.

Let $(U,\{e\})$ be an element in $\mathfrak{F}_B$ such that the 
cardinality of the finite part of $U$ is at most $p^2$ (and $U$ is
hence abelian). Let $I_{S,U}$ denote the image of $\theta_{S,U,\{e\}}$ for 
each of such  $(U,\{e\})$'s. 
By virtue of Burns' technique, we may reduce the condition 
for $(\xi_{U,V})_{(U,V)\in \mathfrak{F}_B}$ to be contained in the
image of $\theta_S$ to the following type of congruence:
\begin{align*}
\xi_{U,\{e \}} & \equiv \varphi(\xi_{G,[G,G]})^{(G:U)/p} &
 \mod{I_{S,U}}
\end{align*}
where $\varphi$ is the Frobenius endomorphism $\varphi \colon
\iw{G^{\mathrm{ab}}}_S \rightarrow \iw{\Gamma}_{(p)}$ induced by the group
homomorphism $G^{\mathrm{ab}}\rightarrow \Gamma; g\mapsto g^p$. Kato, 
Ritter, Weiss and Kakde verified such type of congruence when the index 
$(G:U)$ exactly equals~$p$ \cite{Kato2, RW6, Kakde1} 
by using the theory of Deligne and Ribet upon Hilbert modular forms 
\cite{DR}. It seems, however, 
to be almost impossible to deduce such congruences 
only from Deligne-Ribet's theory when the index 
$(G:U)$ is strictly larger than $p$. 
Nevertheless in Sections~\ref{sc:cong} and~\ref{sc:construction} we 
shall verify these 
congruences by combining Deligne-Ribet's theory with certain {\em inductive 
technique} which was first introduced in \cite{H}.

In computation of the images of $\theta$ and $\theta_S$ 
we use theory upon \p-adic logarithmic homomorphisms.
This causes ambiguity upon 
\p-power torsion elements in the whole calculation, and hence we
have to eliminate this ambiguity as the final step of the proof. 
We shall complete this step by utilising  
the existence of the \p-adic zeta functions 
for Ritter-Weiss-type extensions \cite{RW7} and 
certain inductive arguments.

\medskip
The detailed content of this article is as follows. We shall briefly review 
the basic formulations of the non-commutative Iwasawa main conjecture 
and the equivariant Tamagawa number conjecture in
Section~\ref{sc:review}. Then we  discuss David Burns' 
outstanding strategy for construction 
of the \p-adic zeta function in Section~\ref{sc:burns}.
The precise statement  of our main theorem and its application
will be dealt with in Section~\ref{sc:maintheorem}. 
Sections~\ref{sc:additive}, \ref{sc:prelim} and \ref{sc:trans} are
devoted to computation of the image of (a certain variant of) 
the theta map $\tilde{\theta}$; we first construct 
``the additive theta isomorphism'' $\theta^+$ in
Section~\ref{sc:additive}, and then translate it into the
multiplicative morphism $\tilde{\theta}$ by utilising  
logarithmic homomorphisms in
Section~\ref{sc:trans}. Section~\ref{sc:prelim} is the preliminary
section for Section~\ref{sc:trans}. We study the image of the norm map 
$\tilde{\theta}_S$ for the localised Iwasawa algebra $\iw{G}_S$ 
in Section~\ref{sc:localise}, and derive certain ``weak congruences'' 
upon abelian \p-adic zeta
pseudomeasures in Section~\ref{sc:cong} 
by applying Deligne-Ribet's \q-expansion principle 
\cite{DR} and Ritter-Weiss' approximation technique \cite{RW6}. 
In Section~\ref{sc:construction}  
we refine the congruences obtained in the previous section by using
induction, and construct the \p-adic zeta function ``modulo
\p-torsion.'' We shall finally eliminate ambiguity of 
the \p-power torsion part. 

\tableofcontents

%%%%%%%%%%%%%%%%%%%%%%%%%%%%%%%%%%%%%%%%%%%%%%%%%%%%%%%%%
%
\subsection*{Notation}
%
%%%%%%%%%%%%%%%%%%%%%%%%%%%%%%%%%%%%%%%%%%%%%%%%%%%%%%%%%

In this article, $p$ always denotes a positive prime number. 
We denote by $\N$ the set of natural numbers (the set of {\em
strictly} positive integers). We also denote by $\Z$ (resp.\ $\Z_p$) 
the ring of integers
(resp.\ \p-adic integers).
The symbol $\Q$ (resp.\ $\Q_p$) denotes 
the rational number field (resp.\ the \p-adic number field). 
For an arbitrary group $G$ let $\conj{G}$ denote the set of all
conjugacy classes of $G$.
For an arbitrary pro-finite group $P$, we always denote by $\iw{P}$ 
its Iwasawa algebra over $\Z_p$ and by $\riw{P}$ its Iwasawa algebra
over $\F_p$. More specifically, $\iw{P}$ and $\riw{P}$ are defined by
\begin{align*}
\iw{P}&=\varprojlim_{U} \Z_p[P/U], & \riw{P} &=\varprojlim_U \F_p[P/U]
\end{align*}
where $U$ runs over all open normal subgroups of $P$.
Let $\Gamma$ denote the commutative \p-adic Lie group  
isomorphic to $\Z_p$ (corresponding to the Galois group of the
cyclotomic $\Z_p$-extension). Throughout this paper, we fix a topological 
generator $\gamma$ of $\Gamma$. In other words, we fix Iwasawa-Serre 
isomorphisms
\begin{align*}
\iw{\Gamma} &\xrightarrow{\sim}  \Z_p[[T]], & \riw{\Gamma}
 &\xrightarrow{\sim} \F_p[[T]] \\
\gamma & \mapsto  1+T  & \gamma & \mapsto  1+T
\end{align*}
where $\Z_p[[T]]$ (resp.\ $\F_p[[T]]$) is the formal power series ring
over $\Z_p$ (resp.~$\F_p$). For an arbitrary \p-adic Lie group $W$
isomorphic to the direct product of a finite \p-group and $\Gamma$,
$W^f$ denotes the finite part of $W$.
We always assume that every associative ring has unity. 
The centre of an associative ring $R$ is denoted by $Z(R)$.
For an associative ring $R$, we denote by $\M_n(R)$ 
the ring of $n\times n$-matrices with entries
 in $R$ and by $\GL_n(R)$ the multiplicative group of $\M_n(R)$. 
We always consider that all Grothendieck groups are 
{\em additive} abelian groups, 
whereas all Whitehead groups are {\em multiplicative} abelian groups.
For an arbitrary multiplicative abelian group $A$, let
$A_{p\text{-tors}}$ (resp.\ $A_{\text{tors}}$) denote
the \p-power torsion part (resp.\ the torsion part) of $A$. 
We set $\pK_1(R)=K_1(R)/K_1(R)_{p\text{-tors}}$ 
for an arbitrary associative ring $R$. 
Similarly we set 
$\piw{P}^{\times}= \iw{P}^{\times}/\iw{P}^{\times}_{p\text{-tors}}$
for an arbitrary pro-finite group~$P$.

\begin{ack}
 The author would like to express his sincere gratitude to Professor 
 Takeshi Tsuji for much fruitful discussion and many helpful comments 
 (especially the suggestion
 that we use augmentation ideals in the translation of the additive theta 
 isomorphism, see Section~\ref{ssc:aug} for details). 
 He is also grateful to Mahesh
 Kakde for useful comments upon the preliminary version of this article.
\end{ack}

%%%%%%%%%%%%%%%%%%%%%%%%%%%%%%%%%%%%%%%%%%%%%%%%%%%%%%%%%
%
%
\section{Reviews upon non-commutative arithmetic theory} \label{sc:review}
%
%
%%%%%%%%%%%%%%%%%%%%%%%%%%%%%%%%%%%%%%%%%%%%%%%%%%%%%%%%%

In this section we shall review the formulations of 
the non-commutative Iwasawa main conjecture (only for the cases of 
totally real number fields) and the equivariant Tamagawa number
conjecture (only for the Tate motives). 
Refer to \cite{Bass, Swan} for basic results upon (low-dimensional) 
algebraic \mbox{$K$-theory} used in this section.

%%%%%%%%%%%%%%%%%%%%%%%%%%%%%%%%%%%%%%%%%%%%%%%%%%%%%%%%%
%
\subsection{Non-commutative Iwasawa theory for totally real fields} \label{ssc:review-iwa}
%
%%%%%%%%%%%%%%%%%%%%%%%%%%%%%%%%%%%%%%%%%%%%%%%%%%%%%%%%%

In this subsection we review the formulation of 
the non-commutative Iwasawa main conjecture
for totally real number fields following 
John Henry Coates, Takako Fukaya, Kazuya Kato, Ramdorai Sujatha and Otmar 
Venjakob \cite{CFKSV, FK}. 
We remark that J\"urgen Ritter and Alfred Weiss also formulated 
the non-commutative Iwasawa main conjecture 
---``the `main conjecture' of equivariant Iwasawa theory'' in their 
terminology--- in somewhat different manner \cite{RW1, RW2, RW3, RW4}.
Let $F$ be a totally real number field and $p$ a positive odd prime
number. Let $F_{\infty}$ be a totally real \p-adic Lie extension of $F$
 satisfying the following three conditions:
\begin{enumerate}[$(F_{\infty}\text{-}1)$]
\item the cyclotomic $\Z_p$-extension $F_{\cyc}$
      of $F$ is contained in $F_{\infty}$;
\item only finitely many primes of $F$ ramify in $F_{\infty}$;
\item there exists a finite subextension $F'$ of $F_{\infty}/F$ such
      that $F_{\infty}/F'$ is pro-$p$ and 
	  the Iwasawa $\mu$-invariant for its cyclotomic
      $\Z_p$-extension ${F'}_{\cyc}/F'$ equals zero.
\end{enumerate}
Fix a finite set $\Sigma$ of finite places of $F$ which contains all of those 
ramifying in~$F_{\infty}$. For an arbitrary algebraic extension
$E$ of $F$, we shall denote $\Sigma^{\vee}$ the union of 
the set $\Sigma_{\infty}$ of all infinite places and 
the set $\Sigma_E$ of all finite places above $\Sigma$.
Set $G=\Gal(F_{\infty}/F)$,
$H=\Gal(F_{\infty}/F_{\cyc})$ and $\Gamma=\Gal(F_{\cyc}/F)$. 
The pro-$p$ group $\Gamma$ is isomorphic to $\Z_p$ by definition. 

Let $S$ be the subset of $\iw{G}$ consisting of all elements $f$ such
 that the quotient module $\iw{G}/\iw{G}f$ is finitely generated as a left
 $\iw{H}$-module. The set $S$ is a left and right Ore set of $\iw{G}$
 with no zero divisors 
 \cite[Theorem~2.4]{CFKSV}, which is called {\em the canonical Ore set for
 $F_{\infty}/F$} (refer to \cite{MR, Sten} for general theory upon 
Ore localisation). The Ore localisation $\iw{G}\rightarrow \iw{G}_S$
 induces the following localisation exact sequence in algebraic \mbox{$K$-theory}
 (due to Charles Weibel, Dongyuan Yao, Alan Johnathan Berrick 
 and Michael Keating \cite{WeibYao, BerKeat}):
\begin{align*}
K_1(\iw{G})\rightarrow K_1(\iw{G}_S) \xrightarrow{\partial}
 K_0(\iw{G},\iw{G}_S)\rightarrow 0.
\end{align*}
Surjectivity of the connecting homomorphism $\partial$ was proven in
\cite[Proposition~3.4]{CFKSV}. Now let $\mathscr{C}^{\mathrm{Perf}}(\iw{G})$ 
denote the category of perfect complexes of finitely generated left
$\iw{G}$-modules (that is, the category of complexes of finitely generated left $\iw{G}$-modules which are quasi-isomorphic to bounded complexes of finitely
generated {\em projective} left $\iw{G}$-modules), and let
$\mathscr{C}^{\mathrm{Perf}}_S (\iw{G})$ denote the full subcategory of
$\mathscr{C}^{\mathrm{Perf}}(\iw{G})$ generated by all objects whose
cohomology groups are $S$-torsion left $\iw{G}$-modules. 
The category $\mathscr{C}^{\mathrm{Perf}}(\iw{G})$ 
(resp.\ $\mathscr{C}^{\mathrm{Perf}}_S(\iw{G})$) is regarded as a 
Waldhausen category $(\mathscr{C}^{\mathrm{Perf}}(\iw{G}), \mathrm{qis})$ 
(resp.\ $(\mathscr{C}^{\mathrm{Perf}}_S(\iw{G}), \mathrm{qis})$) 
equipped with weak equivalences consisting of all quasi-isomorphisms.
Then it is well known that the relative
Grothendieck group $K_0(\iw{G},\iw{G}_S)$ is canonically
identified with the
Waldhausen-Grothendieck group $K_0(\mathscr{C}^{\mathrm{Perf}}_S(\iw{G}),
\mathrm{qis})$. Set 
\begin{align*}
C_{F_{\infty}/F}=R\Hom(R\Gamma_{\mathrm{\acute{e}t}}(\Spec
 \mathcal{O}^{\Sigma^{\vee}}_{F_{\infty}}, \Q_p/\Z_p),
 \Q_p/\Z_p)
\end{align*}
where $\Gamma_{\mathrm{\acute{e}t}}$ denotes the global section functor
for the \'etale topology and
$\mathcal{O}_{F_{\infty}}^{\Sigma^{\vee}}$ denotes the
$\Sigma^{\vee}$-integer ring of $F_{\infty}$. 
Its cohomology groups are calculated as follows:
\begin{align} 
H^i(C_{F_{\infty}/F})= \begin{cases} \Z_p & \text{if }i=0, \\
X_{\Sigma} & \text{if }i=-1, \\ 0 & \text{otherwise}.
\end{cases}
\end{align}
Here $X_{\Sigma}=\Gal(M_{\Sigma}/F_{\infty})$ is the Galois group of
the maximal abelian pro-$p$ extension $M_{\Sigma}$ of $F_{\infty}$
unramified outside $\Sigma$. Note that $\Z_p$ is an $S$-torsion module
since it is finitely generated as a left
$\iw{H}$-module (see \cite[Proposition~2.3]{CFKSV} for details). 
The Galois group $X_{\Sigma}$ is also an
$S$-torsion module by condition $(F_{\infty}\text{-}3)$ 
(due to the lemma of Yoshitaka Hachimori and Romyar Sharifi 
\cite[Lemma~3.4]{HS}, see also \cite[Section~1.2]{H}). Therefore
we may regard $C_{F_{\infty}/F}$ as an object of
$\mathscr{C}^{\mathrm{Perf}}_S(\iw{G})$, and by surjectivity of $\partial$ 
there exists an element
$f_{F_{\infty}/F}$ in $K_1(\iw{G}_S)$ satisfying
\begin{align} \label{eq:char}
\partial(f_{F_{\infty}/F})=-[C_{F_{\infty}/F}],
\end{align}
which is called {\em a characteristic element for $F_{\infty}/F$}. 
Characteristic elements are determined uniquely up to multiplication 
by elements in the image of the canonical homomorphism 
$K_1(\iw{G})\rightarrow K_1(\iw{G}_S)$ (due to  
the localisation exact sequence).

\medskip
We next consider the \p-adic zeta function for
$F_{\infty}/F$. From now on we fix 
an algebraic closure of the \p-adic number field $\overline{\Q}_p$,
and we also fix embeddings of the algebraic closure
$\overline{\Q}$ of the rational number field $\Q$ 
into $\C$ ---the complex number field--- and $\overline{\Q}_p$ 
till the end of this subsection.
By the argument in \cite[p.p.~172--173]{CFKSV}, 
we may define {\em the evaluation map}
\begin{align*}
K_1(\iw{G}_S) \rightarrow \overline{\Q}_p\cup \{ \infty\} ; f
 \mapsto f(\rho)
\end{align*}
for an arbitrary continuous
representation $\rho \colon G\rightarrow \GL_d(\mathcal{O})$ (where
$\mathcal{O}$ is the ring of integers of a certain finite extension of $\Q_p$). 
Now let $L_{\Sigma}(s;F_{\infty}/F,\rho)$ be the complex 
Artin $L$-function associated to an Artin representation~$\rho$ 
(recall that $\rho$ is an Artin representation if its image is finite) 
whose local factors at places belonging to $\Sigma$ are removed.
If there exists an element $\xi_{F_{\infty}/F}$ in $K_1(\iw{G}_S)$ which
satisfies {\em the interpolation formula}
\begin{align} \label{eq:interp}
\xi_{F_{\infty}/F} (\rho\kappa^r)=L_{\Sigma}(1-r; F_{\infty}/F, \rho) 
\end{align}
for an arbitrary Artin representation $\rho$ of $G$ and an arbitrary 
natural number 
$r$ divisible by $p-1$, we call $\xi_{F_{\infty}/F}$ {\em the \p-adic
zeta function for $F_{\infty}/F$}. 
The Iwasawa main conjecture for totally real number fields is 
then formulated as follows:

\begin{con} \label{conj:iwasawa}
Let $p$, $F$ and $F_{\infty}/F$ be as above. Then
\begin{enumerate}[$(1)$]
\item $($existence and uniqueness of the \p-adic zeta function$)$

the \p-adic zeta function $\xi_{F_{\infty}/F}$ for $F_{\infty}/F$ exists uniquely$;$

\item $($the Iwasawa main conjecture$)$

the equation $\partial(\xi_{F_{\infty}/F})=-[C_{F_{\infty}/F}]$ holds.
\end{enumerate}
\end{con} 

%%%%%%%%%%%%%%%%%%%%%%%%%%%%%%%%%%%%%%%%%%%%%%%%%%%%%%%%%
%
\subsection{The equivariant Tamagawa number conjecture for Tate motives}
\label{ssc:review-tama}
%
%%%%%%%%%%%%%%%%%%%%%%%%%%%%%%%%%%%%%%%%%%%%%%%%%%%%%%%%%

The Tamagawa number conjecture, which predicts the special values of 
the $L$-functions associated to motives in terms of 
``motivic Tamagawa numbers,'' was first formulated by 
Spencer Bloch and Kazuya Kato \cite{BlKat} 
following the earlier results of Pierre Deligne 
\cite{Deligne1} and Alexander Be\u\i linson \cite{Beilinson}. 
Then Jean-Marc Fontaine, Berdenette Perrin-Riou \cite{FPR} and 
Kazuya Kato \cite{Kato1} gave its reformulation in a somewhat
sophisticated way. The equivariant version of 
the Tamagawa number conjecture was finally 
formulated by David Burns and Mathias Flach \cite{BurFl3} 
which included cases with non-commutative coefficient.
In this subsection we shall give a brief review 
of the formulation of the equivariant Tamagawa number conjecture
for Tate motives.

First recall the notion of {\em the determinant functor} 
\begin{align}
\det_R \colon \mathscr{C}^{\mathrm{Perf}}(R)_{\mathrm{qis}} \rightarrow
 V(R)
\end{align}
which was constructed by Pierre Deligne \cite{Deligne2} 
for an arbitrary associative ring 
(and by Finn Faye Knudsen and David Mumford  \cite{KnudMum} 
when $R$ is commutative)
where $\mathscr{C}^{\mathrm{Perf}}(R)_{\mathrm{qis}}$ denotes 
the subcategory of $\mathscr{C}^{\mathrm{Perf}}(R)$
---the category of perfect complexes of
finitely generated left $R$-modules--- whose objects are the same 
as $\mathscr{C}^{\mathrm{Perf}}(R)$ and morphisms are restricted to
quasi-isomorphisms. The (small) Picard category 
$V(R)$ constructed by Deligne is called {\em the category of virtual objects
associated to $R$}, which satisfies 
associative and commutative constraints. 
In fact  Deligne has constructed the category of 
virtual objects for an arbitrary exact category in \cite{Deligne2} 
by using Daniel Quillen's $S$-construction \cite{Quillen}, 
but we shall omit the details. 
Here we just remark that Takako Fukaya and Kazuya Kato  
gave an alternative and direct construction of $V(R)$ \cite{FK}.
In this article we shall never use the explicit description of $V(R)$.

The determinant functor $\det_R$ enjoys the following properties 
(here we denote the product structure of the Picard category $V(R)$ 
by $\cdot$\ ):
\begin{enumerate}[I)]
\item for an arbitrary exact sequence 
$0 \rightarrow C^{\prime} \rightarrow C \rightarrow
 C^{\prime \prime} \rightarrow 0$ 
in the category $\mathscr{C}^{\mathrm{Perf}}(R)$, 
the determinant functor $\det_R$ induces a canonical isomorphism
\begin{align*}
\det_R(C) \xrightarrow{\sim} \det_R (C^{\prime}) \cdot
 \det_R(C^{\prime \prime})
\end{align*}
in $V(R)$, which is functorial and satisfies so-called ``the 9-terms relation;''
\item if an object $C$ of $\mathscr{C}^{\mathrm{Perf}}(R)$ is acyclic, 
the quasi-isomorphism $0\rightarrow C$ induces 
a canonical isomorphism $\mathbf{1}_R \xrightarrow{\sim} \det_R(C)$
where $\mathbf{1}_R=\det_R(0)$ denotes the unit object 
of the Picard category $V(R)$;

\item the equation $\det_R(C[r])=\det^{(-1)^r}_R(C)$ holds 
for an arbitrary object $C$ of $\mathscr{C}^{\mathrm{Perf}}(R)$ 
and an arbitrary integer $r$. 
Here $C[r]$ is the $r$-th translation of $C$ 
(that is, the cochain complex defined by $C[r]^i=C^{i+r}$ 
with an appropriate differential);

\item the functor $\det_R$ factorises 
$\mathscr{D}^{\mathrm{Perf}}(R)$, the image of $\mathscr{C}^{\mathrm{Perf}}(R)$ 
in the derived category $\mathscr{D}^{\mathrm{b}}(R)$ of bounded complexes 
of finitely generated left $R$-modules.
Moreover it extends to $\mathscr{D}^{\mathrm{Perf}}(R)_{\mathrm{isom}}$, 
the subcategory of $\mathscr{D}^{\mathrm{Perf}}(R)$ whose morphisms are 
restricted to isomorphisms;

\item if an object $C$ of $\mathscr{C}^{\mathrm{Perf}}(R)$ 
is {\em cohomologically perfect} 
(that is, if all cohomology groups $H^i(C)[0]$ 
belongs to $\mathscr{D}^{\mathrm{Perf}}(R)$), the equation
\begin{align*}
\det_R(C)=\prod_{i\in \Z} \det^{(-1)^i}_R(H^i(C))
\end{align*}
holds; 

\item the determinant functor $\det_R$ is stable under arbitrary base change;
that is, the diagram
\begin{align*}
\xymatrix{
\mathscr{D}^{\mathrm{Perf}}(R)_{\mathrm{isom}} \ar[rr]^(0.6){\det_R}
 \ar[d]_{R' \otimes^{\mathbb{L}}_R -} & & V(R) \ar[d]^{R' \otimes_R -} \\
\mathscr{D}^{\mathrm{Perf}}(R')_{\mathrm{isom}} \ar[rr]_(0.6){\det_{R'}} & & V(R')
}
\end{align*}
commutes for an arbitrary $R$-algebra $R'$.
\end{enumerate}
We denote the group of isomorphism classes of $V(R)$ by $\pi_0(V(R))$ 
and the group of isomorphisms of the unit object $\mathbf{1}_R$ 
by $\pi_1(V(R))$ respectively. Then there exist canonical isomorphisms
\begin{align*}
K_0(R) \xrightarrow{\sim} \pi_0(V(R)); & \quad [P] \mapsto [\det_R(P)], \\
K_1(R) \xrightarrow{\sim} \pi_1(V(R)); & \quad [f\colon P\xrightarrow{\sim} P] \mapsto [\mathbf{1}_R \xrightarrow{\det_R(f)} \mathbf{1}_R]
\end{align*}
where we abbreviate to $\det_R(f)$ the isomorphism
\begin{align*}
\mathbf{1}_R=\det_R(P)\cdot \det^{-1}_R(P) \xrightarrow{\det_R(f) \cdot \mathrm{id}_{\det^{-1}_R(P)}} \det_R(P)\cdot \det^{-1}_R(P)=\mathbf{1}_R
\end{align*}
induced by $f$. 
For an arbitrary $R$-algebra $R'$, we define
the category $V(R,R')$ as the fibre-product 
category $V(R)\times_{V(R')} V_0$ if we 
let $V_0$ denote the trivial Picard
category consisting of a unique object $\mathbf{1}$ equipped with 
the trivial automorphism group $\{\mathrm{id}_{\mathbf{1} }\}$; 
in particular an object of $V(R,R')$
is a pair $(L,\lambda)$ with $L$ an object of $V(R)$ and $\lambda$ 
an isomorphism $R' \otimes_R L \xrightarrow{\sim} \mathbf{1}_{R'}$ in
$V(R')$ (we call $\lambda$ {\em a trivialisation of $R'\otimes_R L$}). 
In fact the category $V(R,R')$ is a Picard category and 
there exists a canonical isomorphism between $K_0(R,R')$ and
$\pi_0(V(R,R'))$. For details see \cite[Section~5]{BrBur}.

Next let $F$ be a number field and $F'$ its finite Galois extension 
with the Galois group $G_{F'/F}$. 
For an arbitrary {\em strictly negative} integer $m$, 
consider the $m$\nobreakdash-fold Tate motive  
$\Q(m)_{F'/F}=h^0(\Spec\, F')(m)$ 
endowed with equivariant action of the Galois group $G_{F'/F}$ 
(regarded as defined over $F$). 
The Betti realisation of $\Q(m)_{F'/F}$ is explicitly described as 
\begin{align*}
\Q(m)_{F'/F,B}=\prod_{\tau \colon F' \hookrightarrow \C} \Q 
 (2\pi \sqrt{-1})^m
\end{align*}
where $\tau$ runs over all embeddings of $F'$ into $\C$. Note that 
the complex conjugate acts upon both $\tau$'s and $(2\pi \sqrt{-1})^m$
in a natural way. Let 
$\Q(m)_{F'/F,B}^+$ denote the maximal submodule of $\Q(m)_{F'/F,B}$ 
fixed by the action of the complex conjugate and set
\begin{align*}
\Xi(\Q(m)_{F'/F})=\det_{\Q[G_{F'/F}]}(K_{1-2m}(F')_{\Q}^*) \cdot 
\det_{\Q[G_{F'/F}]}^{-1}(\Q(m)_{F'/F,B}^+)
\end{align*}
where we denote by $X^*$ the $\Q$-dual of 
a left $\Q[G_{F'/F}]$-projective module $X$ endowed with contragredient 
$G_{F'/F}$-action (the virtual object $\Xi(\Q(m)_{F'/F})$ is nothing but 
{\em the fundamental line} for the Tate motive 
in the terminology of \cite{FPR}). 
As is well known, 
Be\u\i linson's regulator map (or Borel's regulator map) induces 
an isomorphism between $K_{1-2m}(F')\otimes_{\Z} \R$ and 
the $\R$-dual of $\Q(m)_{F'/F,B}^+\otimes_{\Q} \R$; hence
we obtain an isomorphism in $V(\R[G_{F'/F}])$ 
({\em the period-regulator isomorphism} in the terminology of \cite{FK})
\begin{align*}
\vartheta_{\infty}  \colon \Xi(\Q(m)_{F'/F}) \otimes_{\Q} \R 
\xrightarrow{\sim} \mathbf{1}_{\R[G_{F'/F}]}.
\end{align*}
We may regard the pair $(\Xi(\Q(m)_{F'/F}), \vartheta_{\infty})$ 
as the object of the relative Picard category $V(\Q[G_{F'/F}], \R[G_{F'/F}])$.

Now let $R\Gamma_{c,\acute{e}t}(\mathcal{O}_{F'}^{\Sigma^{\vee}},
\Q_p(m))$ (or $R\Gamma_{c,\acute{e}t}(\Spec 
\mathcal{O}_{F'}^{\Sigma^{\vee}},\Q_p(m))$) be the compactly
supported \'etale cohomology of $\Spec \mathcal{O}_{F'}^{\Sigma^{\vee}}$ 
with coefficient in $\Q_p(m)$ (regarded as an object of the derived
category) characterised by the distinguished triangle
\begin{align*}
R\Gamma_{c,\acute{e}t}(\mathcal{O}_{F'}^{\Sigma^{\vee}},
 \Q_p(m)) \rightarrow 
R\Gamma_{\acute{e}t}(\mathcal{O}_{F'}^{\Sigma^{\vee}}, \Q_p
 (m)) \rightarrow 
\bigoplus_{\mathfrak{P}\in \Sigma_{F'}} R\Gamma (G_{F'/F,\mathfrak{P}}, \Q_p(m)) 
\rightarrow
\end{align*}
where $R\Gamma(G_{F'/F,\mathfrak{P}}, \Q_p(m))$ denotes 
the $\Q_p(m)$-coefficient Galois cohomology of 
the decomposition group of $G_{F'/F}$ at $\mathfrak{P}$. 
Then we may construct 
an isomorphism in $V(\Q_p[G_{F'/F}])$ 
({\em the \p-adic period-regulator isomorphism} in the
terminology of \cite{FK})
\begin{align*}
\vartheta_p: \Xi(\Q(m)_{F'/F}) \otimes_{\Q} \Q_p \xrightarrow{\sim}
 \det_{\Q_p[G_{F'/F}]}
 (R\Gamma_{c,\acute{e}t}(\mathcal{O}_{F'}^{\Sigma^{\vee}}, \Q_p(m)))
\end{align*}
which is essentially derived from the \p-adic Chern class 
isomorphism
\begin{align*}
K_{1-2m}(\mathcal{O}_{F'})\otimes_{\Z} \Q_p \xrightarrow{\sim}
 H^1(\mathcal{O}_{F'}[1/p], \Q_p(m)).
\end{align*}
Refer to \cite[Sections~1.2.2 and 1.6]{BurFl1} or 
\cite[Section~2]{BurFl2} for details.

We finally consider the leading terms of the equivariant Artin $L$-functions. 
Let $\mathrm{Irr}(G_{F'/F})$
denote the set of all isomorphism classes of 
irreducible finite dimensional ($\overline{\Q}$-valued) 
representations of $G_{F'/F}$. 
Then by Wedderburn's decomposition theorem
we obtain an isomorphism
\begin{align} \label{eq:wed}
Z(\C[G_{F'/F}]) \xrightarrow{\sim} 
\prod_{\rho \in \text{Irr}(G_{F'/F})} \C \cdot \id_{V_{\rho}} \, ; \quad 
x\mapsto (\rho(x))_{\rho \in \mathrm{Irr}(G_{F'/F})}
\end{align}
where $V_{\rho}$ denotes the representation space of $\rho$. 
We identify the right-hand side of (\ref{eq:wed}) with 
the direct product $\prod_{\rho \in \mathrm{Irr}(G_{F'/F})}\C$ 
in an obvious way.
Then we may define {\em the leading term of the equivariant Artin 
$L$-function $L^*(\Q(m)_{F'/F})$ associated 
to the Tate motive $\Q(m)_{F'/F}$} 
as the element in $Z(\C[G_{F'/F}])^{\times}$ corresponding to 
$(L^*(m,\rho))_{\rho \in \text{Irr}(G_{F'/F})}$ via the isomorphism 
(\ref{eq:wed}) (we denote by
$L^*(m,\rho)$ the leading term of the complex Artin $L$-function 
$L(s; F'/F, \rho)$ at $s=m$). In fact $L^*(\Q(m)_{F'/F})$ is 
contained in $Z(\R[G_{F'/F}])^{\times}$ and there exists an element 
$\lambda$ in $Z(\Q[G_{F'/F}])^{\times}$ such that $\lambda
L^*(\Q(m)_{F'/F})$ is contained in the image of the reduced norm map
\begin{align*}
\text{nrd}_{\R[G_{F'/F}]} \colon K_1(\R[G_{F'/F}]) \rightarrow
 Z(\R[G_{F'/F}])^{\times}.
\end{align*}
See \cite[Lemmata~8,9]{BurFl3} for details. 

\begin{con}[Rationality conjecture, {\cite[Conjecture~5]{BurFl3}}] \label{conj:rational}
Let $\partial$ be the connecting homomorphism 
$K_1(\R[G_{F'/F}]) \rightarrow K_0(\Q[G_{F'/F}], \R[G_{F'/F}])$ 
appearing in the localisation exact
 sequence associated to the localisation $\Q[G_{F'/F}]\rightarrow
 \R[G_{F'/F}]$. 
 Then in $K_0(\Q[G_{F'/F}],\R[G_{F'/F}])$ the equation 
\begin{align*}
\partial(\mathrm{nrd}_{\R[G_{F'/F}]}^{-1}(\lambda L^*(\Q(m)_{F'/F})))+[\Xi(\Q(m)_{F'/F}), \vartheta_{\infty}]=0
\end{align*}
holds  $($via the canonical isomorphism as already explained, 
we identify the element $[\Xi(\Q(m)_{F'/F}), \vartheta_{\infty}]$ 
 in $\pi_0(V(\Q[G_{F'/F}], \R[G_{F'/F}]))$ with the corresponding one 
in the relative Grothendieck group 
 $K_0(\Q[G_{F'/F}],\R[G_{F'/F}])$$)$. 
\end{con}

It is known that Conjecture~\ref{conj:rational} 
for the $m$-fold Tate motive $\Q(m)_{F'/F}$ with negative $m$ is 
equivalent to the central conjecture of Benedict Hyman Gross \cite{Gross}. 
Conjecture~\ref{conj:rational} implies that there exists an isomorphism
\begin{align*}
\vartheta^{(\lambda)} \colon \Xi(\Q(m)_{F'/F}) \xrightarrow{\sim}
 \mathbf{1}_{\Q[G_{F'/F}]}
\end{align*}
such that the scalar extension $\vartheta^{(\lambda)}_{\R}=
\vartheta^{(\lambda)}\otimes_{\Q} \R$ coincides with 
the trivialisation of $\Xi (\Q(m)_{F'/F})\otimes_{\Q} \R$ described as 
$\mathrm{nrd}_{\R[G_{F'/F}]}^{-1}(\lambda L^*(\Q(m)_{F'/F}))
\circ \vartheta_{\infty}$. 

\begin{con}[$p$-part of the  Tamagawa number conjecture] \label{con:tamagawa}
Let $\partial_p$ denote the connecting homomorphism 
appearing in the localisation exact sequence 
associated to the localisation $\Z_p[G_{F'/F}]\rightarrow \Q_p[G_{F'/F}]$.  
Then the element 
$T\Omega(\Q(m)_{F'/F})_p$ in $K_0(\Z_p[G_{F'/F}], \Q_p[G_{F'/F}])$ 
defined by
\begin{align*}
T\Omega&(\Q(m)_{F'/F})_p \\
&=[\det_{\Z_p[G_{F'/F}]}(R\Gamma_{c,\acute{e}t}(\mathcal{O}_{F'}^{\Sigma^{\vee}},
 \Z_p(m))), \vartheta^{(\lambda)}_{\Q_p} \circ
 \vartheta_p^{-1}]-\partial_p(\mathrm{nrd}_{\Q_p[G_{F'/F}]}^{-1}(\lambda))
\end{align*}
vanishes where $\vartheta^{(\lambda)}_{\Q_p}$ denotes the scalar extension 
$\vartheta^{(\lambda)}\otimes_{\Q} \Q_p$ 
$($see also \cite[Conjecture~6]{BurFl3}$)$.
\end{con}

\begin{rem}
If $F'/\Q$ is a finite abelian Galois extension and $F$ is a subfield of $F'$, 
the equivariant Tamagawa number conjecture for the Tate motives $\Q(m)_{F'/F}$ 
has been proven for an arbitrary prime number $p$ 
and an arbitrary integer $m$ (not 
necessarily negative) by David Burns, Cornelius Greither and Mathias Flach.
Refer to \cite{BurGr} (for negative $m$ and odd $p$), 
\cite{Flach} (for negative $m$ and arbitrary $p$) and 
\cite{BurFl4} (for arbitrary $m$ and $p$).
Independently Annette Huber-Klawitter, Guido Kings and 
 Kensuke Itakura have also proven 
 the Bloch-Kato conjecture for Dirichlet motives
 \nobreakdash---a somewhat weaker conjecture than the equivariant 
 Tamagawa number conjecture\nobreakdash--- 
 by using rather different technique.
 See \cite{HK2} (for $p\neq 2$) and \cite{Itakura} (for $p=2$) for details.
 In spite of such great development upon commutative cases, 
 very little seems to be known for non-commutative cases.
\end{rem}

%%%%%%%%%%%%%%%%%%%%%%%%%%%%%%%%%%%%%%%%%%%%%%%%%%%%%%%%%
%
%
\section{Burns' technique} \label{sc:burns}
%
%
%%%%%%%%%%%%%%%%%%%%%%%%%%%%%%%%%%%%%%%%%%%%%%%%%%%%%%%%%

There exists a standard strategy to construct the
\p-adic zeta functions for non-commutative extensions
by ``patching'' Serre's \p-adic
zeta pseudomeasures for abelian extensions. 
It was first observed by David Burns 
and applied by Kazuya Kato to his pioneering work \cite{Kato2}. 
Here we shall introduce this outstanding technique 
in a little generalised way. 

Throughout this section we fix 
embeddings of $\overline{\Q}$ into $\C$ and $\overline{\Q}_p$.
Let $p$ be a positive odd prime number, $F$ a totally real number field and
$F_{\infty}/F$ a totally real \p-adic Lie extension satisfying 
conditions $(F_{\infty}\text{-}1)$, $(F_{\infty}\text{-}2)$ and $(F_{\infty}\text{-}3)$ in
the previous section. Let $G$, $H$ and $\Gamma$ be \p-adic Lie
groups defined as in Section~\ref{ssc:review-iwa}. 

\begin{defn}[Brauer families]
Let $\mathfrak{F}_B$ be a family consisting of
 a pair $(U,V)$ where $U$ is an open subgroup of $G$ and $V$ is that of
 $H$ such that $V$ is normal in $U$ and the quotient group $U/V$ is
 abelian. We call $\mathfrak{F}_B$ {\em a Brauer family for the group
 $G$} if it satisfies the following
 condition~$(\sharp)_B$$:$

\begin{quotation}
\noindent $(\sharp)_B$ \,
 an arbitrary Artin representation of $G$ is
 isomorphic to {\em a $\Z$-linear combination} (as a virtual representation) 
 of induced representations
 $\ind{G}{U}{\chi_{U/V}}$, where each 
 $(U,V)$ is an element 
 in $\mathfrak{F}_B$ and $\chi_{U/V}$ is a finite-order character
 of the abelian group $U/V$.
\end{quotation}
\end{defn}

Suppose that there exists a Brauer family $\mathfrak{F}_B$ for $G$. 
Let $\theta_{U,V}$ be the composition 
\begin{align*}
K_1(\iw{G}) \xrightarrow{\Nr_{\iw{G}/\iw{U}}} K_1(\iw{U})
 \xrightarrow{\text{canonical}} K_1(\iw{U/V}) \xrightarrow{\sim} \iw{U/V}^{\times}
\end{align*} 
for each $(U,V)$ in $\mathfrak{F}_B$ 
where $\Nr_{\iw{G}/\iw{U}}$ is the norm map in algebraic 
\mbox{$K$-theory}. Set
\begin{align*}
\theta=(\theta_{U,V})_{(U,V)\in \mathfrak{F}_B} \colon K_1(\iw{G}) \rightarrow \prod_{(U,V)\in \mathfrak{F}_B} \iw{U/V}^{\times}.
\end{align*}
Similarly we may construct the map\footnote{We use the same symbol
$S$ for the canonical Ore set for $F_V/F_U$ by abuse of notation.} 
\begin{align*}
\theta_S=(\theta_{S,U,V})_{(U,V)\in \mathfrak{F}_B} \colon K_1(\iw{G}_S)
 \rightarrow \prod_{(U,V)\in \mathfrak{F}_B} \iw{U/V}_S^{\times}
\end{align*}
for the localised Iwasawa algebra $\iw{G}_S$.
Then we obtain the following
commutative diagram with exact rows:
\[
\footnotesize
 \xymatrix{
    K_1(\iw{G}) \ar[r] \ar@{->}[d]_(0.5){\theta}   &   K_1(\iw{G}_S)  \ar[r]^(0.45){\partial} \ar[d]^(0.5){\theta_S}    &   K_0(\Lambda(G),
\Lambda(G)_S) \ar[r] \ar[d]^(0.5){\mathrm{norm}}  &  0 \\
   \prod_{\mathfrak{F}_B} \iw{U/V}^{\times} \ar@{^(->}[r]
 &  \prod_{\mathfrak{F}_B} \iw{U/V}_S^{\times}
 \ar[r]_(0.475){\partial\qquad \quad}  &  \prod_{\mathfrak{F}_B} K_0(\iw{U/V}, \iw{U/V}_S) \ar[r] & 0.}
\]
Let $f$ be an arbitrary characteristic element for $F_{\infty}/F$ (that
is, an element in $K_1(\iw{G}_S)$ satisfying the
relation~(\ref{eq:char})) and $(f_{U,V})_{(U,V)\in \mathfrak{F}_B}$ 
its image under the map $\theta_S$. Then each $f_{U,V}$ satisfies the
relation $\partial(f_{U,V})=-[C_{U,V}]$ by the functoriality of the connecting
homomorphism $\partial$. 
Now recall that for each pair $(U,V)$ in $\mathfrak{F}_B$, 
the \p-adic zeta pseudomeasure $\xi_{U,V}$ 
for the abelian extension $F_V/F_U$ exists 
as an element in $\iw{U/V}_S^{\times}$  
(see \cite{Serre2} for details) and 
the Iwasawa main conjecture $\partial(\xi_{U,V})=-[C_{U,V}]$ holds 
(due to the honourable results of Andrew Wiles \cite{Wiles}). 
Each \p-adic zeta pseudomeasure
$\xi_{U,V}$ is characterised by the interpolation property
\begin{align} \label{eq:interpab}
\xi_{U,V}(\chi \rho^r)=L_{\Sigma}(1-r;F_V/F_U, \chi)
\end{align}
for an arbitrary finite-order character of the abelian group $U/V$ 
and an arbitrary natural number $r$ divisible by $p-1$.
Let $w_{U,V}$ be the element defined as $\xi_{U,V}f_{U,V}^{-1}$ 
which is in fact an element in $\iw{U/V}^{\times}$ 
by the localisation exact sequence,
and consider the following assumption:
\begin{quotation}
\noindent {\bfseries Assumption} $(\flat)$\ 
 the element $(w_{U,V})_{(U,V)\in \mathfrak{F}_B}$ is contained 
 in the image of $\theta$.
\end{quotation}
Then under Assumption $(\flat)$ there exists an element $w$ in
$K_1(\iw{G})$ which satisfies
$\theta(w)=(w_{U,V})_{(U,V)\in \mathfrak{F}_B}$. Let $\xi$
be the element in $K_1(\iw{G}_S)$ defined as $fw$. Then
$\xi$ readily satisfies the following two conditions by easy diagram chasing:
\begin{enumerate}[$(\xi\text{-}1)$]
\item the equation $\partial(\xi)=-[C_{F_{\infty}/F}]$ holds;
\item the equation $\theta_S(\xi)=(\xi_{U,V})_{(U,V)\in \mathfrak{F}_B}$
      holds.
\end{enumerate}
By using condition $(\sharp)_B$, condition $(\xi\text{-}2)$ and 
the interpolation formulae (\ref{eq:interpab}), we may
verify that $\xi$ satisfies the interpolation formula (\ref{eq:interp})
as follows:
\begin{align*}
\xi(\rho\kappa^r) &=\xi \left( \kappa^r \sum_{(U,V)\in \mathfrak{F}_B} a_{U,V}
 \ind{G}{U}{\chi_{U/V}}\right)  \qquad (\text{by }(\sharp)_B)\\
&=\prod_{(U,V)\in \mathfrak{F}_B}\Nr_{\iw{G}_S/\iw{U}_S}(\xi)(\chi_{U/V}\kappa^r)^{a_{U,V}} \\
&=\prod_{(U,V)\in
 \mathfrak{F}_B}\xi_{U,V}(\chi_{U/V}\kappa^r)^{a_{U,V}} \qquad (\text{by
 } (\xi\text{-}2)) \\
&= \prod_{(U,V) \in \mathfrak{F}_B} L_{\Sigma}(1-r; F_V/F_U,
 \chi_{U/V})^{a_{U,V}} \qquad (\text{by } (\ref{eq:interpab})) \\
&=L_{\Sigma}(1-r; F_{\infty}/F, \sum_{(U,V)\in \mathfrak{F}_B} a_{U,V} \ind{G}{U}{\chi_{U/V}})=L_{\Sigma}(1-r; F_{\infty}/F,\rho) 
\end{align*}
where $\rho$ is an arbitrary Artin representation of $G$ and $r$ is an
arbitrary natural number divisible by $p-1$. Therefore $\xi$
is the desired \p-adic zeta function.
Furthermore $(\xi\text{-}1)$ implies that $\xi$ is also a characteristic 
element for $F_{\infty}/F$; in other words the Iwasawa
main conjecture~\ref{conj:iwasawa} holds for $F_{\infty}/F$. 

By virtue of Burns' technique, both construction of the \p-adic zeta
function and verification of the Iwasawa main conjecture are 
reduced to the following two tasks:
\begin{itemize}
\item characterisation of the images of $\theta$ and $\theta_S$;
\item verification of Assumption $(\flat)$.
\end{itemize}
In general, there are so many pairs in a Brauer family
$\mathfrak{F}_B$ that it is hard to compute and characterise the image of the
norm maps $\theta$ and $\theta_S$. Therefore we shall use not 
only Brauer families but also {\em Artinian families} in arguments
of the rest of this article.

\begin{defn}[Artinian families]
If a family $\mathfrak{F}_A$ consisting of
 an abelian open subgroup of~$G$ satisfies the following
 condition~$(\sharp)_A$, we
 call $\mathfrak{F}_A$ {\em an Artinian family for the group $G$}$:$

\begin{quotation}
\noindent $(\sharp)_A$ \, 
 an arbitrary Artin representation of~$G$ is
 isomorphic to 
 {\em a $\Z[1/p]$-linear combination} 
 (as a virtual representation) of induced representations
 $\ind{G}{U}{\chi_U}$, where each $U$
 is an element in $\mathfrak{F}_A$ 
 and $\chi_U$ is a finite-order character of the abelian group $U$.
\end{quotation}
\end{defn}
Artinian families tend to contain much fewer elements than Brauer
families, which often makes computation remarkably simpler.

%%%%%%%%%%%%%%%%%%%%%%%%%%%%%%%%%%%%%%%%%%%%%%%%%%%%%%%%%
%
%
\section{The main theorem and its application} \label{sc:maintheorem}
%
%
%%%%%%%%%%%%%%%%%%%%%%%%%%%%%%%%%%%%%%%%%%%%%%%%%%%%%%%%%

%%%%%%%%%%%%%%%%%%%%%%%%%%%%%%%%%%%%%%%%%%%%%%%%%%%%%%%%%
%
\subsection{The main theorem}
%
%%%%%%%%%%%%%%%%%%%%%%%%%%%%%%%%%%%%%%%%%%%%%%%%%%%%%%%%%

The precise statement of the main result in this article is as follows:

\begin{thm}[Main theorem]  \label{thm:maintheorem}
Let $p$ be a positive odd prime number and $F$ a totally real number field, 
and let $F_{\infty}$ be a totally real \p-adic Lie extension of $F$
 satisfying conditions $(F_{\infty}\text{-}1)$, $(F_{\infty}\text{-}2)$ and
 $(F_{\infty}\text{-}3)$ in Section~$\ref{ssc:review-iwa}$.  
Suppose that the Galois group of $F_{\infty}/F$ is isomorphic to the direct
      product of a finite \p-group $G^f$ with exponent $p$ and the
 commutative \p-adic Lie
      group $\Gamma$. Then the \p-adic zeta function $\xi_{F_{\infty}/F}$ for $F_{\infty}/F$
 exists uniquely up to multiplication by an element in 
 $SK_1(\Z_p[G^f])$. Moreover, the Iwasawa main conjecture
 $($Conjecture~$\ref{conj:iwasawa}$~$(2))$ is true for $F_{\infty}/F$ 
 $($for arbitrary fixed embeddings 
$\overline{\Q} \hookrightarrow \C$ and  
$\overline{\Q} \hookrightarrow \overline{\Q}_p)$.
\end{thm}

\begin{rem}
If the image of $SK_1(\Z_p[G^f])$ under the canonical localisation homomorphism
 $K_1(\iw{G})\rightarrow K_1(\iw{G}_S)$ vanishes, 
 we may establish the uniqueness result upon 
 the \p-adic zeta function for $F_{\infty}/F$ in
 Theorem~\ref{thm:maintheorem}. 
 However the author does not have any ideas either $SK_1(\Z_p[G^f])$ always 
 vanishes or not in $K_1(\iw{G}_S)$.
\end{rem}

Now let us consider an easy but interesting application of our main
theorem. Let $B^N(\F_p)$ be a multiplicative \p-group 
defined as the subgroup of the general linear group $\GL_{N+1}(\F_p)$ of
degree $N+1$ generated by 
all strongly upper-triangular matrices; that is,

\begin{align*}
B^N(\F_p)=\bordermatrix{ & 1 & 2 & \dotsc & \dotsc & N & N+1 \cr
\quad 1 & 1 & \F_p & \F_p & \dotsc & \F_p & \F_p \cr
\quad 2 & 0 & 1 & \F_p & \dotsc & & \F_p \cr
\quad \vdots & 0 & 0 & 1 & \ddots &  & \vdots \cr
\quad \vdots & \vdots & \vdots & & \ddots & \ddots &  \vdots \cr
\quad N & 0 &  &  &   &  1  &  \F_p \cr
N+1 & 0 & 0 & \dotsc & \dotsc & 0 & 1
}.
\end{align*}

\begin{cor} \label{cor:matrix}
Let $p$ be a positive odd prime number, $F$ a totally real number field 
and $F_B/F$ a totally real \p-adic Lie extension 
satisfying conditions $(F_{\infty}\text{-}1), (F_{\infty}\text{-}2)$ 
and $(F_{\infty}\text{-}3)$ in Section $\ref{ssc:review-iwa}$. 
Assume also that
\begin{enumerate}[$($\upshape i$)$]
\item there exists a certain non-negative integer $N$ such that the
      Galois group of $F_B/F$ is isomorphic to the direct
      product of $B^N(\F_p)$ and the commutative \p-adic Lie
      group $\Gamma;$
\item the prime number $p$ is strictly larger than $N$.
\end{enumerate}
Then the \p-adic zeta function $\xi_{F_B/F}$ for $F_B/F$
 exists uniquely up to multiplication by an element in 
 $SK_1(\Z_p[B^N(\F_p)])$. Moreover the Iwasawa main conjecture
 $($Conjecture~$\ref{conj:iwasawa}$~$(2))$ is true for $F_B/F$
 $($for arbitrary fixed embeddings
$\overline{\Q} \hookrightarrow \C$ and 
$\overline{\Q} \hookrightarrow \overline{\Q}_p)$.
\end{cor}

\begin{proof}
For each $p$ strictly larger than $N$, the exponent of $B^N(\F_p)$ 
 equals~$p$. Therefore the claim is directly deduced from 
 Theorem~\ref{thm:maintheorem}.
\end{proof}

\begin{rem}
The author is grateful to Peter Schneider and Otmar Venjakob 
for kindly informing him that they have verified  
triviality of the group $SK_1(\Z_p[B^N(\F_p)])$ for 
arbitrary $N$ in their ongoing project upon the ``non-commutative 
Coleman map'' (the case where $N$ equals $2$ has been already known 
by the results of Robert Oliver \cite[Proposition~12.7]{Oliver}). 
By combining their results with
 Corollary~\ref{cor:matrix}, we may verify uniqueness of 
the \p-adic zeta function for $F_B/F$.
\end{rem}

\begin{rem}
The case where $N$ is equal to $2$ is a special case of 
Kato's Heisenberg-type extensions
\cite{Kato2}. The case where $N$ is equal to $3$ is nothing but 
the main results of the preceding paper \cite{H}. 
Our original motivation upon this study was 
to generalise these results to the cases where $N$ is larger than $4$, 
and it was convenient to consider the problem
under more general conditions as in Theorem~\ref{thm:maintheorem}.
\end{rem}

In the rest of this article we mainly deal with cases under the
conditions of our main theorem (Theorem~\ref{thm:maintheorem}).

%%%%%%%%%%%%%%%%%%%%%%%%%%%%%%%%%%%%%%%%%%%%%%%%%%%%%%%%%%
%
\subsection{Application to the equivariant Tamagawa number conjecture 
for critical Tate motives} \label{ssc:application}
%
%%%%%%%%%%%%%%%%%%%%%%%%%%%%%%%%%%%%%%%%%%%%%%%%%%%%%%%%%%

The non-commutative Iwasawa main conjecture should be deeply related to 
the (non-commutative) equivariant Tamagawa number conjecture, as was 
pointed out in, for example, \cite{HK1} and \cite{FK}. 
In this subsection, we shall show that 
the \p-part of the (non-commutative) equivariant Tamagawa number 
conjecture for critical Tate motives follows from 
the Iwasawa main conjecture (Conjecture~\ref{conj:iwasawa}) 
by applying a standard descent argument. 
This can be regarded as evidence implying mystic relationship 
between the non-commutative Iwasawa main conjecture and 
the non-commutative Tamagawa number conjecture.

\begin{prop} \label{prop:ETNC}
Let $p$ be a positive odd prime number and $F$ a totally real number field.
Let $F_{\infty}$ be a totally real \p-adic Lie extension of $F$
 satisfying conditions $(F_{\infty}\text{-}1)$, $(F_{\infty}\text{-}2)$ and
 $(F_{\infty}\text{-}3)$ in Section~$\ref{ssc:review-iwa}$.  
Suppose also that Conjectures~$\ref{conj:iwasawa}$ $(1),\, (2)$ are true
 for $F^{\infty}/F$. Then the \p-part of 
 the equivariant Tamagawa number conjecture 
 $($Conjecture~$\ref{con:tamagawa})$ for
 $\Q(1-r)_{F'/F}$ is true for an arbitrary finite Galois subextension $F'$ 
 of $F_{\infty}/F$ and an arbitrary natural number $r$ divisible by $p-1$.
\end{prop}

Note that the Tate motive $\Q(1-r)_{F'/F}$ is a {\em critical} motive 
in the sense of Pierre Deligne \cite[Definition~1.3]{Deligne1}
since both $F$ and $F'$ are totally real and $r$ is even.
Combining this proposition with Theorem~\ref{thm:maintheorem}, 
we obtain:

\begin{cor} \label{cor:ETNC}
Let $p, F$ and $F_{\infty}/F$ be as in Proposition~$\ref{prop:ETNC}$. 
Suppose that the Galois group of $F_{\infty}/F$ is isomorphic to the direct
 product of a finite \mbox{\p-group} $G^f$ with exponent $p$ and the
 commutative \p-adic Lie group $\Gamma$. 
 Then the \p-part of the equivariant Tamagawa number conjecture for
 $\Q(1-r)_{F'/F}$ is true for an arbitrary finite Galois subextension $F'$ 
 of $F_{\infty}/F$ and an arbitrary natural number $r$ divisible by $p-1$.
\end{cor}

This corollary gives a simple but non-trivial example 
strongly suggesting validity of 
the equivariant Tamagawa number conjecture for motives with 
non-commutative coefficient. 
Proposition~\ref{prop:ETNC} is just the direct consequence 
of the Iwasawa main conjecture and descent theory established by
David Burns and Otmar Venjakob \cite{BurVen}, and all materials used in 
the proof should be essentially contained 
in \cite{BurVen}. There, however, does not seem 
to exist explicit suggestion upon 
critical Tate motives there, and thus we shall give the proof of 
Proposition~\ref{prop:ETNC} in the rest of this subsection.

\medskip
First we may easily see that both $\Q(1-r)_{F'/F,B}^+$ and 
$K_{2r-1}(F')_{\Q}^*$ are trivial because $F$ and $F'$ are
totally real fields and $r$ is an even natural number 
(triviality of $K_{2r-1}(F')_{\Q}$ 
is due to Armand Borel \cite{Borel1, Borel2}), and thus
$\Xi(\Q(1-r)_{F'/F})$ is also trivial  
and the period-regulator map 
$\vartheta_{\infty}$ degenerates to the identity map 
on the unit object $\mathbf{1}_{\R[G_{F'/F}]}$. 

\begin{lem} \label{lem:rational}
The leading term $L^*(\Q(1-r)_{F'/F})$ of the equivariant Artin
 $L$-function for $\Q(1-r)_{F'/F}$ is contained in
 $Z(\Q[G_{F'/F}])^{\times}$.
\end{lem}

\begin{proof}
We may verify the claim by an argument similar to that in 
 \cite[Proposition~6.7]{Deligne1}; first note that 
$L(s; F'/F, \rho)$ does not vanish at $s=1-r$ for each $\rho$ in
 $\mathrm{Irr}(G_{F'/F})$ because $1-r$ is a
 critical value for $L(s; F'/F, \rho)$. Therefore $L^*(\Q(1-r)_{F'/F})$
 corresponds to $(L(1-r; F'/F, \rho))_{\rho \in \mathrm{Irr}(G_{F'/F})}$ 
 via the isomorphism (\ref{eq:wed}). Since the action of an automorphism 
 $\tau$ of $\C$ upon $\prod_{\rho \in \mathrm{Irr}(G_{F'/F})} \C$ is given
 by 
\begin{align*}
(x_{\rho})_{\rho\in \mathrm{Irr}(G_{F'/F})} \qquad \mapsto \qquad
 (\tau(x_{\tau^{-1} \rho}))_{\rho \in \mathrm{Irr}(G_{F'/F})},
\end{align*}
 it suffices to prove that 
 $L(1-r; F'/F, \tau \rho)=\tau L(1-r; F'/F, \rho)$ holds 
 for an arbitrary automorphism $\tau$ of $\C$; 
 this follows from the classical results of Helmut Klingen \cite{Klingen} and 
 Carl Ludwig Siegel \cite{Siegel} combining with Serre's assertion
 (see also \cite[Section~1.1]{CL}). 
\end{proof}

Lemma~\ref{lem:rational} implies that the rationality
conjecture~\ref{conj:rational} is true for the Tate motive $\Q(1-r)_{F'/F}$.
Choose an element $\lambda$ in $Z(\Q[G_{F'/F}])^{\times}$ 
such that $\lambda L^*(\Q(1-r)_{F'/F})$ is contained in the image of the
reduced norm map $\mathrm{nrd}_{\Q[G_{F'/F}]}$ (this is injective; see 
\cite[(45.3)]{MR} and \cite[Proposition~2.2]{BurFl3}),
then the map $\vartheta^{(\lambda)}$ introduced in
Section~\ref{ssc:review-tama} coincides with
$\mathrm{nrd}_{\Q[G_{F'/F}]}^{-1} (\lambda L^*(\Q(1-r)_{F'/F}))$. 
Set $\Delta(1-r)_{F'}=R\Gamma_{c,\acute{e}t}(\mathcal{O}_{F'}^{\Sigma^{\vee}},
 \Z_p(1-r))$. We can easily check that 
the natural homomorphism 
\begin{align*}
K_0(\Z_p[G_{F'/F}], \Q_p[G_{F'/F}]) \rightarrow K_0(\Z_p[G_{F'/F}], \overline{\Q}_p[G_{F'/F}])
\end{align*}
induced by the canonical embedding $\Q_p \hookrightarrow
\overline{\Q}_p$ is injective and can calculate 
the image $T\Omega(\Q(1-r)_{F'/F})_{p, \overline{\Q}_p}$ 
of $T\Omega(\Q(1-r)_{F'/F})_p$ under this map as 
\begin{align*} 
\bar{\partial}_p(\mathrm{nrd}_{\overline{\Q}_p[G_{F'/F}]}^{-1}(L^*(\Q(1-r)_{F'/F}))+[\Delta(1-r)_{F'}, \vartheta_{p,\overline{\Q}_p}^{-1}]
\end{align*}
where $\bar{\partial}_p$ is the connecting homomorphism appearing in the
localisation exact sequence associated 
to $\Z_p[G_{F'/F}]\rightarrow \overline{\Q}_p[G_{F'/F}]$. 
Here we use the following relations:
\begin{align*}
[\Delta(1-r)_{F'}, \vartheta^{(\lambda)}_{\Q_p}\circ \vartheta^{-1}_p]
 = [\Delta(1-r)_{F'}, \vartheta^{-1}_p] +
 \partial_p(\vartheta^{(\lambda)}_{\Q_p}), 
\end{align*}
\begin{align*}
\bar{\partial}_p(\vartheta^{(\lambda)}_{\overline{\Q}_p}) 
&=\bar{\partial}_p(\mathrm{nrd}_{\overline{\Q}_p[G_{F'/F}]}^{-1}(\lambda
 L^*(\Q(1-r)_{F'/F}))) \\
&=\bar{\partial}_p(\mathrm{nrd}_{\overline{\Q}_p[G_{F'/F}]}^{-1}(L^*(\Q(1-r)_{F'/F})))
 +\bar{\partial}_p(\mathrm{nrd}^{-1}_{\overline{\Q}_p[G_{F'/F}]}(\lambda)).
\end{align*}
Therefore in order to prove Proposition~\ref{prop:ETNC}
it suffices to show that 
the element $T\Omega(\Q(1-r)_{F'/F})_{p,\overline{\Q}_p}$ vanishes.

\begin{proof}[Proof of Proposition~$\ref{prop:ETNC}$]
Let $\xi_{F_{\infty}/F}$ be the \p-adic zeta function for
 $F_{\infty}/F$ and assume that the Iwasawa main conjecture
 $\partial(\xi_{F_{\infty}/F})=-[C_{F_{\infty}/F}]$ is valid
 (for an arbitrary fixed embedding $j_p \colon \overline{\Q}\hookrightarrow
 \overline{\Q}_p$).  
 Since
 $R\Gamma_{\acute{e}t}(\mathcal{O}_{F_{\infty}}^{\Sigma^{\vee}}, \Q_p/\Z_p)$ 
 is identified with the injective limit of complexes
 $R\Gamma_{\acute{e}t} (\mathcal{O}_L^{\Sigma^{\vee}},
 \Q_p/\Z_p)$ for all finite Galois subextensions $L/F$ of
 $F^{\infty}/F$, 
 we may easily see that $C_{F_{\infty}/F}$ is isomorphic to the complex 
 $\varprojlim_L R\Gamma_{c, \acute{e}t}(\Spec
 \mathcal{O}_L^{\Sigma^{\vee}}, \Z_p(1))[3]$ by virtue of the 
 Poitou-Tate/Artin-Verdier duality theorem. 
 Furthermore for each $L$ the complex 
$R\Gamma_{c, \acute{e}t}(\mathcal{O}_L^{\Sigma^{\vee}},
 \Z_p(1))$ is isomorphic to $R\Gamma_{c,
 \acute{e}t}(\mathcal{O}_F^{\Sigma^{\vee}}, \Z_p[G_{L/F}]^{\sharp}(1))$ 
 by Shapiro's lemma (here $\Z_p[G_{L/F}]^{\sharp}$ denotes 
 $\Z_p[G_{L/F}]$ regarded as a left $G_{F}$\nobreakdash-module whose $G_F$-action is
 given by the right multiplication of the inverse element). Refer to, for
 example, \cite[Appendix~B]{HorKin} for details.
 Hence the following equation holds 
in $K_0(\mathscr{C}^{\mathrm{Perf}}_S(G), \mathrm{qis})$:
\begin{align*}
\partial(\xi_{F_{\infty}/F})& =[\varprojlim_L R\Gamma_{c, \acute{e}t}(\Spec
 \mathcal{O}_F^{\Sigma^{\vee}}, \Z_p[G_{L/F}]^{\sharp}(1))] \\
 &= [R\Gamma_{c, \acute{e}t}(\Spec \mathcal{O}_F^{\Sigma^{\vee}},
 \iw{G}^{\sharp}(1))].
\end{align*}

Now for each natural number $r$ divisible by $p-1$ 
consider the $\Z_p$-linear 
map $\varrho_{\kappa}^r \colon \iw{G}\rightarrow \iw{G}$ 
induced by $\sigma \mapsto \kappa^r(\sigma) \sigma$ 
for $\sigma$ in $G$, which is in fact a ring automorphism 
because $\kappa^r(\sigma)$ is an element in the centre of $\iw{G}$.  
This map also induces a ring automorphism $\varrho_{S,\kappa}^r$ 
on the canonical Ore localisation $\iw{G}_S$ of $\iw{G}$ and 
its composition with the homomorphism 
$\iw{G}_S\rightarrow \mathrm{Frac}(\iw{\Gamma})$ induced 
by the projection $G\rightarrow \Gamma$ coincides with the morphism
$\Phi_{\kappa^r}$ introduced in \cite[Lemma~3.3]{CFKSV}. 
Hence the definition of the evaluation map asserts that the
 interpolation formula
\begin{align*}
\varrho_{S,\kappa}^r(\xi_{F_{\infty}/F})(\rho)=L_{\Sigma}(1-r;F_{\infty}/F,
 \rho)
\end{align*}
holds for an arbitrary Artin representation $\rho$ of $G$.
On the other hand the Tate twist $\iw{G}^{\sharp}(1)\rightarrow
 \iw{G}^{\sharp}(1-r)$ defines a $\varrho_{S,\kappa}^r$-semilinear 
isomorphism (due to the $\iw{G}$-module structure of 
$\iw{G}^{\sharp}$), and in $\mathscr{D}^{\mathrm{Perf}}(\iw{G})$ 
we therefore obtain the isomorphism 
\begin{align*}
\iw{G}\otimes_{\iw{G}, \varrho_{\kappa}^r} R\Gamma_{c, \acute{e}t}
 (\mathcal{O}_{F}^{\Sigma^{\vee}}, \iw{G}^{\sharp}(1))
 \xrightarrow{\sim}
 R\Gamma_{c,\acute{e}t}(\mathcal{O}_F^{\Sigma^{\vee}},
 \iw{G}^{\sharp}(1-r))
\end{align*}
(we remark that this argument is a non-commutative variant 
of that in \cite[Lemma~5.13~a)]{Flach}). Then we obtain
\begin{align} \label{eq:variwa}
 \partial(\varrho_{S,\kappa}^r(\xi_{F_{\infty}/F})) =
 [R\Gamma_{c,\acute{e}t}(\mathcal{O}_F^{\Sigma^{\vee}},
 \iw{G}^{\sharp}(1-r))]
\end{align}
by the functoriality of the connecting homomorphism. Moreover 
\begin{align*}
& \quad \Z_p[G_{F'/F}]\otimes_{\iw{G}}^{\mathbb{L}}
 R\Gamma_{c,\acute{e}t}(\mathcal{O}_F^{\Sigma^{\vee}},
 \iw{G}^{\sharp}(1-r)) \\
&=R\Gamma_{c,\acute{e}t}(\mathcal{O}_F^{\Sigma^{\vee}},
 \Z_p[G_{F'/F}]^{\sharp}(1-r))
= R\Gamma_{c,\acute{e}t}(\mathcal{O}_{F'}^{\Sigma^{\vee}}, \Z_p(1-r)) =\Delta(1-r)_{F'}
\end{align*}
holds for an arbitrary finite Galois subextension $F'/F$ of $F_{\infty}/F$. 
Since the localisation $\Q_p \otimes_{\Z_p} \Delta(1-r)_{F'}$ is 
acyclic (essentially due to the criticalness of $\Q(1-r)_{F'/F}$; 
refer to \cite[(9),(10)]{BurFl2}),
the equation (\ref{eq:variwa}) descends to 
\begin{equation} \label{eq:tamagawa}
\begin{aligned}
K_1(\overline{\Q}_p[G_{F'/F}]) \quad \qquad & \xrightarrow{\quad \partial\quad }  K_0(\Z_p[G_{F'/F}],
 \overline{\Q}_p[G_{F'/F}]) \\
(L_{\Sigma}(1-r; F'/F, j_p \rho))_{\rho \in \mathrm{Irr}(G_{F'/F})}
 & \quad \mapsto  \quad \qquad [\Delta(1-r)_{F'}] 
\end{aligned}
 \end{equation}
by the results of Burns and Venjakob \cite[Theorem~2.2]{BurVen}. 
Here we remark that the element $[\Delta(1-r)_{F'}]$ 
in the relative Grothendieck group can be naturally identified with 
the element $-[\det_{\Z_p[G_{F'/F}]}(\Delta(1-r)_{F'}), \mathrm{acyc}]$ 
in $\pi_0(\Z_p[G_{F'/F}], \overline{\Q}_p[G_{F'/F}])$ where 
``$\mathrm{acyc}$'' denotes the trivialisation induced by 
 acyclicity of $\overline{\Q}_p\otimes_{\Z_p}\Delta(1-r)_{F'}$ 
 (see Remark~\ref{rem:sign} for the problem upon sign
 convention). The difference between two
 trivialisations $\vartheta_{p,\overline{\Q}_p}^{-1}$ 
and $\mathrm{acyc}$ is calculated 
 in \cite{BurFl2} as 
\begin{align*}
\mathrm{acyc}=\vartheta_{p,\overline{\Q}_p}^{-1} \cdot \prod_{v\in \Sigma}\phi_v^{-1} 
\end{align*}
where each $\overline{\Q}_p$-isomorphism $\phi_v \colon V\rightarrow V$
 is defined as in \cite[Section~1.2]{BurFl1} or
 \cite[Section~2.4.2]{FK} which we regard as an element in
 $K_1(\overline{\Q}_p[G_{F'/F}])$. Then by definition 
 the image of $\phi_v^{-1}$
 under the Wedderburn decomposition
\begin{align*} 
\xymatrix{
K_1(\overline{\Q}_p[G_{F'/F}])
 \ar[rr]^{\sim}_{\mathrm{nrd}_{\overline{\Q}_p[G_{F'/F}]}} & &
 Z(\overline{\Q}_p[G_{F'/F}])^{\times} \ar[r]^{\sim} & \prod_{\rho \in
 \mathrm{Irr}(G_{F'/F})} \overline{\Q}_p^{\times}
 }
\end{align*}
 coincides with 
 $(L_v(1-r; F'/F, \rho))_{\rho \in \mathrm{Irr}(G_{F'/F})}$, 
 the local factor of the equivariant Artin $L$-function at $v$. 
 Combining this fact with the relation (\ref{eq:tamagawa}), 
 we can easily verify that $T\Omega(\Q(1-r)_{F'/F})_{p,\overline{\Q}_p}$ 
 vanishes.
\end{proof}

\begin{rem}
If we take $F'=F$, Proposition~\ref{prop:ETNC} is equivalent to 
the \p-part of the cohomological Lichtenbaum conjecture
\begin{align*}
|\zeta_F(1-r)|_p^{-1}=
\frac{|\sharp H^2_{c,\acute{e}t}(\mathcal{O}_{F}^{\Sigma^{\vee}},
 \Z_p(1-r))|_p^{-1}}{|\sharp H^1_{c,\acute{e}t}(\mathcal{O}_F^{\Sigma^{\vee}},
 \Z_p(1-r))|_p^{-1}} 
\end{align*}
via certain specialisation (here $\zeta_F(s)$ is the complex Dedekind
 zeta function for $F$ and $|\cdot |_p$ is the \p-adic
 valuation normalised by $|p|_p=1/p$). This is directly deduced from the
 (classical) Iwasawa main conjecture for totally real number fields 
 verified by Andrew Wiles \cite{Wiles}. 
 Proposition~\ref{prop:ETNC} gives its certain generalisation 
 for cases with non-commutative coefficient.
\end{rem}

\begin{rem}[sign convention] \label{rem:sign}
Let $R$ be an associative ring and $S$  a left Ore subset of $R$.
We let $S^{-1}R$ denote the left Ore localisation of $R$ 
with respect to $S$. 
In \cite{Swan}, the relative Grothendieck group 
$K_0(R, S^{-1}R)$ is defined as 
a certain quotient of the free abelian group generated by all
triples $[P,\lambda, Q]$ where each $P$ and $Q$ are finitely generated 
projective left $R$-modules and $\lambda$ is an
 $S^{-1}R$-isomorphism $\lambda \colon S^{-1}R\otimes_R P
 \xrightarrow{\sim} S^{-1}R\otimes_R Q$. 
Then we may identify a homomorphism $P\rightarrow Q$ 
induced by $\lambda$ with a cochain complex concentrated in terms of 
degree $0$ and $1$, and we use this identification as the 
normalisation of the isomorphism between $K_0(R,S^{-1}R)$ and 
$K_0(\mathscr{C}^{\mathrm{Perf}}_S(R), \mathrm{qis})$ 
(this normalisation is the same one as used in \cite{FK}). In \cite{BurVen},
 however, they identify $K_0(R,S^{-1}R)$ with $\pi_0(V(R,S^{-1}R))$ 
in the following manner: when both 
$\mathrm{Ker}(\lambda)$ and $\mathrm{Coker}(\lambda)$ are projective, 
the element $[P,\lambda, Q]$ in $K_0(R,S^{-1}R)$ 
 is identified with the element in $\pi_0(V(R,S^{-1}R))$ defined as 
 $[\det^{-1}_R(P)\cdot \det_R(Q), \det_R^{-1}(\lambda)
 \cdot \mathrm{id}_{\det_R(Q)}]$; in other words they implicitly regard 
 $P\rightarrow Q$ as a complex concentrated in terms of
 degree $-1$ and $0$. Hence there appears difference in sign convention
\begin{eqnarray*}
K_0(\mathscr{C}^{\mathrm{Perf}}_S(R), \mathrm{qis}) & \quad \xrightarrow{\sim}
 \quad & \quad \pi_0(V(R,S^{-1}R))  \\ \nonumber
[C] \qquad \quad &\quad \leftrightarrow \quad & -[\det_R(C), \mathrm{acyc}]
\end{eqnarray*}
(on the other hand they used, in \cite{BrBur}, the different normalisation
\begin{align*}
[P,\lambda, Q] \quad \leftrightarrow \quad [\det_R(P)\cdot
 \det^{-1}_R(Q), \det_R(\lambda) \cdot \mathrm{id}_{\det^{-1}_R(Q)}],
\end{align*}
and the element $[C]$ in $K_0(\mathscr{C}^{\mathrm{Perf}}_S(R),
 \mathrm{qis})$ therefore corresponds to the element 
$[\det(C), \mathrm{acyc}]$ in $\pi_0(V(R,S^{-1}R))$).
\end{rem}

%%%%%%%%%%%%%%%%%%%%%%%%%%%%%%%%%%%%%%%%%%%%%%%%%%%%%%%%%
%
%
\section{Construction of the theta isomorphism I ---additive theory---}
\label{sc:additive}
%
%
%%%%%%%%%%%%%%%%%%%%%%%%%%%%%%%%%%%%%%%%%%%%%%%%%%%%%%%%%

In this section we first define the Artinian families $\mathfrak{F}_A$,
$\mathfrak{F}_A^{c}$ and the Brauer family $\mathfrak{F}_B$ 
(see Section~\ref{ssc:family} for definition), which will play important
roles in the following arguments. We then 
construct the additive version of the theta isomorphism
(see Section~\ref{ssc:addtheta}). Later we shall translate it into the
multiplicative morphism in Section~\ref{sc:trans}. We remark that 
Mahesh Kakde has recently established more general construction of
the additive theta isomorphism \cite{Kakde2} (his
construction can be applied 
to case sin which $G^f$ is an arbitrary finite \p-group 
not necessarily with exponent~$p$).

%%%%%%%%%%%%%%%%%%%%%%%%%%%%%%%%%%%%%%%%%%%%%%%%%%%%%%%%%
%
\subsection{Artinian families $\mathfrak{F}_A$, $\mathfrak{F}_A^{c}$ and
  Brauer family $\mathfrak{F}_B$} \label{ssc:family}
%
%%%%%%%%%%%%%%%%%%%%%%%%%%%%%%%%%%%%%%%%%%%%%%%%%%%%%%%%%

Let $p$, $F$ and $F_{\infty}/F$ be as in
Theorem~$\ref{thm:maintheorem}$. Let $G$ be the Galois group of
$F_{\infty}/F$ and $p^N$ the order of the finite part $G^f$ of $G$
(and $N$ is hence a non-negative integer).
The finite \p-group $G^f$ acts upon the set of all its cyclic subgroups 
by conjugation. Choose a set of representatives of the
orbital decomposition under this action, and choose also a generator 
for each representative cyclic group. 
Let $\mathfrak{H}$ denote the set of these fixed
generators, and for each $h$ in $\mathfrak{H}$ let $U_h^f$ be the cyclic
subgroup of $G^f$ generated by $h$. Since the exponent of $G^f$ is $p$, 
the degree of each $U_h^f$
exactly equals $p$ except for $U_e^f=\{e\}$ 
(here we denote the unit of~$G^f$ by $e$). 
Let $U_h$ be the
open subgroup of $G$ isomorphic to the direct product of $U_h^f$
and $\Gamma$ for each $h$ in $\mathfrak{H}$, and  
consider the family of open subgroups of $G$ consisting of
all such $U_h$ which we denote by
$\mathfrak{F}_A$ (we always identify $U_e$ with $\Gamma$).

\begin{prop} \label{prop:artfamily}
The family $\mathfrak{F}_A$ satisfies condition
 $(\sharp)_A$. In other words, the family $\mathfrak{F}_A$ is 
 an Artinian family for the group $G$.
\end{prop}

\begin{proof}
The claim is directly deduced from the classical induction theorem 
of Emil Artin (see, for example, \cite[Corollaire de~Th\'eor\`eme~15]{Serre1}).
\end{proof}

For the usage of induction in Section~\ref{sc:construction}, 
we now introduce another Artinian family
$\mathfrak{F}_A^c$. When $N$ equals zero, we set 
$\mathfrak{F}_A^c=\mathfrak{F}_A=\{(\Gamma, \{ e\})\}$. 
When $N$ is larger than $1$, 
choose a central element $c\neq e$ in $\mathfrak{H}$ 
and fix it (there exists such an
element $c$ because $G^f$ is a \p-group). 
For each $h$ in $\mathfrak{H}$, 
let $U_{h,c}^f$ be the abelian subgroup of $G^f$ 
generated by $h$ and $c$, and let
$U_{h,c}$ be the open subgroup of $G$ isomorphic to the 
direct product of~$U_{h,c}^f$ and $\Gamma$. Let $\mathfrak{F}_A^c$ denote
the family of open subgroups of $G$ consisting of all elements in
$\mathfrak{F}_A$ and all $U_{h,c}$ (we 
identify both $U_{e,c}$ and $U_{c,c}$ with $U_c$).  
Then the family $\mathfrak{F}_A^c$ is also an Artinian family for $G$
because $\mathfrak{F}_A^c$ contains the Artinian family $\mathfrak{F}_A$.

We finally define $\mathfrak{F}_B$ as the family consisting of all pairs 
$(U,V)$ such that $U$ is an open
subgroup of $G$ containing $\Gamma$ and  
$V$ is the commutator subgroup of $U$. Then the family 
$\mathfrak{F}_B$
satisfies condition $(\sharp)_B$ by Richard Brauer's
induction theorem \cite[Th\'eor\`em~22]{Serre1} 
(note that for an arbitrary finite
\p-group, the family of all its Brauer
elementary subgroups coincides with that of all its subgroups by definition);
hence $\mathfrak{F}_B$ is a Brauer family.

%%%%%%%%%%%%%%%%%%%%%%%%%%%%%%%%%%%%%%%%%%%%%%%%%%%%%%%%%
%
\subsection{Calculation of the images of trace homomorphisms} \label{ssc:calculation}
%
%%%%%%%%%%%%%%%%%%%%%%%%%%%%%%%%%%%%%%%%%%%%%%%%%%%%%%%%%

First recall the definition of {\em trace homomorphisms}; for an
arbitrary finite group $\Delta$, let $\Z_p[\conj{\Delta}]$ be the free
$\Z_p$-module of finite rank with basis $\conj{\Delta}$, and for an
arbitrary pro-finite group $P$, let $\Z_p[[\conj{P}]]$ be 
the projective limit of free $\Z_p$-modules
$\Z_p [\conj{P_{\lambda}}]$ over finite quotient groups $P_{\lambda}$ of
$P$.

\begin{defn}[trace homomorphisms] 
Let $P$ be an arbitrary pro-finite group and $U$ its arbitrary open
 subgroup. Let $\{ a_1, a_2, \dotsc , a_s\}$ be one of the representatives of the left coset decomposition $P/U$. For an
 arbitrary conjugacy class $[g]$ of $P$ and for each integer 
 $1\leq j\leq s$, define $\tau_j([g])$ as 
\begin{align*}
\tau_j([g])=
\begin{cases} [a_j^{-1}ga_j] & \text{if }a_j^{-1}ga_j \text{ is
 contained in } U, \\
0 & \text{otherwise}.
\end{cases}
\end{align*}
Then the element $\Tr_{\Z_p[[\conj{P}]]/\Z_p[[\conj{U}]]}([g]) =\sum_{j=1}^s
 \tau_j([g])$ is determined independently of the choice of representatives
 $\{a_j\}_{j=1}^s$. We call the induced $\Z_p$-module homomorphism
\begin{align*}
\Tr_{\Z_p[[\conj{P}]]/\Z_p[[\conj{U}]]} \colon \Z_p[[\conj{P}]]
 \rightarrow \Z_p[[\conj{U}]]
\end{align*}
{\em the trace homomorphism from $\Z_p[[\conj{P}]]$ to
 $\Z_p[[\conj{U}]]$}.
\end{defn}

Let $c$ be the fixed central element of $G^f$ as in the previous
subsection and let $\theta^+_U$ denote the trace homomorphism
$\Tr_{\Z_p[[\conj{G}]]/\Z_p[[U]]}$ for each $U$ in~$\mathfrak{F}_A^c$.
We now calculate each image $I_U$ of $\theta^+_U$. 
Let $NU^f$ denote the normaliser of $U^f$ in $G^f$ 
for each $U$ in $\mathfrak{F}_A$. 
We denote by $p^{n_h}$ the cardinality of $NU_h^f$ 
for each $h$ in $\mathfrak{H}$.

\medskip
\noindent
\textbf{Calculation of $I_{\Gamma}(=I_{U_e})$.} When $N$ is equal to zero,
the $\Z_p$-module $I_{\Gamma}$ obviously coincides with 
$\iw{G}=\iw{\Gamma}$. Now suppose
that $N$ is larger than~$1$. In this case, $\theta_{\Gamma}^+([g])$ 
does not vanish if and only if $g$ is contained in $\Gamma$. We may 
regard the finite part $G^f$ as a set of representatives of the left
coset decomposition $G/\Gamma$, and for each $\gamma$ in $\Gamma$ and
$a$ in $G^f$, the element $a^{-1}\gamma a$ equals $\gamma$ 
(note that $\gamma$ and $a$ commute). Therefore we have 
\[
 I_{\Gamma} =p^N\Z_p[[\Gamma]]
\]
(this equality is also valid for the case in which $N$ equals zero).

\medskip
\noindent
\textbf{Calculation of $I_{U_h}$ for $h$ in
$\mathfrak{H}$ except for $e$} ($N\geq 1$). When $N$ is equal to $1$,
the $\Z_p$-module $I_{U_h}$ obviously coincides with
$\iw{G}=\iw{U_h}$. Hence suppose that 
$N$ is larger than $2$. In this case $\theta_{U_h}^+([g])$ 
does not vanish if and
only if $g$ is contained in one of the conjugates of $U_h$, and 
we may therefore assume that $g$ is contained in $U_h$ 
without loss of generality. 
The normaliser $NU_h^f$ acts upon $U^f_h$ by conjugation, which
induces a group antihomomorphism 
$\inn \colon (NU_h^f)^{\mathrm{op}} \rightarrow \Aut(U_h^f) \cong
\F_p^{\times}$. Note that it is trivial 
since $NU_h^f$ is a \p-group. 
Therefore for every $u$ in~$U_h$ not contained in
$\Gamma$, its conjugate $a^{-1} u a$ is
equal to $u$ if $a$ is contained in $NU_h^f$ and 
is not contained in $U_h$ 
otherwise. For each $\gamma$ in $\Gamma$, its conjugate $a^{-1}\gamma a$ 
always equals $\gamma$ as in the previous case. 
Consequently we have
\begin{align*}
 I_{U_h} =p^{N-1}\Z_p[[\Gamma]] \oplus \bigoplus_{i=1}^{p-1} p^{n_h-1} h^i \Z_p[[\Gamma]]
\end{align*}
(this equality is also valid when $N$ equals $1$).

\medskip
\noindent
\textbf{Calculation of $I_{U_{h,c}}$ for $h$ in
$\mathfrak{H}$ except for $e$ and $c$} ($N\geq 2$). 
We obtain a group antihomomorphism 
\begin{align*}
\inn \colon (NU_{h,c}^f)^{\mathrm{op}} \rightarrow \Aut(U_{h,c}^f)
\end{align*} 
in the same argument as that in the previous case. Since the automorphism group
$\Aut(U_{h,c}^f)$  is isomorphic to $\GL_2(\F_p)$ and 
its cardinality is equal to $p(p-1)^2(p+1)/2$, 
we have to consider the following two cases:
\begin{enumerate}[\bfseries {Case}\,(a)]
\item the image of $\inn$ is trivial;
\item the image of $\inn$ is a cyclic group of degree $p$.
\end{enumerate}

In Case (a) it is easy to see that $NU_{h,c}^f$ coincides with $NU_h^f$
(in particular the cardinality of $NU_{h,c}^f$ is equal to $p^{n_h}$). 
Therefore we may calculate $I_{U_{h,c}}$ in the same way 
as $I_{U_h}$, and we obtain
\begin{align*}
I_{U_{h,c}}=p^{N-2}\Z_p[[U_c]] \oplus \bigoplus_{i=1}^{p-1} p^{n_h-2} h^i \Z_p[[U_c]].
\end{align*}

In Case (b) we may readily show by easy computation that the image of the map 
$\inn$ is generated by automorphisms induced by 
$h^ic^j \mapsto h^i c^{ki+j}$ for each 
$0\leq k\leq p-1$. Its kernel obviously coincides with
$NU_h^f$, and the cardinality of $NU_{h,c}^f$ is thus equal to
$p^{n_h+1}$. This enables us to calculate $I_{U_{h,c}}$ as 
\begin{align*}
I_{U_{h,c}}=p^{N-2} \Z_p[[U_c]] \oplus \bigoplus_{i=1}^{p-1} p^{n_h-2} h^i (1+c+\dotsc +c^{p-1}) \Z_p[[\Gamma]].
\end{align*}

%%%%%%%%%%%%%%%%%%%%%%%%%%%%%%%%%%%%%%%%%%%%%%%%%%%%%%%%%
%
\subsection{Additive theta isomorphisms} \label{ssc:addtheta}
%
%%%%%%%%%%%%%%%%%%%%%%%%%%%%%%%%%%%%%%%%%%%%%%%%%%%%%%%%%

Now set
\begin{align*}
 \theta_A^+=(\theta^+_U)_{U\in \mathfrak{F}_A}
\colon \Z_p[[\conj{G}]] \rightarrow \prod_{U\in \mathfrak{F}_A}
\Z_p[[U]]
\end{align*}
and let $\Phi$ be the $\Z_p$-submodule of 
$\prod_{U\in \mathfrak{F}_A} \Z_p[[U]]$ consisting of 
all elements $y_{\bullet}$ satisfying the following two conditions:
\begin{itemize}
  \item (trace relation) the equation 
	$\Tr_{\Z_p[[U_h]]/\Z_p[[\Gamma]]}y_h=y_e$ holds 
	for each 
	$\Z_p[[U_h]]$-component $y_h$ of $y_{\bullet}$ 
	(see Figure~\ref{fig:trace});
 \item each $\Z_p[[U]]$-component $y_U$ of $y_{\bullet}$
       is contained in $I_U$.
\end{itemize}

\begin{figure}
\begin{align*}
\xymatrix{
 \Z_p[[U_h]] \ar[d] \\
  \parbox{2cm}{\leavevmode \hbox{$\quad \Z_p[[\Gamma]]$} \\  
  \hbox{$(=\Z_p[[U_e]])$}}
}
\end{align*}
\caption{Trace and norm relation for $\mathfrak{F}_A$.} 
\label{fig:trace}
\end{figure}

\begin{prop} \label{prop:addtheta}
The map $\theta_A^+$ induces an isomorphism of $\Z_p$-modules
\begin{align*}
\theta_A^+ \colon \Z_p[[\conj{G}]] \xrightarrow{\sim} \Phi.
\end{align*}
We call the induced isomorphism $\theta_A^+$ 
{\em the additive theta isomorphism for $\mathfrak{F}_A$}. 
\end{prop}

\begin{proof} It is easy to see that $\Phi$ contains the image of
 $\theta_A^+$ by construction.

\medskip
\noindent \textbf{Injectivity}. Take an element $y$ from the kernel of
 $\theta_A^+$ and let $\rho$ be an arbitrary Artin representation of
 $G$. Note that $\rho$ is isomorphic to
 a \mbox{$\Z[1/p]$-linear} combination $\sum_{U\in \mathfrak{F}_A} a_U
 \mathrm{Ind}^G_U \chi_U$ by condition $(\sharp)_A$
 where each $\chi_U$ is a finite-order character of the abelian
 group $U$. If we let $\chi_{\rho}$ denote the character
 associated to the Artin representation $\rho$, 
 we obtain the equation
\begin{align*}
\chi_{\rho} (y) =\sum_{U\in \mathfrak{F}_A} a_U \chi_U \circ \Tr_{\Z_p[[\conj{G}]]/\Z_p[[U]]} (y)
\end{align*}
by the explicit formula for induced characters 
 \cite[Section~7.2]{Serre1}. This implies that  
 $\chi_{\rho} (y)$ vanishes (recall that $y$ is an element 
 in the kernel of $\theta_A^+$); in other words, 
 the evaluation at $y$ of an arbitrary class function on $G$
 is equal to zero. Therefore $y$ itself is trivial.

\medskip
\noindent \textbf{Surjectivity}. For an arbitrary element $y_{\bullet}$ in $\Phi$, let $y$ be the element in $\Z_p[[\conj{G}]]$ defined by
\begin{align*}
y=p^{-N} [y_e]+\sum_{h\in \mathfrak{H}\setminus\{e\}} p^{-n_h+1} ([y_h]-p^{-1}[y_e])
\end{align*}
(we use the bracket $[\, \cdot \,]$ for the corresponding 
 element in $\Z_p[[\conj{G}]]$).
Then it is not difficult at all to check that 
 the image of $y$ under the map $\theta_A^+$ coincides with
 $y_{\bullet}$.
\end{proof}

\begin{cor} \label{cor:determine}
Every element $y$ in $\Z_p[[\conj{G}]]$ is completely determined by
 its trace images $\{\theta_{U}^+(y)\}_{U\in \mathfrak{F}_A}$.
\end{cor}

Next we extend the notion of the additive theta isomorphism 
to the~Brauer family $\mathfrak{F}_B$; for each $(U,V)$ in $\mathfrak{F}_B$
let $\theta_{U,V}^+$ be the composite map 
\begin{align*}
\Z_p[[\conj{G}]] \xrightarrow{\Tr_{\Z_p[[\conj{G}]]/\Z_p[[\conj{U}]]}}
 \Z_p[[\conj{U}]] \xrightarrow{\mathrm{canonical}} \Z_p[[U/V]]
\end{align*}
and set
$\theta^+_B=(\theta_{U,V}^+)_{(U,V)\in \mathfrak{F}_B}$. 
We define the $\Z_p$-submodule $\Phi_B$ of the direct product 
$\prod_{(U,V)\in \mathfrak{F}_B} \Z_p[[U/V]]$ 
as the submodule consisting of all elements 
$(y_{U,V})_{(U,V)\in
\mathfrak{F}_B}$ satisfying the following {\em trace compatibility condition}
(TCC, see Figure~\ref{fig:tcc}):
\begin{quote}
 the equation 
 $\Tr_{\Z_p[[U/V]]/\Z_p[[U'/V]]}(y_{U,V})=\can^{V'}_V(y_{U',V'})$ holds for
 arbitrary pairs
 $(U,V)$ and $(U',V')$ in $\mathfrak{F}_B$ such that $U$ contains $U'$ 
 and $U'$ contains $V$ respectively 
 (we denote by $\can^{V'}_V$ the canonical surjection 
 $\Z_p[[U'/V']]\rightarrow \Z_p[[U'/V]]$);
\end{quote}
and the following {\em conjugacy compatibility condition} (CCC$+$):
\begin{quote}
the equation 
 $y_{U',V'}=\psi_a(y_{U,V})$ holds 
 if $(U,V)$ and $(U',V')$ are elements in
 $\mathfrak{F}_B$ such that $U'=a^{-1}Ua$ and $V'=a^{-1}Va$ hold for 
 a certain element $a$ in $G$ 
 (we denote by $\psi_a$ the isomorphism 
 $\Z_p[[U/V]] \xrightarrow{\sim} \Z_p[[U'/V']]$ induced 
 by the conjugation $U/V
 \rightarrow U'/V'; u \mapsto a^{-1}ua$).
\end{quote}
Note that we may naturally regard $\mathfrak{F}_A$ as a subfamily of
$\mathfrak{F}_B$ (by identifying $U$ in~$\mathfrak{F}_A$ with the pair 
$(U,\{e\})$ in~$\mathfrak{F}_B$).

\begin{figure}
\begin{align*}
\xymatrix{
\Z_p[[U/V]] \ar[dr]_{\mathrm{trace}} &  &  \Z_p[[U'/V']]
 \ar[dl]^{\can^{V'}_V}  \\
 & \Z_p[[U'/V]]
}
\end{align*}
\caption{Trace compatibility condition for $\mathfrak{F}_B$.}
\label{fig:tcc}
\end{figure}

\begin{prop} \label{prop:addtheta-B}
Let $(y_{U,V})_{(U,V)\in \mathfrak{F}_B}$ be an element in
 $\Phi_B$ and assume that $(y_{U,\{e\}})_{U\in \mathfrak{F}_A}$ is contained in
 $\Phi$. Then there exists a unique element $y$ in
 $\Z_p[[\conj{G}]]$ which satisfies $\theta_B^+(y)=(y_{U,V})_{(U,V)\in
 \mathfrak{F}_B}$. 
\end{prop}

\begin{proof}
Consider the following commutative diagram (we denote 
the canonical projection by $\mathrm{proj}$):
\begin{align*}
\xymatrix{
\Z_p[[\conj{G}]] \ar[rr]^(0.45){\theta_B^+} \ar@{=}[d] & &
 \prod_{(U,V)\in \mathfrak{F}_B} \Z_p[[U/V]]
 \ar[d]^{\mathrm{proj}} \\
\Z_p[[\conj{G}]] \ar[r]_(0.7){\theta_A^+}^(0.7){\sim} & \Phi \ar@{^(->}[r] &
 \prod_{U\in \mathfrak{F}_A} \Z_p[[U]].
}
\end{align*}
Then 
 $\mathrm{proj}((y_{U,V})_{(U,V)\in \mathfrak{F}_B})$ is contained in
 $\Phi$ by assumption, and thus there exits a unique element $y$ 
 in $\Z_p[[\conj{G}]]$ which corresponds to 
 the element $\mathrm{proj}((y_{U,V})_{(U,V)\in\mathfrak{F}_B})$ via 
 the additive theta isomorphism $\theta_A^+$ for $\mathfrak{F}_A$
 (Proposition~\ref{prop:addtheta}). We have to show that 
 $\theta_B^+(y)$ coincides with $(y_{U,V})_{(U,V)\in \mathfrak{F}_B}$,
 and for this purpose it suffices to
 show that $\mathrm{proj}$ induces an injection on $\Phi_B$ (note
 that the element $(\theta^+_{U,V}(y))_{(U,V)\in \mathfrak{F}_B}$ 
 obviously satisfies both (TCC) and (CCC$+$)
 by construction; hence $\theta_B^+(y)$ is  an element in $\Phi_B$).
Let $(z_{U,V})_{(U,V)\in \mathfrak{F}_B}$ be an element in $\Phi_B$
 satisfying the following equation:
\begin{align} \label{eq:indfirst}
\mathrm{proj}((z_{U,V})_{(U,V)\in \mathfrak{F}_B})=(z_{U,\{e\}})_{U\in \mathfrak{F}_A}=0.
\end{align}
 We shall prove that $z_{U,V}=0$ holds for each $(U,V)$ in $\mathfrak{F}_B$ 
 by induction
 on the cardinality of $U^f$. 
 First note that $z_{U,\{e\}}=0$ holds for $(U,\{ e\})$ 
 if the cardinality of $U^f$ is smaller than $p$ 
 (use (\ref{eq:indfirst}) and
 the conjugacy compatibility condition (CCC$+$)). Now let $(U,V)$ be  
 an element in $\mathfrak{F}_B$ such that the degree of $U^f$ 
 is equal to $p^k$ for certain $k$ larger than $2$ and set $W=U^f/V^f$. 
 Then the abelian group
 $W$ is isomorphic to $(\Z/p\Z)^{\oplus d}$ 
 for a certain natural number $d$ smaller than $k$
 (the structure theorem of finite abelian groups). Moreover
 we may assume that $d$ is 
 larger than $2$.\footnote{We may easily verify 
 that the cardinality of $V^f$ is 
 always smaller than $p^{k-2}$ by induction on the cardinality of $U^f$.}
 Since the element $z_{U,V}$ is represented 
 as a $\iw{\Gamma}$-linear combination $\sum_{w\in W} a_w w$, 
 it suffices to prove that $a_w$ equals zero for every $w$ in $W$.
 Take an arbitrary element $x$ of degree $p$ in $W$, and let~$U_x^f$
 denote the inverse image of $\langle x \rangle$---the cyclic subgroup
 of $W$ generated by $x$--- under the canonical
 surjection $U^f \rightarrow W$. Obviously the cardinality of $U_x^f$ 
 is strictly smaller than $p^k$. 
 If we set $U_x=U_x^f \times \Gamma$, we may explicitly calculate the image 
 of $z_{U,V}$ under the trace map from $\Z_p[[U/V]]$ to $\Z_p[[U_x/V]]$
 as $\sum_{i=0}^{p-1} p^{d-1} a_{x^i} x^i$. On the other
 hand the element $z_{U_x,V_x}$ is equal to zero by
 induction hypothesis (here $V_x$ denotes the commutator subgroup of $U_x$). 
 Therefore $a_{x^i}=0$ holds for each $i$ by (TCC). 
 Replacing $x$ appropriately, we may verify that
 $a_w=0$ holds for every $w$ in $W$.
\end{proof}

%%%%%%%%%%%%%%%%%%%%%%%%%%%%%%%%%%%%%%%%%%%%%%%%%%%%%%%%%
%
%
\section{Preliminaries for logarithmic translation} \label{sc:prelim}
%
%
%%%%%%%%%%%%%%%%%%%%%%%%%%%%%%%%%%%%%%%%%%%%%%%%%%%%%%%%%

This section is devoted to technical preliminaries 
for arguments in Section~\ref{sc:trans}.

%%%%%%%%%%%%%%%%%%%%%%%%%%%%%%%%%%%%%%%%%%%%%%%%%%%%%%%%%
%
\subsection{Augmentation theory} \label{ssc:aug}
%
%%%%%%%%%%%%%%%%%%%%%%%%%%%%%%%%%%%%%%%%%%%%%%%%%%%%%%%%%

For each $(U,V)$ in $\mathfrak{F}_B$, let $\aug_{U,V}$ denote 
the augmentation map
from $\iw{U/V}$ to $\iw{\Gamma}$ (namely it is 
a ring homomorphism induced by the projection $U/V\rightarrow \Gamma$), 
and let $\overline{\aug}_{U,V} \colon
\riw{U/V} \rightarrow \riw{\Gamma}$ be its reduction modulo $p$. 
Let $\varphi \colon \iw{G}\rightarrow \iw{\Gamma}$
denote ``the Frobenius endomorphism'' on $\iw{G}$ defined as the composition
\begin{align*}
\iw{G} \xrightarrow{\aug_G} \iw{\Gamma} \xrightarrow{\varphi_{\Gamma}}
 \iw{\Gamma}
\end{align*}
where $\aug_G$ is the canonical augmentation map and 
$\varphi_{\Gamma}$ is the Frobenius endomorphism on $\iw{\Gamma}$ induced
by $\gamma \mapsto \gamma^p$. 
Let $\theta_{U,V}$ denote the composition of the norm map
 $\Nr_{\iw{G}/\iw{U}}$ with the canonical map $K_1(\iw{U})\rightarrow
 \iw{U/V}^{\times}$.

The author is grateful to Takeshi Tsuji for presenting the following 
useful proposition to him.

\begin{prop} \label{prop:cong-j}
Let $(U,V)$ be an element in $\mathfrak{F}_B$ and $J_{U,V}$ 
the kernel of the composite map
\begin{align*}
\iw{U/V} \xrightarrow{\aug_{U,V}} \iw{\Gamma} \xrightarrow{\mathrm{mod}\, p}
 \riw{\Gamma}.
\end{align*}
Then the element $\varphi(x)^{-(G:U)/p} \theta_{U,V}(x)$ is contained 
in $1+J_{U,V}$
 for each $x$ in $K_1(\iw{G})$ if $U$ is a {\em proper} subgroup of $G$.
In other words, the congruence 
$\theta_{U,V}(x) \equiv \varphi(x)^{(G:U)/p} \, \mod{J_{U,V}}$ 
holds unless $U$ coincides with $G$.
\end{prop}

Before the proof we remark that 
the image of an element $x$ in $K_1(\iw{G})$ under the map $\theta_{U,V}$ 
can be calculated 
as follows: since 
the Iwasawa algebra $\iw{G}$ is regarded as a left free
$\iw{U}$-module of finite rank $r=(G:U)$, the right-multiplication-$x$ map
is represented by an invertible matrix $A_x$ 
with entries in $\iw{U}$.\footnote{By abuse of notation, 
we use the same symbol $x$ for an arbitrary
lift of $x$ to $\iw{G}^{\times}$.}
The element $\theta_{U,V}(x)$ then coincides with the determinant of 
the image of $A_x$ under the canonical map 
$\GL_r(\iw{U}) \rightarrow \GL_r(\iw{U/V})$.

\begin{proof}
 The claim is equivalent to the following
 Proposition \ref{prop:cong-jmodp} since both $\iw{G}$ and $\iw{U/V}$ are
 \p-adically complete.
\end{proof}

Let $\bar{\varphi} \colon \riw{G} \rightarrow \riw{\Gamma}$ denote the
Frobenius endomorphism on $\riw{G}$ defined as $\varphi \otimes_{\Z_p}
\F_p$ and 
let $\bar{\theta}_{U,V}$ denote the composition of the norm map
 $\Nr_{\riw{G}/\riw{U}}$ with the canonical map $K_1(\riw{U})\rightarrow
 \riw{U/V}^{\times}$.

\begin{prop} \label{prop:cong-jmodp}
Let $\bar{J}_{U,V}$ be the kernel of the augmentation map
\begin{align*}
\riw{U/V} \xrightarrow{\overline{\aug}_{U,V}} \riw{\Gamma}.
\end{align*}
Then the element defined as $\bar{\varphi}(x)^{-(G:U)/p}
 \bar{\theta}_{U,V}(x)$ is contained in $1+\bar{J}_{U,V}$
 for each $x$ in $K_1(\riw{G})$. 
\end{prop}

\begin{rem}
Proposition~\ref{prop:cong-jmodp} is valid even if $U$ coincides with $G$ 
(indeed $\bar{\varphi}(x)$ can be described 
as a \p-th power of a certain element, 
see the proof of Proposition~\ref{prop:cong-jmodp}). 
However there exists an obstruction for taking the projective limit 
if the exponent $(G:U)/p$ of $\varphi(x)$ is {\em not} integral. 
Therefore the case where $U$ coincides with $G$ remains 
as an exception to Proposition~\ref{prop:cong-j}.
\end{rem}

Proposition~\ref{prop:cong-jmodp} is deduced 
from the following elementary lemma in modular representation theory.

\begin{lem} \label{lem:nilp}
Let $K$ be a field of positive characteristic $p$, $\Delta$ a finite
 \p-group and $V$ a finite dimensional representation space of $\Delta$ over
 $K$. Let $\aug$ denote the canonical augmentation map $K[\Delta]\rightarrow K$.
Take a natural number $n$ such that $p^n$ is larger 
than the $K$-dimension of $V$. 
Then for each $x$ in $K[\Delta]$, the action of $x^{p^n}$ upon $V$
 coincides with the multiplication by $\aug(x)^{p^n}$.
In particular the equation $x^{\sharp \Delta}=\aug(x)^{\sharp \Delta}$ holds.
\end{lem}

\begin{proof}
Let $d$ be the $K$-dimension of $V$.
The group ring $K[\Delta]$ is a local ring whose maximal ideal is the
 augmentation ideal since $K$ is of characteristic~$p$ and $\Delta$ is a
 \p-group. Therefore the only simple 
$K[\Delta]$-module (up to isomorphisms) is $K$ endowed with trivial
 $\Delta$-action, and moreover
there exists a Jordan-Schreier composition series 
\begin{align*}
V=V_1 \supsetneq V_2 \supsetneq \cdots \supsetneq V_d \supsetneq V_{d+1}=\{0\}
\end{align*}
such that each quotient space $V_i/V_{i+1}$ is 
isomorphic to $K$. Take 
an arbitrary element $e_i$ in $V_i$ not contained in
 $V_{i+1}$ for each $1\leq i\leq d$. Then $\{ e_1, e_2. \dotsc , e_d\}$
 forms a basis of $V$ over $K$, with respect to which the action of $x$
 is represented by an upper triangular matrix all of whose diagonal
 components are equal to $\aug(x)$. This implies the first claim.
 The second claim is deduced from the first one (take $n$ as the 
 \p-order of $\Delta$ and apply the claim to the
 regular representation $V=K[\Delta]$).
\end{proof}

\begin{proof}[Proof of Proposition $\ref{prop:cong-jmodp}$]
Identify the modulo $p$ Iwasawa algebra $\riw{U/V}$ 
with the group ring $\riw{\Gamma}[U^f/V^f]$, and let $K$ be 
the fractional field of~$\riw{\Gamma}$. 
Then we may naturally regard each $t$ in $\riw{U/V}$ as an
 element in $K[U^f/V^f]$, and therefore the equation
\begin{align} \label{eq:nilp}
t^{\sharp (U^f/V^f)}=\overline{\aug}_{U,V}(t)^{\sharp (U^f/V^f)}
\end{align}
holds by Lemma \ref{lem:nilp}.
Now let $x$ be an arbitrary element in $\riw{G}^{\times}$, and set
 $z=\overline{\aug}_G(x)$ and $y=xz^{-1}$. 
Then we obtain $\overline{\aug}_G(y)=1$
 and 
\begin{align} \label{eq:norm}
 \bar{\theta}(x)=\bar{\theta}(y) \bar{\theta}(z)
\end{align}
by definition (here we denote the map $\bar{\theta}_{U,V}$ by 
 $\bar{\theta}$ to simplify the notation). Since $z$ is an element in
 $\riw{\Gamma}$ (and hence $z$ is contained in the centre of
 $\riw{G}$), the image of $z$ under the norm map $\bar{\theta}$ 
 coincides with $z^{(G:U)}$ by direct calculation. 
 On the other hand we may calculate
 $\bar{\varphi}(x)$ as follows: 
\begin{align} \label{eq:z^p}
 \bar{\varphi}(x) =\bar{\varphi}(y) \bar{\varphi}(z)
 =\bar{\varphi}(\overline{\aug}_G(y)) \bar{\varphi}(z)
 =z^p \qquad (\text{use }\overline{\aug}_G(y)=1). 
\end{align} 
Hence the equation $\bar{\varphi}(x)^{-(G:U)/p} \bar{\theta}(x)
 =\bar{\theta}(y)$ holds by (\ref{eq:norm}) and (\ref{eq:z^p}). 
Moreover (\ref{eq:nilp}) implies that $y^{p^N}$ is 
 equal to $\overline{\aug}_G(y)^{p^N}=1$, 
 and therefore $\bar{\theta}(y)^{p^N}$ is 
 also trivial. The same argument as above 
 derives a similar equation $\bar{\theta}(y)^{\sharp
 (U^f/V^f)}=\overline{\aug}_{U,V}(\bar{\theta}(y))^{\sharp (U^f/V^f)}$,
 and consequently the equation
 \begin{align*}
  \overline{\aug}_{U,V}(\bar{\theta}(y))^{p^N}=\bar{\theta}(y)^{p^N}=1
\end{align*} 
holds. Since $\riw{\Gamma}$ is a domain  
(recall that $\riw{\Gamma}$ is isomorphic to the formal power
 series ring $\F_p[[T]]$), the last equation implies
 that $\overline{\aug}_{U,V}(\bar{\theta}(y))$ itself is trivial; 
in other words $\bar{\theta}(y)$ is contained in $1+\bar{J}_{U,V}$.
\end{proof}

The last paragraph of the proof above implies that $\theta_{U,V} (y)$ 
is contained in
$1+J_{U,V}$ if $y$ is an element in $\iw{G}$ satisfying 
$\aug_G(y)\equiv 1 \, \mod p$. 
By replacing $G$ and $U$ appropriately, we obtain the following useful
corollary:

\begin{cor} \label{cor:aug-norm}
Let $(U,V)$ be an element in $\mathfrak{F}_B$ such that 
$U$ does not contain a non-trivial central element $c$. 
Then the norm map $\Nr_{\iw{U\times \langle c \rangle /V}/\iw{U/V}}$ 
induces a group homomorphism
 from $1+J_{U\times \langle c \rangle, V}$ to $1+J_{U,V}$.
\end{cor}

\begin{rem}
Both $1+J_{U\times \langle c \rangle,V}$ and $1+J_{U,V}$ 
are actually multiplicative groups; see
 Proposition~\ref{prop:log-j} for details. 
\end{rem}

%%%%%%%%%%%%%%%%%%%%%%%%%%%%%%%%%%%%%%%%%%%%%%%%%%%%%%%%%
%
\subsection{Logarithmic theory} \label{ssc:log}
%
%%%%%%%%%%%%%%%%%%%%%%%%%%%%%%%%%%%%%%%%%%%%%%%%%%%%%%%%%

Let us study the \p-adic logarithm 
on $1+J_{U,V}$ for each $(U,V)$ in $\mathfrak{F}_B$, as well as those on 
$1+I_U$ for each $U$ in $\mathfrak{F}_A^c$. 

\begin{prop} \label{prop:log-j}
For each $(U,V)$ in $\mathfrak{F}_B$, let $J_{U,V}$ be 
as in Proposition~$\ref{prop:cong-j}$. Then 
\begin{enumerate}[$(1)$]
\item the subset $1+J_{U,V}$ of $\iw{U/V}$ is 
      a multiplicative subgroup of $\iw{U/V}^{\times}${\upshape ;}
\item for each $y$ in $J_{U,V}$, the logarithm 
      $\log (1+y)=\sum_{m=1}^{\infty} (-1)^{m-1}
      y^m/m$ converges \p-adically 
      in $\iw{U/V}\otimes_{\Z_p}\Q_p${\upshape ;} 
\item the kernel $($resp.\  image$)$ of the induced homomorphism 
      \begin{align*}
       \log \colon 1+J_{U,V} \rightarrow \iw{U/V}\otimes_{\Z_p} \Q_p
      \end{align*} 
      is $\mu_p(\iw{U/V})$
      $($resp.\ is contained in $\iw{U/V})$ where $\mu_p(\iw{U/V})$ denotes
      the multiplicative subgroup of $\iw{U/V}^{\times}$ consisting of 
      all \p-power roots of unity.
\end{enumerate}
\end{prop}

\begin{proof} 
Define $\overline{\aug}_{U,V} \colon \riw{U/V} \rightarrow
 \riw{\Gamma}$ similarly to the previous subsection, and let
 $\bar{J}_{U,V}$ be the kernel of $\overline{\aug}_{U,V}$. 
Since $\riw{U/V}$
 is commutative, we have 
\begin{align*}
\bar{y}^p &= \sum_{u\in U^f/V^f} \bar{y}_u^p u^p =\sum_{u\in U^f/V^f}
 \bar{y}_u^p 
 = \left( \sum_{u\in U^f/V^f} \bar{y}_u \right)^p
 =(\overline{\aug}_{U,V}(\bar{y}))^p =0
\end{align*}
 for an element $\bar{y}=\sum_{u\in U^f/V^f} \bar{y}_u u$ in
 $\bar{J}_{U,V}$ (each $\bar{y}_u$ is an element in
 $\riw{\Gamma}$). 
 Therefore $y^p$ is contained in $p\iw{U/V}$ for each $y$ in
 $J_{U,V}$.
\begin{enumerate}[(1)]
\item By the remark above, $(1+y)^{-1}=\sum_{m=0}^{\infty} (-y)^m$ converges
      \p-adically in $1+J_{U,V}$ for each $1+y$ in $1+J_{U,V}$.
\item In a similar way the element 
      $y^m$ is contained in $p^{[m/p]}\iw{U/V}$
      for each $y$ in $J_{U,V}$ (for a real number $x$, we denote by $[x]$ the
      largest integer not greater than $x$), and hence the claim holds.
\item  If we take an element $x=1+y$ from $1+J_{U,V}$,  
       we may calculate as $\bar{x}^p=1+\bar{y}^p=1$ 
       where $\bar{x}=1+\bar{y}$ is the image of $x$ in
       $1+\bar{J}_{U,V}$. 
       This implies that $x^p$ is an element in
      $1+p\iw{U/V}$ since the \p-adic exponential map and 
       the \p-adic logarithmic map
      define an isomorphism between $p\iw{U/V}$ and $1+p\iw{U/V}$
      in general (recall that  
      $p$ is odd). Therefore $p\log x(=\log x^p)$ is contained in
      $p\iw{U/V}$, or equivalently $\log(1+J_{U,V})$ is contained in 
      $\iw{U/V}$. Furthermore if we assume that  
      $\log x =0$ holds for $x$ in $1+J_{U,V}$, 
      we obtain  $x^p=1$ by the calculation above, which implies that
      $x$ is an element in $\mu_p(\iw{U/V})$. 
      Conversely $\log x$ vanishes for an arbitrary
      element $x$ in $\mu_p(\iw{U/V})$ since
      $\iw{U/V}$ is free of \p-torsion.
\end{enumerate}
\end{proof}

\begin{lem} \label{lem:prop-is}
For each $U$ in $\mathfrak{F}_A^c$, let $I_U$ be the $\Z_p$-submodule of
 $\iw{U}$ defined as in Section~$3$. Then $I_U^2$ is contained in
 $I_U$. Moreover, 
\begin{enumerate}[$(1)$]
\item when $N$ is larger than $1$, the $\Z_p$-module 
      $I_{\Gamma}$ is contained in
      $J_{\Gamma}=p\iw{\Gamma}$$;$
\item when $N$ is larger than $2$,
      the $\Z_p$-module 
      $I_{U_h}$ is contained in $p\iw{U_h}$ $($hence also in $J_{U_h})$ 
      for each $h$ in $\mathfrak{H}\setminus \{e\}$, and 
	there exist canonical inclusions 
	$p^{k(N-1)} I_{U_h} \subseteq I_{U_h}^{k+1} \subseteq 
	p^{k(n_h-1)}I_{U_h}$ for an arbitrary natural number $k$$;$ 
\item when $N$ is larger than $3$, the $\Z_p$-module 
      $I_{U_{h,c}}$ is contained in
      $p\iw{U_{h,c}}$ $($hence also in~$J_{U_{h,c}})$ 
      for each $h$ in $\mathfrak{H}\setminus \{ e, c\}$ 
      satisfying the condition of Case~$(a)$, and 
	$p^{k(N-2)} I_{U_{h,c}} \subseteq I_{U_{h,c}}^{k+1} \subseteq 
	p^{k(n_h-2)}I_{U_{h,c}}$ holds for an arbitrary natural number $k$$;$ 
\item when $N$ is larger than $3$, the $\Z_p$-module 
      $I_{U_{h,c}}$ is contained in
      $J_{U_{h,c}}$ for each $h$ in $\mathfrak{H}\setminus \{ e, c\}$ 
      satisfying the
      condition of Case $(b)$, and there exist canonical inclusions  
      $p^{k(N-2)}I_{U_{h,c}}^2 \subseteq 
      I_{U_{h,c}}^{k+2} \subseteq p^{k(n_h-1)} I_{U_{h,c}}^2$ for an
      arbitrary natural number $k$. 
\end{enumerate}
\end{lem}

\begin{proof}
The first claim is easily checked by direct calculation.
\begin{enumerate}[(1)] 
\item Obvious from the exact description of $I_{\Gamma}$ (see Section~\ref{ssc:addtheta}). 
\item If $h$ is contained in the
 centre of $G^f$, the equation $n_h=N$ clearly holds. Otherwise $NU_h^f$
 has to contain the centre of $G^f$, and therefore $n_h$ is at
      least~$2$. 
      In both
 cases $I_{U_h}$ is contained in $p\iw{U_h}$. The last claim is obvious from 
	the explicit description of $I_{U_h}$.
\item Recall that $\sharp NU_{h,c}^f=p^{n_h}$ holds in Case
      (a). Let $\bar{U}^f_h$ denote the quotient
      group $U^f_{h,c}/U_c^f$. If $\bar{h}$ (the image of
      $h$ in $\bar{U}^f_h$) is contained in the centre of
      $\bar{G}^f=G^f/U_c^f$, 
      the normaliser of $\bar{U}_h^f$ 
      obviously coincides with~$\bar{G}^f$, which implies 
      that $n_h$ is equal to $N$. 
      Otherwise there exists a non-trivial element $\bar{a}$ in the
      centre of $\bar{G}^f$. Let $a$ be its lift to $G^f$, then the finite 
      subgroup of $G^f$ generated by $c,h$ and $a$ is contained in $NU_{h,c}^f$
      by construction. This implies that $n_h$
      is at least $3$. In both cases we may conclude 
      that $I_{U_{h,c}}$ is 
      contained in $p\iw{U_{h,c}}$. The last claim is obvious from 
	the explicit description of $I_{U_{h,c}}$.
\item First note that $2\leq n_h\leq N-1$ holds since the cardinality of
      $NU_{h,c}^f$  (which is smaller than $p^N$) is equal to $p^{n_h+1}$.
      By using this fact, we may exactly calculate as
      \begin{align*}
       I_{U_{h,c}}^k&=(p^{k(N-2)}\Z_p[[U_c]]+p^{k(n_h-1)-1}(1+c+\cdots
       +c^{p-1}) \Z_p[[\Gamma]])  \\
       & \qquad \oplus \bigoplus_{i=1}^{p-1} 
       p^{k(n_h-1)-1}h^i (1+c+\cdots +c^{p-1})\Z_p[[\Gamma]].
      \end{align*}
      for each $k$ larger than $2$. 
      The claim holds by this calculation.
\end{enumerate}
\end{proof}

\begin{prop} \label{prop:log-i}
Let $U$ be an element in $\mathfrak{F}_A^c$ and assume that $U$ does not
 coincide with $G$ if $N$ equals either $0$, $1$ or $2$. Then
\begin{enumerate}[$(1)$]
 \item the subset $1+I_U$ of $\iw{U}$ is 
       a multiplicative subgroup of $\iw{U}^{\times}${\upshape ;}
\item for each $y$ in $I_U$, the logarithm 
      $\log (1+y)=\sum_{m=1}^{\infty} (-1)^{m-1}
      y^m/m$ converges \p-adically in $I_U${\upshape ;}
\item the \p-adic logarithmic homomorphism induces an isomorphism between
      $1+I_U$ and $I_U$.
\end{enumerate}
\end{prop}

\begin{proof}
The claims of (1) and (2) follow from 
 Lemma \ref{lem:prop-is} (use the fact that
 $y^{p^m}/p^m$ is contained in $I_U$ for each $y$ in $I_U$ if $p$ is odd).
For (3), first note that $1+I_U^k$ is a multiplicative
 subgroup of $1+I_U$ and the \p-adic logarithm induces a homomorphism 
from $1+I_U^k$ to $I_U^k$ for each natural number $k$ 
 (similarly to (1) and (2)). 
 Moreover the $I_U$-adic topology on $I_U$ coincides with the \p-adic
 topology by Lemma~\ref{lem:prop-is}. 
 Therefore it suffices to show that the \p-adic
 logarithm induces an isomorphism
\begin{align*}
\log \colon 1+I_U^k/1+I_U^{k+1} \rightarrow I_U^k/I_U^{k+1} ;
 1+y \mapsto y
\end{align*}
for each natural number $k$. 
 Let $y$ be an element in $I_U^k$. We have only 
 to show that $y^{p^m}/p^m$ is contained in $I_U^{k+1}$ for each
 $m\geq 1$, or equivalently, $p^{-m}I_U^{kp^m}$ is contained 
 in $I_U^{k+1}$ for every $k$ and $m$. We may verify it 
 by direct calculation.\footnote{In this calculation we use the fact that $p$ is {\em odd}.}

\end{proof}

\begin{rem}
Suppose that $N$ equals either $0,1$ or $2$. 
 Then the Artinian family $\mathfrak{F}_A^c$
 contains the whole group $G$ by definition. The $\Z_p$-module $I_G$ 
 obviously coincides
 with $\iw{G}$. Thus the \p-adic logarithm never converges on
 $1+I_G$. We remark that the $\Z_p$-module 
 $I_G=\iw{G}$ is the only exception to our
 logarithmic theory discussed in this subsection.
\end{rem}

%%%%%%%%%%%%%%%%%%%%%%%%%%%%%%%%%%%%%%%%%%%%%%%%%%%%%%%%%%%%%
%
%
\section{Construction of the theta isomorphism II ---translation---} \label{sc:trans}
%
%
%%%%%%%%%%%%%%%%%%%%%%%%%%%%%%%%%%%%%%%%%%%%%%%%%%%%%%%%%%%%%

In this section we shall construct the multiplicative theta isomorphism 
by using the facts studied in Section~\ref{sc:prelim}.

%%%%%%%%%%%%%%%%%%%%%%%%%%%%%%%%%%%%%%%%%%%%%%%%%%%%%%%%%
%
\subsection{The multiplicative theta isomorphism} \label{ssc:mult}
%
%%%%%%%%%%%%%%%%%%%%%%%%%%%%%%%%%%%%%%%%%%%%%%%%%%%%%%%%%

Let $(U,V)$ be an element in $\mathfrak{F}_B$.
We use the notion 
``$x\equiv y \, \pmod{\mathfrak{I}}$'' for elements $x$ and $y$ 
in $\piw{U/V}^{\times}$ such that $xy^{-1}$ is contained in 
$1+\mathfrak{I}^{\, \widetilde{\,}}$ \nobreakdash---the image 
of $1+\mathfrak{I}$ under the canonical surjection 
$\iw{U/V}^{\times}\rightarrow \piw{U/V}^{\times}$\nobreakdash--- 
if $\mathfrak{I}$ is a $\Z_p$-submodule
of $\iw{U/V}$ such that $1+\mathfrak{I}$ is a multiplicative
subgroup of $\iw{U/V}^{\times}$.
Let $\widetilde{\Psi}'$ denote the subgroup of $\prod_{(U,V)\in
\mathfrak{F}_B} \piw{U/V}^{\times}$ consisting of all elements
$\eta_{\bullet}=(\eta_{U,V})_{(U,V)\in \mathfrak{F}_B}$ 
satisfying the following three conditions:

\begin{itemize}
 \item (norm compatibility condition, NCC)

\noindent the
       equation $\Nr_{\iw{U/V}/\iw{U'/V}}(\eta_{U,V})=\can^{V'}_V(\eta_{U',V'})$
       holds for 
       $(U,V)$ and $(U',V')$ in $\mathfrak{F}_B$ such that
       $U$ contains $U'$ and $U'$ contains $V$ respectively
       (here $\can^{V'}_V$ is the canonical map 
       $\iw{U'/V'}\rightarrow \iw{U'/V}$);

 \item (conjugacy compatibility condition, CCC)

\noindent the equation $\eta_{U',V'}=\psi_a(\eta_{U,V})$ holds 
       for $(U,V)$ and $(U',V')$ in $\mathfrak{F}_B$ such that 
       $U'=a^{-1}Ua$ and $V'=a^{-1}Va$ hold
       for a certain element $a$ in $G$ 
       (we denote by $\psi_a$ the isomorphism $\iw{U/V}^{\times}
       \xrightarrow{\sim} \iw{U'/V'}^{\times}$ induced by the
       conjugation 
       $U/V \rightarrow U'/V'; u \mapsto a^{-1}ua$);

 \item (congruence condition)

\noindent the congruence $\eta_{U,V} \equiv \varphi(\eta_{\mathrm{ab}})^{{(G:U)}/p} \,
       \pmod{J_{U,V}}$ holds for
       $(U,V)$ in $\mathfrak{F}_B$ except for $(G,[G,G])$
       where $\eta_{\mathrm{ab}}$ denotes 
       the $\piw{G^{\mathrm{ab}}}^{\times}$-component of
       $\eta_{\bullet}$ (see the previous section for the 
       definition of $J_{U,V}$). 
\end{itemize}

Let $\widetilde{\Psi}$ (resp.\ $\widetilde{\Psi}_c$) be the subgroup of
$\widetilde{\Psi}'$ consisting of all elements $\eta_{\bullet}$ 
satisfying the following {\em additional congruence condition} (see
Section~\ref{sc:additive} for the definition of $I_U$):

\begin{quotation}
 (additional congruence condition)

\noindent the congruence 
 $\eta_U \equiv \varphi(\eta_{\mathrm{ab}})^{(G:U)/p} \,
 \pmod{I_U}$ holds for each $U$ in $\mathfrak{F}_A$ (resp.\ $\mathfrak{F}_A^c$).
\end{quotation}

\begin{rem} \label{rem:G}
When $N$ equals either $0$, $1$ or $2$, we regard 
the additional congruence condition for the total group $G$ 
as {\em the trivial condition} (in other words, we do not impose 
any congruence condition upon $G$). Therefore we have only to consider 
an element $(U,V)$ in $\mathfrak{F}_B$ (resp.\ $U$ in $\mathfrak{F}_A^c$) 
such that $U$ is a {\em proper} subgroup of $G$ in arguments 
concerning with congruence conditions.
\end{rem}

\begin{rem} \label{rem:psic-psi}
For each $U$ in $\mathfrak{F}_A^c$, we may easily check that 
the ideal $J_U$ contains $I_U$ unless $U$ coincides with $G$
by using the explicit description of $I_U$ 
given in Section~\ref{ssc:calculation}; 
in particular $\widetilde{\Psi}_c$ is a subgroup of $\widetilde{\Psi}$. 
\end{rem}

Let $\theta_{U,V}$ be as in Section~\ref{ssc:aug} and set
$\theta=(\theta_{U,V})_{(U,V)\in \mathfrak{F}_B}$, then the map 
$\theta$ induces a group homomorphism $\tilde{\theta} \colon
\pK_1(\iw{G}) \rightarrow \prod_{(U,V)\in \mathfrak{F}_B}
\piw{U/V}^{\times}$.

\begin{prop} \label{prop:mult-theta}
The multiplicative group $\widetilde{\Psi}$ coincides with
 $\widetilde{\Psi}_c$. Moreover the map $\tilde{\theta}$ induces an isomorphism
\begin{align*}
\tilde{\theta} \colon \pK_1(\iw{G}) \xrightarrow{\sim} \widetilde{\Psi}
 \quad (=\widetilde{\Psi}_c).
\end{align*}
\end{prop}

In order to prove Proposition~\ref{prop:mult-theta}, 
it suffices to verify surjectivity of
$\pK(\iw{G}) \rightarrow \widetilde{\Psi}$ and injectivity of 
$\pK(\iw{G}) \rightarrow \widetilde{\Psi}_c$ (see Remark~\ref{rem:psic-psi}). 
The arguments to verify 
these two claims will occupy the rest of this section.

%%%%%%%%%%%%%%%%%%%%%%%%%%%%%%%%%%%%%%%%%%%%%%%%%%%%%%%%%
%
\subsection{Integral logarithmic homomorphism} \label{ssc:intlog}
%
%%%%%%%%%%%%%%%%%%%%%%%%%%%%%%%%%%%%%%%%%%%%%%%%%%%%%%%%%

We now introduce {\em the integral logarithmic homomorphisms}; 
for an arbitrary finite \p-group $\Delta$, Robert Oliver and Laurence
Robert Taylor defined a homomorphism of abelian groups (called 
the integral logarithm)
\begin{align*}
\Gamma_{\Delta} \colon K_1(\Z_p[\Delta]) \rightarrow
 \Z_p[\conj{\Delta}]; \, x \mapsto \log(x)-p^{-1} \varphi (\log(x))
\end{align*}
where $\varphi$ is ``the Frobenius correspondence'' on
$\Z_p[\conj{\Delta}]$ characterised by
\begin{align*}
\varphi \left(\sum_{[d]\in \conj{\Delta}} a_{[d]} [d] \right) =\sum_{[d] \in
\conj{\Delta}} a_{[d]} [d^p].
\end{align*} 
The integral logarithmic homomorphisms are compatible with
group homomorphisms; that is, the diagram
\begin{align} \label{eq:comphom}
\xymatrix{
K_1(\Z_p[\Delta]) \ar[r]^(0.475){\Gamma_{\Delta}} \ar[d]_(0.45){f_*} &
 \Z_p[\conj{\Delta}] \ar[d]^(0.45){f_*} \\
K_1(\Z_p[\Delta']) \ar[r]_(0.475){\Gamma_{\Delta'}} & \Z_p[\conj{\Delta'}]
}
\end{align}
commutes for an arbitrary homomorphism 
$f \colon \Delta \rightarrow \Delta'$ of finite $p$-groups 
(the symbol $f_*$ denotes 
the homomorphism of abelian groups induced by $f$). It is known that
the sequence
\begin{align} \label{eq:intlogstr}
1 \rightarrow K_1(\Z_p[\Delta])/K_1(\Z_p[\Delta])_{\mathrm{tors}}
 \xrightarrow{\Gamma_{\Delta}} \Z_p[\conj{\Delta}] \xrightarrow{\omega_{\Delta}}
 \Delta^{\mathrm{ab}} \rightarrow1
\end{align}
is exact where $\omega_{\Delta}$ is the homomorphism of abelian groups 
defined by 
\begin{align*}
\omega_{\Delta} \left( \sum_{[d]\in \conj{\Delta}} a_{[d]}[d] \right)=\prod_{[d] \in
\conj{\Delta}} \bar{d}^{a_{[d]}}
\end{align*}
(here we denote by $\bar{d}$ the image of $[d]$
in $\Delta^{\mathrm{ab}}$). Refer to \cite{Oliver, OT} 
for details of the properties of integral logarithms.

Now consider the case $G=G^f \times \Gamma$: let us apply the exact sequence (\ref{eq:intlogstr}) to the finite
\p-group $G^{(n)}=G^f \times \Gamma/\Gamma^{p^n}$ for each natural number
$n$. The structure of the torsion part of $K_1(\Z_p[G^{(n)}])$ has been
well studied in \cite[Section~4.4]{H}; in~fact, it is described as 
\begin{align} \label{eq:wall}
K_1(\Z_p[G^{(n)}])_{\mathrm{tors}}\cong \mu_{p-1}(\Z_p)\times
 G^{(n),\mathrm{ab}} \times SK_1(\Z_p[G^f])
\end{align}
by the theorem of Charles Terence Clegg Wall \cite[Theorem~4.1]{Wall} 
where $\mu_{p-1}(\Z_p)$ denotes the 
subgroup of $\Z_p^{\times}$ consisting of all ($p-1$)-th roots of unity. 
By taking the projective limit,\footnote{Since
$K_1(\Z_p[G^{(n+1)}])/K_1(\Z_p[G^{(n+1)}])_{\text{tors}} \rightarrow
K_1(\Z_p[G^{(n)}])/K_1(\Z_p[G^{(n)}])_{\text{tors}}$ is surjective, the
exact sequence (\ref{eq:intlogstr}) for the projective system with
respect to $\{ G^{(n)}\}_{n\in \N}$ satisfies the Mittag-Leffler
condition. Therefore we may take the projective limit.} 
we obtain the following exact sequence (note that 
the projective limit $\varprojlim_n K_1(\Z_p[G^{(n)}])$ actually
coincides with $K_1(\iw{G})$; see \cite[Proposition~1.5.1]{FK}):
\begin{align} \label{eq:basicseq}
1 \rightarrow K_1(\iw{G})/\varprojlim_n K_1(\Z_p[G^{(n)}])_{\mathrm{tors}} \xrightarrow{\Gamma_G} \Z_p[[\conj{G}]]
 \xrightarrow{\omega_G} G^{\mathrm{ab}} \rightarrow 1.
\end{align} 
Moreover (\ref{eq:wall}) implies that the projective limit  
$\varprojlim_n K_1(\Z_p[G^{(n)}])_{\mathrm{tors}}$ is 
isomorphic to the direct product 
$\mu_{p-1}(\Z_p) \times G^{\mathrm{ab}} \times SK_1(\Z_p[G^f])$.
We may, therefore, identify the \p-torsion part
 $K_1(\iw{G})_{p\text{-tors}}$ of the Whitehead group 
$K_1(\iw{G})$ with
$G^{f,\mathrm{ab}} \times SK_1(\Z_p[G^f])$ (recall that $SK_1(\Z_p[G^f])$ is a
finite \p-group \cite[Theorem~2.5]{Wall}). 

We remark that the \p-th power Frobenius
endomorphism $g\mapsto g^p$ is well defined on $G$ in our case 
since the exponent of $G^f$ equals $p$. We use the same symbol 
$\varphi$ for the Frobenius endomorphism on $G$, then it 
obviously induces the Frobenius correspondence on $\Z_p[[\conj{G}]]$. 
The notion $\varphi$ introduced here is compatible with the one 
defined in Section~\ref{sc:prelim}.

%%%%%%%%%%%%%%%%%%%%%%%%%%%%%%%%%%%%%%%%%%%%%%%%%%%%%%%%%%%%%
%
\subsection{The group $\widetilde{\Psi}_c$ contains the image of
  $\tilde{\theta}$} \label{ssc:contain}
%
%%%%%%%%%%%%%%%%%%%%%%%%%%%%%%%%%%%%%%%%%%%%%%%%%%%%%%%%%%%%%

In this subsection we prove that $\widetilde{\Psi}_c$ contains the image
of $\tilde{\theta}$ (and hence $\widetilde{\Psi}$ also does 
by Remark~\ref{rem:G}).

\begin{lem} \label{lem:xicontain}
The multiplicative group 
 $\widetilde{\Psi}'$ contains the image of $\tilde{\theta}$.
\end{lem}

\begin{proof}
 The element 
 $(\tilde{\theta}_{U,V}(\eta))_{(U,V)\in \mathfrak{F}_B}$ satisfies
 both (NCC) and (CCC) for each $\eta$ in $\pK_1(\iw{G})$ 
 by the basic properties of norm maps in algebraic \mbox{$K$-theory}. 
 Moreover the congruence $\tilde{\theta}_{U,V}(\eta) \equiv
 \varphi(\tilde{\theta}_{\mathrm{ab}}(\eta))^{(G:U)/p} \,
 \pmod{J_{U,V}}$ holds unless $U$ coincides with $G$ 
 by Proposition~\ref{prop:cong-j} 
 (we denote by $\tilde{\theta}_{\mathrm{ab}}$ the 
 homomorphism $\pK_1(\iw{G})\rightarrow \piw{G^{\mathrm{ab}}}^{\times}$ 
 induced by the abelisation map; 
 note that $\varphi(\tilde{\theta}_{\mathrm{ab}}(\eta))$ obviously coincides
 with $\varphi(\eta)$ by definition).
\end{proof}

By virtue of Lemma~\ref{lem:xicontain} we have only to verify 
the following proposition to show that
$\widetilde{\Psi}_c$ contains the image of $\tilde{\theta}$.

\begin{prop} \label{prop:normcong}
Let $\eta$ be an element in $\pK_1(\iw{G})$. Then the
 congruence $\tilde{\theta}_U(\eta) \equiv 
\varphi(\tilde{\theta}_{\mathrm{ab}}(\eta))^{(G:U)/p} \,
 \pmod{I_U}$ holds for each $U$ in $\mathfrak{F}_A^c$. 
\end{prop}

The following lemma relates 
norm maps in algebraic \mbox{$K$-theory} to trace homomorphisms 
defined in Section~\ref{ssc:calculation} via \p-adic logarithms.

\begin{lem}[compatibility lemma] \label{lem:lognorm}
Let $(U,V)$ and $(U',V')$ be elements in $\mathfrak{F}_B$ such
 that $U$ contains $U'$. Then the following diagram commutes$:$
\begin{align*}
\xymatrix{
K_1(\iw{U}) \ar[r]^(0.45){\log} \ar[d]_{\Nr_{\iw{U}/\iw{U'}}} &
 \Q_p[[\conj{U}]] \ar[d]^{\Tr_{\Q_p[[\conj{U}]]/\Q_p[[\conj{U'}]]}}  \\
K_1(\iw{U'}) \ar[r]_(0.45){\log} & \Q_p[[\conj{U'}]].
}
\end{align*}
\end{lem}

\begin{proof}
We may prove that the diagram commutes for each finite quotient
 $U^{(n)}=U^f \times \Gamma/\Gamma^{p^n}$ and ${U'}^{(n)}={U'}^f \times
 \Gamma/\Gamma^{p^n}$ by the same argument as that in
 \cite[Lemma~4.7]{H}. 
 Hence the claim holds by taking the projective limit.
\end{proof}

\begin{proof}[Proof of Proposition~$\ref{prop:normcong}$]
We may assume that $U$ does not coincide with $G$ without loss of generality 
(see Remark~\ref{rem:G}). 
Let $\theta_{\mathrm{ab}}$ (resp.\
 $\theta_{\mathrm{ab}}^+$) be the homomorphism 
 $K_1(\iw{G})\rightarrow \iw{G^{\mathrm{ab}}}^{\times}$ (resp.\
 $\Z_p[[\conj{G}]] \rightarrow \Z_p[[G^{\mathrm{ab}}]]$) induced by the
 abelisation map $G\rightarrow G^{\mathrm{ab}}$. Then we may easily check that the following
diagram commutes for each $(U,V)$ in $\mathfrak{F}_B$:

\begin{align} \label{eq:comm}
\begin{CD}
\Q_p[[\conj{G}]] @>{\qquad \Q_p \otimes_{\Z_p} \theta_{\mathrm{ab}}^+ \qquad}>>  \Q_p[[G^{\mathrm{ab}}]]  \\
@V{\frac{1}{p}\varphi}VV @VV{\frac{(G:U)}{p}\varphi}V \\
\Q_p[[\conj{G}]] @>>{\qquad \Q_p \otimes_{\Z_p} \theta_{U,V}^+ \qquad }> \Q_p[[U/V]]. \\
\end{CD}
\end{align}
Note that $\varphi(\tilde{\theta}_{\mathrm{ab}}(\eta))^{-(G:U)/p} 
\tilde{\theta}_U(\eta)$ is 
contained in $1+J_U^{\, \widetilde{\,}}$ for each $U$ in
 $\mathfrak{F}_A^c$ because $( \tilde{\theta}_{U,V}(\eta))_{(U,V)\in \mathfrak{F}_B}$ is an
 element in $\widetilde{\Psi}'$ (Lemma~\ref{lem:xicontain}).
 Then Proposition~\ref{prop:log-j} (3) asserts that the element 
 $\log (\varphi(\tilde{\theta}_{\mathrm{ab}}(\eta))^{-(G:U)/p}
\tilde{\theta}_U(\eta))$ is contained in $\iw{U}$. 
On the other hand, we may calculate as
\begin{equation} \label{eq:logtheta}
\begin{aligned}
\theta^+_{U,V} \circ \Gamma_G (\eta) &= (\Q_p \otimes_{\Z_p} \theta^+_{U,V})(\log (\eta))
 -(\Q_p \otimes_{\Z_p} \theta^+_{U,V})(p^{-1}\varphi(\log(\eta))) \\
 &=\log(\theta_{U,V}(\eta))-\frac{(G:U)}{p}\varphi( \log
 (\theta_{\mathrm{ab}}(\eta))) \\
&= \log 
\frac{\theta_{U,V}(\eta)}{\varphi(\theta_{\mathrm{ab}}(\eta))^{(G:U)/p}}
\end{aligned} 
\end{equation}
for each $(U,V)$ in $\mathfrak{F}_B$ (the first equality is 
 nothing but the
 definition of the integral logarithm and the second follows
 from~Lemma~\ref{lem:lognorm} 
and (\ref{eq:comm})).\footnote{For the abelisation
$\theta_{\mathrm{ab}}=\theta_{G,[G,G]}$, 
we use the notation 
$\log(\varphi(\eta)^{-1/p} \theta_{\mathrm{ab}}(\eta))$ 
for an element defined as 
$\Gamma_{G_{\mathrm{ab}}}(\theta_{\mathrm{ab}}(\eta))=
\log(\theta_{\mathrm{ab}}(\eta))-p^{-1}\log(\varphi(\theta_{\mathrm{ab}}(\eta)))$ 
by abuse of notation.} 
In particular
 $\log(\varphi(\theta_{\mathrm{ab}}(\eta))^{-(G:U)/p}\theta_U(\eta))$ 
is contained in $I_U$
 for each $U$ in $\mathfrak{F}_A^c$ by definition. Recall that 
 for each $U$ in $\mathfrak{F}_A^c$ 
the \p-adic logarithm is injective on $1+J_U^{\, \widetilde{\,}}$  
 (Proposition~\ref{prop:log-j}) and
it induces an isomorphism between $1+I_U^{\, \widetilde{\,}}$ and $I_U$
 (Proposition~\ref{prop:log-i} (3)) unless $U$ coincides with $G$. Therefore
we may conclude that 
$\varphi(\tilde{\theta}_{\mathrm{ab}}(\eta))^{-(G:U)/p}\tilde{\theta}_U(\eta)$ 
is contained in $1+I_U^{\,\widetilde{\,}}$, which implies the desired
 additional congruence for $U$.
\end{proof} 

By Lemma~\ref{lem:xicontain} and Proposition~\ref{prop:normcong}, we may
conclude that $\widetilde{\Psi}$ (resp.\ $\widetilde{\Psi}_c$) contains
the image of $\tilde{\theta}$; in other words, $\tilde{\theta}$ induces
a homomorphism 
\begin{align*}
\tilde{\theta} \colon \pK_1(\iw{G}) \rightarrow \widetilde{\Psi} \quad (\text{resp.\
 $\widetilde{\Psi}_c$}).
\end{align*}

%%%%%%%%%%%%%%%%%%%%%%%%%%%%%%%%%%%%%%%%%%%%%%%%%%%%%%%%%
%
\subsection{Proof of the isomorphy of $\tilde{\theta}$} \label{ssc:isomorphy}
%
%%%%%%%%%%%%%%%%%%%%%%%%%%%%%%%%%%%%%%%%%%%%%%%%%%%%%%%%%

We shall verify the isomorphy of $\tilde{\theta}$ in this subsection.

\begin{prop} \label{prop:injection}
The homomorphism $\tilde{\theta} \colon \pK_1(\iw{G}) \rightarrow
 \widetilde{\Psi}_c$ is injective.
\end{prop}

\begin{proof} 
Take an arbitrary element from the kernel of $\tilde{\theta}$ and 
let $\eta$ denote its lift to $K_1(\iw{G})$. Then 
 $\theta_{U,V}^+\circ \Gamma_G(\eta)$ vanishes for each $(U,V)$ in
 $\mathfrak{F}_B$ by~(\ref{eq:logtheta}). 
Hence $\Gamma_G(\eta)$ coincides with zero since
 $\theta^+_B$ is injective (Proposition~\ref{prop:addtheta-B}); 
 equivalently the element $\eta$ is contained in the kernel of 
 the integral logarithm~$\Gamma_G$. Combining this fact with Wall's theorem
 (see \cite[Theorem~4.1]{Wall} and (\ref{eq:wall})), 
 we may regard $\eta$ 
 as an element in $\mu_{p-1}(\Z_p)\times G^{\mathrm{ab}} \times SK_1(\Z_p[G^f])$. 
 Furthermore the abelisation map $\theta_{\mathrm{ab}}$ induces 
 the canonical projection from \mbox{$\mu_{p-1}(\Z_p)\times
 G^{\mathrm{ab}} \times SK_1(\Z_p[G^f])$} onto $\mu_{p-1}(\Z_p)
 \times G^{\mathrm{ab}}$ when it is restricted to the kernel of
 $\Gamma_G$. Since $\tilde{\theta}_{\mathrm{ab}}(\eta)$
 vanishes by assumption, the element $\eta$ is contained in
 $G^{\mathrm{ab},f} \times SK_1(\Z_p[G^f])$, 
 and in particular $\eta$ is a p-torsion element. 
This implies that the image of $\eta$ in
 $\pK_1(\iw{G})$ reduces to be trivial.
\end{proof}

\begin{prop} \label{prop:surjection}
The homomorphism $\tilde{\theta} \colon \pK_1(\iw{G}) \rightarrow
 \widetilde{\Psi}$ is surjective.
\end{prop}

Let $\eta_{\bullet}$ be an element in $\widetilde{\Psi}$. Since
$\eta_{\bullet}$ is in particular contained in $\widetilde{\Psi}'$, 
the element  
$\log (\varphi(\eta_{\mathrm{ab}})^{-(G:U)/p} \eta_{U,V})$ can be defined as
an element in $\iw{U/V}$ 
for each $(U,V)$ in $\mathfrak{F}_B$
(Proposition~\ref{prop:log-j} (2) and the definition of the 
integral logarithm for $G^{\mathrm{ab}}$).

\begin{lem} \label{lem:logcontained} 
The element $(\log(\varphi(\eta_{\mathrm{ab}})^{-(G:U)/p} \eta_{U,V}))_{(U,V)\in
 \mathfrak{F}_B}$ is contained in~$\Phi_B$. Moreover $(\log
 (\varphi(\eta_{\mathrm{ab}})^{-(G:U)/p} \eta_U))_{U\in \mathfrak{F}_A}$
 is contained in $\Phi$.
\end{lem}

\begin{proof}
Set $y_{U,V}=\log (\varphi(\eta_{\mathrm{ab}})^{-(G:U)/p} \eta_{U,V})$
 for each $(U,V)$ in $\mathfrak{F}_B$. Then we may easily verify that 
$(y_{U,V})_{(U,V)\in
 \mathfrak{F}_B}$ satisfies both (TCC) and (CCC$+$) (due to (NCC), (CCC) and
 Lemma~\ref{lem:lognorm}). Hence $(y_{U,V})_{(U,V)\in
 \mathfrak{F}_B}$  is contained in $\Phi_B$. Moreover
 $\varphi(\eta_{\mathrm{ab}})^{-(G:U)/p} \eta_U$ is contained in
 $1+I_U^{\, \widetilde{\,}}$ for each $U$ in~$\mathfrak{F}_A$ by 
 additional congruence condition, and thus
 $y_U=\log(\varphi(\eta_{\mathrm{ab}})^{-(G:U)/p} \eta_U)$ is contained in
 $I_U$ by~Proposition~\ref{prop:log-i}. This implies that $(y_U)_{U\in
 \mathfrak{F}_A}$ is an element in $\Phi$.
\end{proof}

\begin{proof}[Proof of Proposition~$\ref{prop:surjection}$]
First note that there exists a unique element 
 $y$ in $\Z_p[[\conj{G}]]$ which satisfies
 $\theta_B^+(y)=(\log(\varphi(\eta_{\mathrm{ab}})^{-(G:U)/p}\eta_{U,V}))_{(U,V)\in \mathfrak{F}_B}$
 by~Proposition~\ref{prop:addtheta-B} and Lemma~\ref{lem:logcontained}. 
In particular the equation
\begin{equation} \label{eq:abelisation}
 \theta^+_{\mathrm{ab}} (y) =\log \eta_{\mathrm{ab}}-\frac{1}{p}\varphi
 (\log \eta_{\mathrm{ab}})=\Gamma_{G_{\mathrm{ab}}}(\eta_{\mathrm{ab}})
\end{equation}
holds. Then we may calculate as
\begin{align*}
\omega_{G}(y)=\omega_{G^{\mathrm{ab}}} \circ \theta^+_{\mathrm{ab}}
 (y) 
 =\omega_{G^{\mathrm{ab}}} \circ \Gamma_{G^{\mathrm{ab}}}
 (\eta_{\mathrm{ab}}) =1
\end{align*}
where the first equality directly follows from the definition of
 $\omega_{G}$ and $\omega_{G^{\mathrm{ab}}}$ (see
 Section~\ref{ssc:intlog}), the second follows from
 (\ref{eq:abelisation}) and the last follows from
 (\ref{eq:basicseq}). The sequence (\ref{eq:basicseq}) also asserts that 
 there exists an
 element $\eta'$ in $K_1(\iw{G})$ which satisfies $\Gamma_G(\eta')=y$.
Furthermore we obtain
\begin{align*}
\Gamma_{G^{\mathrm{ab}}} (\tilde{\theta}_{\mathrm{ab}}(\eta')) =
 \theta^+_{\mathrm{ab}} \circ \Gamma_G(\eta')=\theta^+_{\mathrm{ab}}(y)
 =\Gamma_{G^{\mathrm{ab}}} (y_{\mathrm{ab}})
\end{align*}
by using (\ref{eq:abelisation}).
Since the kernel of $\Gamma_{G^{\mathrm{ab}}}$ is identified with 
 $\mu_{p-1}(\Z_p)
 \times G^{\mathrm{ab}}$ by the theorem of Graham Higman \cite{Higman}, 
there exists an element $\tau$ in $\mu_{p-1}(\Z_p)\times
 G^{\mathrm{ab}}$ such that the
 equation 
 $\tilde{\theta}_{\mathrm{ab}}(\eta')\tau=\eta_{\mathrm{ab}}$ holds. 
 Set $\eta=\eta'\tau$. By construction, the abelisation 
 $\tilde{\theta}_{\mathrm{ab}}(\eta)$ of~$\eta$ coincides with
 $\eta_{\mathrm{ab}}$ and 
\begin{align*}
\log \frac{\eta_{U,V}}{\varphi(\eta_{\mathrm{ab}})^{(G:U)/p}} =
 \theta_{U,V}^+(y)  
= \theta_{U,V}^+ \circ\Gamma_{G}(\eta) =\log
 \frac{\tilde{\theta}_{U,V}(\eta)}{\varphi(\tilde{\theta}_{\mathrm{ab}}(\eta))^{(G:U)/p}}
\end{align*}
holds for each $(U,V)$ in $\mathfrak{F}_B$ except for $(G,[G,G])$ 
(the first equality is due to the construction of $y$ and the last 
is due to (\ref{eq:logtheta})). Then
 $\tilde{\theta}_{U,V}(\eta)$ coincides with $\eta_{U,V}$ because the
 \p-adic logarithm induces an injection on $1+J_{U,V}^{\,
 \widetilde{\,}}$ (Proposition~\ref{prop:log-j}); 
 in other words the image of $\eta$ under the map $\tilde{\theta}$ 
 coincides with $\eta_{\bullet}$, which asserts that 
$\tilde{\theta}\colon \pK_1(\iw{G})\rightarrow \widetilde{\Psi}$ is surjective. 
\end{proof}

%%%%%%%%%%%%%%%%%%%%%%%%%%%%%%%%%%%%%%%%%%%%%%%%%%%%%%%%%
%
%
\section{Localized version} \label{sc:localise}
%
%
%%%%%%%%%%%%%%%%%%%%%%%%%%%%%%%%%%%%%%%%%%%%%%%%%%%%%%%%%

In this section we study ``the localised theta map;'' 
more precisely, let $\theta_{S,U,V}$ be the 
composition of the norm map $\Nr_{\iw{G}_S/\iw{U}_S}$ 
with the canonical homomorphism 
$K_1(\iw{U}_S) \rightarrow \iw{U/V}_S^{\times}$ 
for each $(U,V)$ in $\mathfrak{F}_B$ and 
set $\theta_S=(\theta_{S,U,V})_{(U,V)\in \mathfrak{F}_B}$. It is obvious that 
$\theta_{S}$ induces a group homomorphism $\tilde{\theta}_{S} \colon \pK_1(\iw{G}_S) \rightarrow \prod_{(U,V)\in \mathfrak{F}_B} 
\piw{U/V}_S^{\times}$. We shall study the image of $\tilde{\theta}_S$.

Let $\iw{\Gamma}_{(p)}$ denote the localisation of 
the Iwasawa algebra $\iw{\Gamma}$ with respect to 
the prime ideal $p\iw{\Gamma}$, and let $R$ denote its 
\p-adic completion $\iw{\Gamma}_{(p)}^{\, \widehat{\,}}$
for simplicity. 
We remark that for each finite \p-group $\Delta$, 
the localised 
Iwasawa algebra $\iw{\Delta\times \Gamma}_S$ is identified with the group ring 
$\iw{\Gamma}_{(p)}[\Delta]$ under the identification 
$\iw{\Delta\times \Gamma}\cong \iw{\Gamma}[\Delta]$ 
(see \cite[Lemma~2.1]{CFKSV}). 
Now for each $(U,V)$ in $\mathfrak{F}_B$, let $J_{S,U,V}$ (resp.\ $J_{U,V}^{\, \widehat{\,}}$) be the kernel of the composition
\begin{align*}
& \iw{U/V}_S \xrightarrow{\text{augmentation}} \iw{\Gamma}_{(p)} \rightarrow \iw{\Gamma}_{(p)}/p\iw{\Gamma}_{(p)} \\
(\text{resp. } & R[U^f/V^f] \xrightarrow{\text{augmentation}} R 
\xrightarrow{\hspace*{1.5cm}} R/pR).
\end{align*} 
Then we may easily verify that the intersection of $J_{U,V}^{\, \widehat{\,}}$ 
and $\iw{U/V}_S$ (resp.\ $J_{S,U,V}$ and $\iw{U/V}$) 
coincides with $J_{S,U,V}$ (resp.\ $J_{U,V}$) 
under the identification $\iw{U/V}_S \cong
\iw{\Gamma}_{(p)}[U^f/V^f]$. 
Since the group ring $R[U^f/V^f]$ is \mbox{\p-adically} complete, 
the \p-adic logarithm converges on 
$1+J_{U,V}^{\, \widehat{\,}}$ and 
induces an injection
$\log \colon (1+J_{U,V}^{\, \widehat{\,}})^{\, \widetilde{\,}} 
\rightarrow R[U^f/V^f]$ unless $U$ coincides with $G$ 
(similarly to Proposition~\ref{prop:log-j}).
Let $\widetilde{\Psi}'_S$ be the subgroup 
of the direct product 
$\prod_{(U,V)\in \mathfrak{F}_B} \piw{U/V}_S^{\times}$ 
consisting of all elements 
$\eta_{S,\bullet}$ satisfying norm compatibility condition (NCC)${}_S$,
conjugacy compatibility condition (CCC)${}_S$ and 
the following congruence for each 
$(U,V)$ in $\mathfrak{F}_B$ 
except for $(G,[G,G])$:\footnote{We may naturally extend both (NCC) and (CCC) 
to the localised versions (NCC)${}_S$ and (CCC)${}_S$ in an obvious manner.}
\begin{align*}
\eta_{S,U,V} \equiv \varphi(\eta_{S,\mathrm{ab}})^{(G:U)/p} \, \pmod{J_{S,U,V}}.
\end{align*}
Let $\widetilde{\Psi}_S$ (resp.\ $\widetilde{\Psi}_{S,c}$) be 
the subgroup of $\widetilde{\Psi}'_S$ consisting of all elements
$\eta_{S,\bullet}$ satisfying the following additional 
congruence condition:
\begin{quotation}
(additional congruence condition)

\noindent the congruence $\eta_{S,U}\equiv \varphi(\eta_{S,\mathrm{ab}})^{(G:U)/p} \, \pmod{I_{S,U}}$ holds for each $U$ in $\mathfrak{F}_A$ 
(resp.\ $\mathfrak{F}_A^c$) where $I_{S,U}$ is the
 $\iw{\Gamma}_{(p)}$-module 
defined as $I_U \otimes_{\iw{\Gamma}} \iw{\Gamma}_{(p)}$.
\end{quotation}
The group $\widetilde{\Psi}_{S,c}$ is a subgroup of $\widetilde{\Psi}_S$
(as $\widetilde{\Psi}_c$ is that of $\widetilde{\Psi}$; see also 
Remark~\ref{rem:G}).

\begin{lem} \label{lem:intersection} 
The intersection of $I_U^{\, \widehat{\,}}$ and $\iw{U}_S$ $($resp.\
 $I_{S,U}$ and $\iw{U})$ coincides with $I_{S,U}$ $($resp.\ $I_U)$.
\end{lem}

\begin{proof}
We shall only prove the claim $I_{S,U}\cap \iw{U}=I_U$ (the other
 one is verified by much simpler calculation). The $\Z_p$-module $I_U$
 is obviously contained in the intersection $I_{S,U}\cap \iw{U}$ by
 construction. Note that $I_{S,U}$ is a free
 $\iw{\Gamma}_{(p)}$-submodule of $\iw{U}_S$ each of whose generators is
 obtained as finite sum of $\{p^j u\}_{0\leq j\leq N, u\in U^f}$ 
 (see the explicit description of $I_U$ given in
 Section~\ref{ssc:calculation}). 
 Hence an arbitrary element in $I_{S,U}\cap \iw{U}$ is
 described as a 
 $\iw{\Gamma}_{(p)}\cap \iw{\Gamma}[p^{-1}]$-linear
 combination of generators of $I_{S,U}$, which implies that 
 the intersection $I_{S,U}\cap \iw{U}$ is contained in $I_U$ (observe
 that $\iw{\Gamma}_{(p)}\cap \iw{\Gamma}[p^{-1}]$
 coincides with $\iw{\Gamma}$ and generators of $I_{S,U}$ over
 $\iw{\Gamma}_{(p)}$ coincides with those of $I_U$ over $\iw{\Gamma}$).
\end{proof}

\begin{prop} \label{prop:thetaloc}
Both $\widetilde{\Psi}_S$ and 
$\widetilde{\Psi}_{S,c}$ contain the image of $\tilde{\theta}_S$.
\end{prop}

\begin{proof}[Sketch of the proof]
Let $\eta_S$ be an arbitrary element in $\pK_1(\iw{G}_S)$. 
By the same argument as that in the proof of Lemma~\ref{lem:xicontain}, 
we may verify that $\widetilde{\Psi}'_S$ 
contains the image of $\tilde{\theta}_S$. 
Then  the element $\varphi(\theta_{S,\mathrm{ab}}(\eta_S))^{-(G:U)/p}
 \theta_{S,U,V}(\eta_S)$ (which we denote by $\eta'_{S,U,V}$ in the
 following) 
is contained in $1+J_{S,U,V}$ for $(U,V)$ in $\mathfrak{F}_B$ 
except for $(G,[G,G])$ by congruence condition, 
and it is regarded as an element in $1+J_{U,V}^{\, \widehat{\,}}$ 
in a natural way.
Hence we may define $\log \eta'_{S,U,V}$ as an element in 
$R[U^f/V^f]$. 
On the other hand we may easily show that for each $U$ in $\mathfrak{F}_A^c$ 
the image of the trace map $\Tr_{R[\conj{G^f}]/R[U^f]}$ coincides 
with the $R$-module $I_U^{\, \widehat{\,}}$ defined as 
$I_U \otimes_{\iw{\Gamma}} R$ 
(here we assume that $U$ does not coincide with $G$; see Remark~\ref{rem:G}). 
Moreover the image of $\eta_S$ under the composite map
$\Tr_{R[\conj{G^f}]/R[U^f]}\circ \Gamma_{R,G^f}$ is 
calculated as $\log \eta'_{S,U}$ 
by calculation similar to (\ref{eq:logtheta}) where 
$\Gamma_{R,G^f}$ is the integral logarithm 
$K_1(R[G^f])\rightarrow R[\conj{G^f}]$ 
with coefficient in $R$ (see \cite[Section~1.1 and Remark~5.2]{H}). 
This implies that $\log \eta'_{S,U}$ is  
contained in $I_U^{\, \widehat{\,}}$ for each $U$ in~$\mathfrak{F}_A^c$. 
Then we obtain the congruence $\tilde{\theta}_{S,U}(\eta_S)\equiv
 \varphi(\tilde{\theta}_{S,\mathrm{ab}}(\eta_S))^{(G:U)/p} \, \pmod{I_{S,U}}$
 by the logarithmic isomorphism 
$1+I_U^{\, \widehat{\,}} \xrightarrow{\sim} I_U^{\, \widehat{\,}}$ 
(readily verified in the same manner as Proposition~\ref{prop:log-i}) 
and  the relation
$I_U^{\, \widehat{\,}}\cap \iw{U}_S=I_{S,U}$
 (Lemma~\ref{lem:intersection}). 
 Consequently
 $(\tilde{\theta}_{S,U,V}(\eta_S))_{(U,V)\in \mathfrak{F}_B}$ is contained
 in $\widetilde{\Psi}_{S,c}$ (and hence in $\widetilde{\Psi}_S$).
\end{proof}

\begin{prop} \label{prop:intersection}
The intersection of $\widetilde{\Psi}_S$ $($resp.\ $\widetilde{\Psi}_{S,c})$ 
and the direct product $\prod_{(U,V)\in \mathfrak{F}_B} \piw{U/V}^{\times}$ 
coincides with $\widetilde{\Psi}$ $(=\widetilde{\Psi}_c)$. 
\end{prop}

\begin{proof}
Use the relations $I_{S,U}\cap \iw{U}=I_U$ for each $U$ 
in $\mathfrak{F}_A^c$ (Lemma~\ref{lem:intersection}) 
and $J_{S,U,V}\cap \iw{U/V}=J_{U,V}$ for each 
$(U,V)$ in $\mathfrak{F}_B$.
\end{proof}

%%%%%%%%%%%%%%%%%%%%%%%%%%%%%%%%%%%%%%%%%%%%%%%%%%%%%%%%%
%
%
\section{Weak congruences upon abelian $p$-adic zeta functions} \label{sc:cong}
%
%
%%%%%%%%%%%%%%%%%%%%%%%%%%%%%%%%%%%%%%%%%%%%%%%%%%%%%%%%%

In this section we study properties of the \p-adic zeta pseudomeasures
for extensions corresponding to certain abelian subquotients of $G$, especially 
{\em congruences} which they satisfy. 
In the rest of this article, {\em we fix embeddings 
$\overline{\Q} \hookrightarrow \C$ and 
$\overline{\Q} \hookrightarrow \overline{\Q}_p$}.

%%%%%%%%%%%%%%%%%%%%%%%%%%%%%%%%%%%%%%%%%%%%%%%%%%%%%%%%%
%
\subsection{Weak Congruences} \label{ssc:wcong}
%
%%%%%%%%%%%%%%%%%%%%%%%%%%%%%%%%%%%%%%%%%%%%%%%%%%%%%%%%%

For each $(U,V)$ in $\mathfrak{F}_B$, let $\xi_{U,V}$ denote 
Serre's \p-adic zeta pseudomeasure for the abelian extension $F_V/F_U$ 
(which is an element in $\iw{U/V}_S^{\times}$). 

\begin{lem} \label{lem:ncc-ccc}
The element $(\xi_{U,V})_{(U,V)\in \mathfrak{F}_B}$ in $\prod_{(U,V)\in
 \mathfrak{F}_B} \piw{U/V}_S^{\times}$ satisfies both norm compatibility
 condition $(\mathrm{NCC})$ and conjugacy compatibility 
 condition $(\mathrm{CCC})$.
\end{lem}

\begin{proof}
Let $(U,V)$ and $(U',V')$ 
be elements in $\mathfrak{F}_B$ such that $U$ contains $U'$ and $U'$
 contains $V$ respectively. Then we may easily verify that
\begin{align*}
\Nr_{\iw{U/V}_S/\iw{U'/V}_S}(f)(\rho)=f(\ind{U}{U'}{\rho})
\end{align*} 
holds for an arbitrary element $f$ in $\iw{U/V}_S^{\times}$ and 
an arbitrary continuous \p-adic character $\rho$ of the abelian group $U'/V$ 
(due to  the definition of the evaluation map). 
Hence for an arbitrary finite-order character $\chi$ of $U'/V$ and 
an arbitrary natural number $r$ divisible by $p-1$, the following
 equation holds by the interpolation property (\ref{eq:interpab}) 
of $\xi_{U,V}$:
\begin{align*}
\Nr_{\iw{U/V}_S/\iw{U'/V}_S}(\xi_{U,V})(\chi \kappa^r) 
&= \xi_{U,V}(\ind{U}{U'}{\chi \kappa^r}) \\
&= L_{\Sigma} (1-r; F_V/F_U , \ind{U}{U'}{\chi}) \\
&= L_{\Sigma} (1-r; F_V/F_{U'}, \chi) 
=\xi_{U',V}(\chi\kappa^r).
\end{align*}
Then uniqueness of the abelian \p-adic zeta pseudomeasures 
 for $F_V/F_{U'}$ asserts 
 that the norm image $\Nr_{\iw{U/V}_S/\iw{U'/V}_S}(\xi_{U,V})$ 
 of $\xi_{U,V}$ coincides
 with $\xi_{U',V}$. The equation $\can^{V'}_V(\xi_{U',V'})=\xi_{U',V}$
 is also verified straightforwardly, 
 and therefore $(\xi_{U,V})_{(U,V)\in \mathfrak{F}_B}$ 
 satisfies (NCC). By a similar formal argument 
 we may also prove that 
 $(\xi_{U,V})_{(U,V)\in \mathfrak{F}_B}$ satisfies (CCC), 
 but we omit the details.
\end{proof}

Therefore if $(\xi_{U,V})_{(U,V) \in \mathfrak{F}_B}$ satisfies both 
congruence condition and additional congruence condition, we may
conclude that $(\xi_{U,V})_{(U,V)\in \mathfrak{F}_B}$ is contained in
$\widetilde{\Psi}_{S,c}$ (hence also in $\widetilde{\Psi}_S$). 
It is, however, difficult to prove the desired congruences 
for $\{\xi_{U,V}\}_{(U,V)\in \mathfrak{F}_B}$ directly. 
In the rest of this section we shall prove the following {\em weak congruences} 
by using Deligne-Ribet's theory upon Hilbert modular forms \cite{DR} 
(especially using {\em the \q-expansion principle}).

\begin{prop}[weak congruences] \label{prop:congzeta}
Let $(U,V)$ be an element in $\mathfrak{F}_B$ such that $U$ does not coincide 
with $G$, then there exists an element $c_{U,V}$
 in $\piw{\Gamma}_{(p)}^{\times}$ and the congruence 
\begin{align} \label{eq:congzetaj}
\xi_{U,V} \equiv c_{U,V} \quad \pmod{J_{S,U,V}}
\end{align}
holds. If $U$ is an element in $\mathfrak{F}_A^c$, the congruence 
\begin{align} \label{eq:congzetai}
\xi_U \equiv c_U \quad \pmod{I_{S,U}'}
\end{align}
also holds where $I_U'$ is the image of the trace map from
 $\Z_p[[\conj{NU}]]$ to $\Z_p[[U]]$ 
and $I_{S,U}'$ is its scalar extension 
 $I'_U \otimes_{\iw{\Gamma}}\iw{\Gamma}_{(p)}$.
\end{prop}

\begin{rem}
We may obtain the explicit description of each $I'_{U}$ by  
 calculation similar to that in Section~\ref{ssc:calculation} as
 follows:
\begin{align*}
I'_{\Gamma}&= p^N \Z_p[[\Gamma]],  \\
I'_{U_h} &=p^{n_h-1} \Z_p[[U_h]] \qquad \qquad \quad
 \text{for $h$ in }  \mathfrak{H} \setminus \{e\},  \\
I'_{U_{h,c}} &= p^{n_h-2} \Z_p[[U_{h,c}]] \qquad \quad 
 \text{for $h$ in } \mathfrak{H}\setminus \{e,c\}  \text{ satisfying (Case-1)},  \\
I'_{U_{h,c}} &= p^{n_h-1} \Z_p[[U_c]] \oplus
 \bigoplus_{i=1}^{p-1} p^{n_h-2} h^i(1+c^2 +\cdots +c^{p-1})
 \Z_p[[\Gamma]] \\
 & \qquad \qquad \qquad \qquad \qquad \quad
  \text{for $h$ in }  \mathfrak{H}\setminus \{e,c\}  \text{ satisfying (Case-2)}.  
\end{align*}
Each $I_U'$ (resp.\ $I_{S,U}'$) obviously contains $I_U$ (resp.\ $I_{S,U}$). 
Moreover the \p-adic logarithm induces an
 isomorphism between $1+I_U'$ and $I_U'$ 
 (resp.\ between $1+(I_U')^{\, \widehat{\,}}$ and 
 $(I_U')^{ \, \widehat{\,}}$ where 
 $(I_U')^{ \, \widehat{\,}}$ is defined as 
 $I'_U \otimes_{\iw{\Gamma}}R$) by the same
 argument as that in the proof of Proposition~\ref{prop:log-i}.
\end{rem}

%%%%%%%%%%%%%%%%%%%%%%%%%%%%%%%%%%%%%%%%%%%%%%%%%%%%%%%%%
%
\subsection{Ritter-Weiss' approximation technique} \label{ssc:ritterweiss}
%
%%%%%%%%%%%%%%%%%%%%%%%%%%%%%%%%%%%%%%%%%%%%%%%%%%%%%%%%%

In \cite{RW6}, J\"urgen Ritter and Alfred Weiss approximated the \p-adic 
zeta pseudomeasure by using special values of partial zeta functions,
and derived certain congruences among \p-adic zeta pseudomeasures. 
In this section we shall derive sufficient 
condition for Proposition~\ref{prop:congzeta} to hold by applying their 
approximation technique. 
Fix an element $(U,V)$ in $\mathfrak{F}_B$ 
such that $U$ does not coincide with $G$, and 
let $W$ denote the quotient group $U/V$ 
(which is abelian by definition). 
For an arbitrary open subgroup $\mathcal{U}$ of~$W$, we define the 
natural number  $m(\mathcal{U})$ 
by $\kappa^{p-1}(\mathcal{U})=1+p^{m(\mathcal{U})}\Z_p$
where $\kappa$ is the \p-adic cyclotomic character. 
Then we obtain an isomorphism
\begin{align} \label{eq:projlim}
 \Z_p[[W]] \xrightarrow{\sim} \varprojlim_{\mathcal{U} \trianglelefteq W \colon \text{open}}\Z_p[W/\mathcal{U}]/p^{m(\mathcal{U})}\Z_p[W/\mathcal{U}]
\end{align}
(see \cite[Lemma~1]{RW6} for details). 

\begin{defn}[partial zeta function]
Let $\varepsilon$ be a $\C$-valued locally constant function on $W$. If
 $\varepsilon$ is constant on an open subgroup $\mathcal{U}$ of $W$, we
 may identify $\varepsilon$ with a $\C$-linear combination 
 $\sum_{x\in W/\mathcal{U}}
 \varepsilon(x)\delta^{(x)}$ where $\delta^{(x)}$ is ``the Dirac delta
 function at~$x$'' (that is, $\delta^{(x)}(w)$ equals $1$ if $w$ is in the
 coset $x$ and $0$ otherwise). Then we define 
{\em the $(\Sigma$-truncated$)$ partial zeta
 function $\zeta^{\Sigma}_{F_V/F_U}(s,\varepsilon)$ 
 for $F_V/F_U$ with respect to the locally constant function
 $\varepsilon$} as 
 $\sum_{x\in
 W/\mathcal{U}} \varepsilon(x)
 \zeta^{\Sigma}_{F_V/F_U}(s,\delta^{(x)})$ 
 where $\zeta^{\Sigma}_{F_V/F_U}(s,\delta^{(x)})$ is defined as the Dirichlet 
series
\begin{align*}
\zeta^{\Sigma}_{F_V/F_U}(s,\delta^{(x)})=\sum_{0\neq \mathfrak{a}
 \subseteq \mathcal{O}_{F_U}\colon  \text{integral ideal prime to }\Sigma} 
\dfrac{\delta^{(x)} ( ( F_V/F_U , \mathfrak{a} ))}{(\mathcal{N}
 \mathfrak{a})^s}
\end{align*}
$($the symbol $(F_V/F_U, -)$ denotes 
 the Artin symbol for the abelian extension
 $F_V/F_U$ and $\mathcal{N}\mathfrak{a}$ denotes the absolute norm of the
 ideal $\mathfrak{a})$. It is meromorphically continued to the whole 
 complex plane $\C$.
\end{defn}

For an arbitrary natural number $k$ divisible by $p-1$ and an
arbitrary element $w$ in $W$, set
\begin{align*}
\Delta_{F_V/F_U}^w(1-k, \varepsilon)=\zeta^{\Sigma}_{F_V/F_U}(1-k, \varepsilon)-
 \kappa(w)^k\zeta^{\Sigma}_{F_V/F_U}(1-k, \varepsilon_w)
\end{align*}
which is a \p-adic rational number due to the results of 
 Helmut Klingen and Carl Ludwig Siegel \cite{Klingen, Siegel} 
(we denote by $\varepsilon_w$ the
function defined by $\varepsilon_w(w')=\varepsilon(ww')$).

\begin{prop}[approximation lemma, Ritter-Weiss] \label{prop:approximation}
Let $\mathcal{U}$ be an arbitrary open normal subgroup of $W$. 
Then for each $k$ divisible by $p-1$ and each $w$ in $W$, 
the image of the element $(1-w)\xi_{U,V}$ 
 under the canonical surjection
 $\Z_p[[W]] \rightarrow
 \Z_p[W/\mathcal{U}]/p^{m(\mathcal{U})}\Z_p[W/\mathcal{U}]$ is described as
\begin{align*}
\sum_{x\in W/\mathcal{U}} \Delta_{F_V/F_U}^w(1-k, \delta^{(x)})
 \kappa(x)^{-k} x \qquad \mod{p^{m(\mathcal{U})}}.
\end{align*}
\end{prop}

\begin{proof}
See \cite[Proposition~2]{RW6}.
\end{proof}

Let $j$ be a natural number and $NU$ the normaliser of $U$. 
Then the quotient group $NU/U$ acts upon $W/\Gamma^{p^j}$ by conjugation 
(recall that $\Gamma^{p^j}$ is abelian). 
For each coset $y$ of $W/\Gamma^{p^j}$, let $(NU/U)_y$
denote the isotropy subgroup of $NU/U$ at $y$ under this action. 

\begin{prop}[sufficient condition] \label{prop:sufficient}
Let $(U,V)$ be an element in $\mathfrak{F}_B$ except for $(G,[G,G])$. 
Then the congruence
 $(\ref{eq:congzetaj})$ holds if the congruence
\begin{align} \label{eq:suffzeta}
 \Delta_{F_V/F_U}^w (1-k, \delta^{(y)}) \equiv 0 \qquad \mod \sharp(NU/U)_y\Z_p 
\end{align}
holds for an arbitrary element $w$ in $\Gamma$, an arbitrary natural
 number 
 $k$ divisible by~$p-1$ and an arbitrary coset $y$ 
 of $W/\Gamma^{p^j}$ not contained in $\Gamma$. 
If $U=(U,\{e\})$ is an element in $\mathfrak{F}_A^c$, the congruence
 $(\ref{eq:suffzeta})$ also gives sufficient condition for the 
congruence $(\ref{eq:congzetai})$ to hold. 
\end{prop}

\begin{proof}
Apply the approximation lemma (Proposition~\ref{prop:approximation}) to
 the element $(1-w) \xi_{U,V}$, and then its image under 
the canonical surjection from $\Z_p[[W]]$ onto 
$\Z_p[W/\Gamma^{p^j}]/p^{m(\Gamma^{p^j})} \Z_p[W/\Gamma^{p^j}]$ is described as
\begin{align} \label{eq:orbit}
\sum_{y\in W/\Gamma^{p^j}} \Delta_{F_V/F_U}^w(1-k, \delta^{(y)})\kappa(y)^{-k}
 y && \mod p^{m(\Gamma^{p^j})}
\end{align} 
for an arbitrary natural number $k$ divisible by $p-1$. 
Let $y$ be a coset of
 $W/\Gamma^{p^j}$ not contained in $\Gamma$, and consider the
 $NU/U$-orbital sum in (\ref{eq:orbit}) containing the term 
 associated to $y$. 
 We may calculate it by applying (\ref{eq:suffzeta}) as 
follows:
\begin{align*}
 & \qquad \sum_{\sigma \in (NU/U)/(NU/U)_y} \Delta_{F_V/F_U}^w
 (1-k,\delta^{(\sigma^{-1}y\sigma)}) \kappa(\sigma^{-1}y\sigma)^{-k} \sigma^{-1}y\sigma  \\
 &= \Delta_{F_V/F_U}^w (1-k,\delta^{(y)}) \kappa(y)^{-k} \sum_{\sigma \in
 (NU/U)/(NU/U)_y} \sigma^{-1}y\sigma \\
 &\equiv \sharp(NU/U)_y \sum_{\sigma \in (NU/U)/(NU/U)_y} \sigma^{-1} y
 \sigma \qquad \qquad \mod \sharp(NU/U)_y.
\end{align*} 
 The element $\aug_{U,V}(P_y)=\sharp (NU/U) \aug_{U,V}(y)$ 
 is obviously divisible by $p$ if we set $P_y$ 
 as the last expression of the equation above 
 (note that the normaliser of $U$ is
 strictly larger than $U$ since $U$ is a proper open subgroup of 
 the pro-$p$ group $G$). 
This calculation implies that $P_y$ is an element 
in $J_{U,V}\otimes_{\iw{\Gamma}}\Z_p[\Gamma/\Gamma^{p^j}]
 /p^{m(\Gamma^{p^j})}$. If $U$ is an element in
 $\mathfrak{F}_A^c$, 
the element $P_y$ is no other than the image of 
 $y$ under the trace map from $\Z_p[[\conj{NU}]]$ to $\Z_p[[U]]$. 
Therefore $P_y$
 is contained in $I'_U\otimes_{\iw{\Gamma}}
 \Z_p[\Gamma/\Gamma^{p^j}]/p^{m(\Gamma^{p^j})}$. 
 Clearly $\Delta_{F_V/F_U}^w(1-k,\delta^{(y)})\kappa(y)^{-k}y$ is 
an element in $\Z_p[\Gamma/\Gamma^{p^j}]/p^{m(\Gamma^{p^j})}$ 
if $y$ is a coset contained in $\Gamma$, and hence
 we may show by taking the projective limit that the element  
 $(1-w)\xi_{U,V}$ (resp.\ $(1-w)\xi_U$ for $U$ in $\mathfrak{F}_A^c$) 
 is contained in $\iw{\Gamma}+J_{U,V}$ (resp.\
 $\iw{\Gamma}+I'_U$). Since $1-w$ is
 an invertible element in~$\iw{U/V}_S$, we obtain the desired congruences 
$(\ref{eq:congzetaj})$ and $(\ref{eq:congzetai})$.
\end{proof}

%%%%%%%%%%%%%%%%%%%%%%%%%%%%%%%%%%%%%%%%%%%%%%%%%%%%%%%%%
%
\subsection{Deligne-Ribet's theory upon Hilbert modular forms} \label{ssc:DR}
%
%%%%%%%%%%%%%%%%%%%%%%%%%%%%%%%%%%%%%%%%%%%%%%%%%%%%%%%%%

 We briefly summarise the theory of 
Pierre Deligne and Kenneth Alan Ribet upon Hilbert modular forms \cite{DR} 
in this subsection, 
which we shall use in verification of sufficient condition (\ref{eq:suffzeta}).

Let $K$ be a totally real number field of degree~$r$ and 
$K_{\infty}/K$ an {\em abelian} totally real \p-adic Lie extension. 
Let $\mathfrak{D}$ be the different of $K$ and $\Sigma$  
a finite set of prime ideals of $K$ and assume that $\Sigma$ 
contains all primes which ramify in $K_{\infty}$ (we fix such a finite
 set $\Sigma$ throughout the following argument).
We denote by $\mathfrak{h}_K$ the Hilbert upper-half space associated to $K$ 
defined as $\{\tau \in K\otimes \C \mid \mathrm{Im}(\tau) \gg 0\}$. 
For an even natural number $k$, we define the action of 
$\GL_2(K)^+$---subgroup of $\GL_2(K)$ consisting of all matrices
with totally positive determinants--- upon the set of $\C$-valued 
functions on $\mathfrak{h}_K$ by 
\begin{align*}
(F|_k \begin{pmatrix} a & b \\ c & d \end{pmatrix}) (\tau)=\mathcal{N}(ad-bc)^{k/2}\mathcal{N}(c\tau +d)^{-k} F(\dfrac{a\tau +b}{c\tau +d} )
\end{align*}
where $\mathcal{N} \colon K\otimes \C \rightarrow \C$ denotes the usual norm map.

\begin{defn}[Hilbert modular forms]
Let $\mathfrak{f}$ be an integral ideal of $\mathcal{O}_K$ 
all of whose prime factors are contained in $\Sigma$ and set
\begin{align*}
\Gamma_{00}(\mathfrak{f})=\{ \begin{pmatrix} a & b \\ c & d \end{pmatrix} \in \SL_2(K) \mid 
a,d \in 1+\mathfrak{f}, b \in \mathfrak{D}^{-1}, c \in \mathfrak{fD} \}.
\end{align*} 
 Then {\em a Hilbert modular form $F$ of $($parallel$)$ weight $k$ on
 $\Gamma_{00}(\mathfrak{f})$} is defined as a holomorphic function $F \colon
 \mathfrak{h}_K\rightarrow \C$ which is fixed by the action of
 $\Gamma_{00}(\mathfrak{f})$ (namely $F|_k M=F$ holds for an arbitrary
 element $M$ in $\Gamma_{00}(\mathfrak{f})$).\footnote{If $K$ is the
 rational number field $\Q$, we assume that $F$ is holomorphic at the
 cusp $\infty$.} 
\end{defn}

Let $\A^{\mathrm{fin}}_K$ denote the finite ad\`ele ring of $K$. 
Then $\SL_2(\A^{\mathrm{fin}}_K)$ is decomposed as
$\hat{\Gamma}_{00}(\mathfrak{f}) \cdot \SL_2(K)$ 
by the strong approximation theorem (we denote 
by $\hat{\Gamma}_{00}(\mathfrak{f})$ the topological closure 
of $\Gamma_{00}(\mathfrak{f})$ in $\SL_2(\A^{\mathrm{fin}}_K)$). 
We define 
the action of $\SL_2(\A_K^{\mathrm{fin}})$ upon the set of all $\C$-valued
functions on $\mathfrak{h}_K$ by $F|_k M=F|_k M_{\SL_2(K)}$ 
where $M_{\SL_2(K)}$ is the $\SL_2(K)$-factor of $M$ in 
$\SL_2(\A_K^{\mathrm{fin}})$. For a 
finite id\`ele $\alpha$ of $K$ and a Hilbert modular form $F$ of weight
$k$ on $\Gamma_{00}(\mathfrak{f})$, set
\begin{align*}
F_{\alpha}=F|_k \begin{pmatrix} \alpha & 0 \\ 0 & \alpha^{-1} \end{pmatrix}.
\end{align*} 
Then $F_{\alpha}$ has a Fourier series expansion 
\begin{align*}
F_{\alpha}=c(0,\alpha)+\sum_{\mu \in \mathcal{O}_K, \mu \gg 0} c(\mu,\alpha) q_K^{\mu}, \quad q_K^{\mu}=\exp(2\pi \sqrt{-1} \Tr_{K/\Q}(\mu \tau))
\end{align*}
which we call {\em the \q-expansion of $F$ at the cusp determined by 
$\alpha$}. 
Especially, the \q-expansion of $F$ at the cusp $\infty$ (determined by
$1$) is called {\em the 
standard \q-expansion of $F$}. Deligne and Ribet proved the following
deep theorem.

\begin{thm}[{\upshape \cite[Theorem~(0.2)]{DR}}] \label{thm:DR}
Let $F_k$ be a Hilbert modular form of weight $k$ on 
$\Gamma_{00}(\mathfrak{f})$. Assume that all coefficients of 
the \q-expansion of $F_k$ at an arbitrary cusp are rational numbers, 
and assume also that $F_k$ is equal to zero for all but finitely many $k$.  
Set $\mathscr{F}(\alpha)=\sum_{k\geq 0} \mathcal{N} \alpha_p^{-k}
 F_{k,\alpha}$ for a finite id\`ele $\alpha$ of $K$ whose \p-th component
 we denote by $\alpha_p$.
Then if the \q-expansion of $\mathscr{F}(\gamma)$ has all its coefficients in
 $p^j\Z_p$ 
 for a certain finite id\`ele $\gamma$ and a certain integer $j$, 
 the \q-expansion of $\mathscr{F}(\alpha)$ 
 for an arbitrary finite id\`ele $\alpha$ 
 also has all its coefficients in $p^j\Z_p$.
\end{thm}

The following corollary---so-called {\em the \q-expansion principle}--- 
plays the most important role in verification of 
sufficient condition (\ref{eq:suffzeta}).

\begin{cor}[\q-expansion principle] \label{cor:DRpri}
Let $F_k$ and $\mathscr{F}(\alpha)$ be as in Theorem~$\ref{thm:DR}$ and $j$ 
 an integer. Suppose that the \q-expansion of $\mathscr{F}(\gamma)$ has 
 all its non-constant coefficients in $p^j\Z_{(p)}$ for a certain finite
 id\`ele~$\gamma$. Then for arbitrary two distinct finite id\`eles $\alpha$
 and $\beta$, the difference between the constant terms of the
 \q-expansions of 
 $\mathscr{F}(\alpha)$ and $\mathscr{F}(\beta)$ is also contained 
 in $p^j\Z_{(p)}$.
\end{cor}

\begin{proof}
 Just apply Theorem~\ref{thm:DR} to $\mathscr{F}(\alpha)-c(0,\gamma)$ 
 and $\mathscr{F}(\beta)-c(0,\gamma)$ 
 where $c(0,\gamma)$ is the constant term of the \q-expansion of 
 $\mathscr{F}(\gamma)$. See also \cite[Corollary~(0.3)]{DR}.
\end{proof}

Finally we introduce the Hilbert-Eisenstein series 
attached to a locally constant $\C$-valued function $\varepsilon$ 
on $\Gal(K_{\infty}/K)$.

\begin{thm}[Hilbert-Eisenstein series] Let $\varepsilon$ be 
a locally constant function on $\Gal(K_{\infty}/K)$ and 
$k$ an even natural number. Then there exists an integral ideal 
$\mathfrak{f}$ of $\mathcal{O}_K$ all of whose prime factors 
are contained in $\Sigma$, and there exists 
a Hilbert modular form $G_{k,\varepsilon}$ of weight $k$ 
on $\Gamma_{00}(\mathfrak{f})$ $($which is called 
{\em the Hilbert-Eisenstein series of weight $k$ 
attached to $\varepsilon$}$)$ whose standard \q-expansion is given by 
\begin{align*}
2^{-r} \zeta^{\Sigma}_{K_{\infty}/K}(1-k, \varepsilon)+\sum_{\mu \in \mathcal{O}_K, \mu 
\gg 0} \left( \sum_{\mu \in \mathfrak{a} \subseteq \mathcal{O}_K,
 \text{\upshape prime to } \Sigma} \varepsilon(\mathfrak{a})\kappa(\mathfrak{a})^{k-1} \right) q_K^{\mu}
\end{align*}
$($we use the notation $\varepsilon(\mathfrak{a})$ and
 $\kappa(\mathfrak{a})$ for elements defined as 
$\varepsilon ( ( K_{\infty}/K, \mathfrak{a}))$ and 
$\kappa(( K_{\infty}/K, \mathfrak{a}))$ respectively 
where $(K_{\infty}/K,-)$ denotes the Artin symbol 
for the abelian extension $K_{\infty}/K)$. 
The \q-expansion of $G_{k,\varepsilon}$ at the cusp determined
 by a finite id\`ele $\alpha$ is given by
\begin{align} \label{eq:q-exp}
\mathcal{N}((\alpha))^k \left\{ 2^{-r} \zeta^{\Sigma}_{K_{\infty}/K}(1-k,
 \varepsilon_{a}) +\sum_{\tiny \begin{array}{@{}c@{}} 
\mu\in \mathcal{O}_K\\ \mu \gg 0 \end{array}} \left(
 \sum_{\tiny \begin{array}{@{}c@{}} 
  \mu \in \mathfrak{a}\subseteq \mathcal{O}_K \\ 
		    \text{\upshape prime to }\Sigma \end{array}}
 \varepsilon_a (\mathfrak{a})
 \kappa(\mathfrak{a})^{k-1} \right) q_K^{\mu} \right\}
\end{align}
where $(\alpha)$ is the ideal generated by $\alpha$ and $a$ is 
 an element in $\Gal(K_{\infty}/K)$ defined as 
 $(K_{\infty}/K, (\alpha)\alpha^{-1})$.
\end{thm}

For details, see \cite[Theorem~(6.1)]{DR}.

%%%%%%%%%%%%%%%%%%%%%%%%%%%%%%%%%%%%%%%%%%%%%%%%%%%%%%%%%%
%
\subsection{Proof of sufficient conditions} \label{ssc:proof-suff}
%
%%%%%%%%%%%%%%%%%%%%%%%%%%%%%%%%%%%%%%%%%%%%%%%%%%%%%%%%%%

In the rest of this section we shall verify sufficient condition
(\ref{eq:suffzeta}). This part is a subtle generalisation of the argument
in \cite[Section~6.6]{H}. 
Let $j$ be a sufficiently large integer and $y$ a coset of 
$W/\Gamma^{p^j}$ not contained in $\Gamma$. 
Choose an integral ideal $\mathfrak{f}$ of $\mathcal{O}_{F_{NU}}$  such
that the Hilbert-Eisenstein series $G_{k,\delta^{(y)}}$ over
$\mathfrak{h}_{F_U}$ is defined on $\Gamma_{00}(\mathfrak{f}
\mathcal{O}_{F_U})$. Then it is easy to see that the restriction 
$\mathscr{G}=G_{k,\delta^{(y)}}|_{\mathfrak{h}_{F_{NU}}}$ 
of~$G_{k,\delta^{(y)}}$
to~$\mathfrak{h}_{F_{NU}}$ is also a Hilbert modular form of weight
$p^{n_U}k$ on $\Gamma_{00}(\mathfrak{f})$ where $p^{n_U}$ is the cardinality
of the quotient group $NU/U$. The \q-expansion of $\mathscr{G}$ is 
directly calculated as  
\begin{align*}
2^{-[F_U:\Q]}  \zeta^{\Sigma}_{F_V/F_U} 
 (1-k, \delta^{(y)}) + \sum_{\tiny\begin{array}{@{}c@{}}
  \nu \in \mathcal{O}_{F_U} \\ \nu \gg 0\end{array}}
 \left( \sum_{\tiny\begin{array}{@{}c@{}} \nu \in \mathfrak{b} \subseteq
  \mathcal{O}_{F_U} \\ \text{prime to }\Sigma \end{array}}
 \delta^{(y)}( \mathfrak{b}) \kappa(\mathfrak{b})^{k-1} \right)
 q_{F_{NU}}^{\mathfrak{tr}(\nu)}
\end{align*}
where $q_{F_{NU}}^{\mathfrak{tr}(\nu)}$ denotes $\exp(2\pi {\sqrt{-1}}
\Tr_{F_{NU}/\Q}(\Tr_{F_U/F_{NU}}(\nu) \tau))$. 
Note that the quotient group $NU/U$ naturally
acts upon the set of all pairs $(\mathfrak{b}, \nu)$ such that $\mathfrak{b}$
is a non-zero integral ideal of $\mathcal{O}_{F_U}$ prime to $\Sigma$ 
and $\nu$ is a
totally positive element in $\mathfrak{b}$. 
First suppose that the isotropy subgroup $(NU/U)_{(\mathfrak{b}, \nu)}$ is
 trivial. Then we can easily calculate the 
 $NU/U$-orbital sum in the \q-expansion of $\mathscr{G}$ 
 containing the term associated to $(\mathfrak{b},\nu)$ as follows:
\begin{align*}
 \sum_{\sigma\in NU/U}  \! \! \! \! \! \delta^{(y)}(\mathfrak{b}^{\sigma})
      \kappa(\mathfrak{b}^{\sigma})^{k-1} q_{F_{NU}}^{\tr(\nu^{\sigma})} 
=\sharp(NU/U)_y \! \! \! \! \! \! \! \! \! \! \! \sum_{\sigma\in
      (NU/U)/(NU/U)_y} \! \! \! \! \! \! \! \! \! \! \!\delta^{(\sigma y \sigma^{-1})}(\mathfrak{b})
      \kappa(\mathfrak{b})^{k-1} q_{F_{NU}}^{\mathfrak{tr}(\nu)}
\end{align*}
(use the obvious formula $\Tr_{F_U/F_{NU}}(\nu^{\sigma})=\Tr_{F_U/F_{NU}}(\nu)$).

Next suppose that the isotropy subgroup $(NU/U)_{(\mathfrak{b}, \nu)}$ is
 not trivial. Let $F_{(\mathfrak{b},\nu)}$ be the fixed subfield of
 $F_U$ by $(NU/U)_{(\mathfrak{b}, \nu)}$ and
 $F^{\mathrm{comm}}_{(\mathfrak{b},\nu)}$ the fixed subfield 
 of $F_{\infty}$ by the
 commutator subgroup of $NU_{(\mathfrak{b}, \nu)}$. Then $(\mathfrak{b}, \nu)$
 is fixed by the action of $\Gal(F_U/F_{(\mathfrak{b}, \nu)})$, and hence
 $\nu$ is an element in $F_{(\mathfrak{b}, \nu)}$ and there
 exists a non-zero integral ideal $\mathfrak{a}$ of
 $\mathcal{O}_{F_{(\mathfrak{b}, \nu)}}$ such that
 $\mathfrak{a}\mathcal{O}_{F_U}$ coincides with $\mathfrak{b}$. 
 For such $(\mathfrak{a}, \nu)$, the equation
\begin{align*}
 \delta^{(y)} (\mathfrak{b}) =\delta^{(y)} (( F_V/F_U,\mathfrak{a}\mathcal{O}_{F_U})) = \delta^{(y)} \circ \Ver (( F_{(\mathfrak{b},\nu)}^{\mathrm{comm}}/F_{(\mathfrak{b},
 \nu)}, \mathfrak{a} )) =0
\end{align*}
holds because the image of the Verlagerung homomorphism 
is contained in $\Gamma$ (indeed the Verlagerung coincides 
with the $n_U$-th power of the 
Frobenius endomorphism $\varphi^{n_U}$ 
if the finite part of the Galois group
is of exponent $p$; see \cite[Lemma~4.3]{H} for details)
but $y$ is not contained in $\Gamma$.

The calculation above implies that $\mathscr{G}$ has all 
non-constant coefficients in $\sharp(NU/U)_y \Z_{(p)}$. Take a finite
id\`ele $\gamma$ such that 
$(F_V/F_U, (\gamma){\gamma}^{-1})$ coincides with $w$. Then 
by Deligne-Ribet's \q-expansion
principle (Corollary~\ref{cor:DRpri}) the constant term of
$\mathscr{G}-\mathscr{G}(\gamma)$ is also contained 
in $\sharp(NU/U)_y\Z_{(p)}$, which we may calculate as
$2^{-[F_U:\Q]} \Delta_{F_V/F_U}^{w} (1-k, \delta^{(y)})$ (use the
explicit formula (\ref{eq:q-exp}) 
for the \q-expansion of $\mathscr{G}(\gamma)$). Therefore 
sufficient condition $(\ref{eq:suffzeta})$ holds 
(recall that $2$ is invertible since $p$ is odd).

%%%%%%%%%%%%%%%%%%%%%%%%%%%%%%%%%%%%%%%%%%%%%%%%%%%%%%%%%
%
%
\section{Inductive construction of the $p$-adic zeta functions} \label{sc:construction}
%
%
%%%%%%%%%%%%%%%%%%%%%%%%%%%%%%%%%%%%%%%%%%%%%%%%%%%%%%%%%

We shall complete the proof of our main theorem 
(Theorem~\ref{thm:maintheorem}). We first construct the \p-adic zeta
function ``modulo \p-torsion'' for $F_{\infty}/F$, and 
then eliminate ambiguity of the \p-torsion part.

%%%%%%%%%%%%%%%%%%%%%%%%%%%%%%%%%%%%%%%%%%%%%%%%%%%%%%%%%
%
\subsection{Choice of the central element $c$} \label{ssc:choice}
%
%%%%%%%%%%%%%%%%%%%%%%%%%%%%%%%%%%%%%%%%%%%%%%%%%%%%%%%%%

In order to let induction work effectively, we have to choose 
a ``good'' central element $c$ 
which is used in the construction of the Artinian family $\mathfrak{F}_A^c$ 
(see Section~\ref{ssc:family}). 
The following elementary lemma implies how to choose 
such a ``good'' central element.

\begin{lem} \label{lem:non-comm}
Let $\Delta$ be a finite \p-group with exponent $p$ and $c$ a non-trivial 
central element in $\Delta$. 
If $c$ is not contained in the commutator subgroup of $\Delta$, 
the \p-group $\Delta$ is isomorphic to the direct product 
of the cyclic group $\langle c \rangle$ generated by $c$ and 
the quotient group $\Delta/\langle c \rangle$.
\end{lem}

\begin{proof}
By the structure theorem of finite abelian groups, 
the abelisation $\Delta^{\mathrm{ab}}$ of $\Delta$ is 
decomposed as the direct product of the image of the cyclic group 
$\langle c \rangle$ and a certain finite abelian \p-group 
$\bar{H}$ with exponent $p$. Let $H$ denote the inverse image of 
$\bar{H}$ under the abelisation map 
$\Delta \rightarrow \Delta^{\mathrm{ab}}$. 
Then one may easily verify that $H$ and $\langle c \rangle$ 
generate $\Delta$. The intersection of $H$ and 
$\langle c \rangle$ is obviously trivial, and $H$ commutes with 
elements in $\langle c\rangle$ since $c$ is central.
\end{proof}

This lemma asserts that there exists a non-trivial central 
element $c$ which is contained in the commutator 
subgroup of $G^f$ if $G$ is not abelian. 
We may assume that $G$ is non-commutative without loss of generality 
(abelian cases are just the results of Deligne, Ribet and Wiles), 
and thus we may always find a non-trivial central element 
contained in $[G,G]$. 

In the following argument we take a non-trivial central element $c$ 
from the commutator subgroup of $G$ and fix it.

%%%%%%%%%%%%%%%%%%%%%%%%%%%%%%%%%%%%%%%%%%%%%%%%%%%%%%%%%
%
\subsection{Construction of the $p$-adic zeta function  ``modulo
  $p$-torsion''} \label{ssc:mod-p-construct}
%
%%%%%%%%%%%%%%%%%%%%%%%%%%%%%%%%%%%%%%%%%%%%%%%%%%%%%%%%%

In this subsection we construct the \p-adic zeta function 
``modulo \p-torsion'' for $F_{\infty}/F$, by mimicking Burns' technique
(see Section~\ref{sc:burns}).
First consider the following commutative diagram with exact rows:
\[
\footnotesize
 \xymatrix{
    K_1(\iw{G}) \ar[r] \ar@{->}[d]_(0.45){\theta}   &   K_1(\iw{G}_S)  
\ar[r]^(0.45){\partial} \ar[d]^(0.45){\theta_S}    &   K_0(\iw{G},
\iw{G}_S) \ar[r] \ar[d]^(0.45){\mathrm{norm}}  &  0 \\
  \prod_{\mathfrak{F}_B} \iw{U/V}^{\times} \ar@{^(->}[r]
 & \prod_{\mathfrak{F}_B} \iw{U/V}_S^{\times}
 \ar[r]_(0.355){\partial}  &   \prod_{\mathfrak{F}_B} 
K_0(\iw{U/V}, \iw{U/V}_S) \ar[r] & 0.}
\]
Let $f$ be an arbitrary characteristic element for $F_{\infty}/F$ (see
Section~\ref{ssc:review-iwa}) 
and set $\theta_S(f)=(f_{U,V})_{(U,V)\in \mathfrak{F}_B}$. 
For each $(U,V)$ in $\mathfrak{F}_B$ let $w_{U,V}$ be the 
element defined as $\xi_{U,V}f_{U,V}^{-1}$, which is contained 
in $\iw{U/V}^{\times}$ by an argument similar to Burns' technique. 
Let $\tilde{w}_{U,V}$ denote the image of $w_{U,V}$ in $\piw{U/V}^{\times}$.
Since both $(f_{U,V})_{(U,V)\in \mathfrak{F}_B}$ and
$(\xi_{U,V})_{(U,V) \in \mathfrak{F}_B}$ satisfy conditions 
(NCC) and (CCC) (see
Proposition~\ref{prop:thetaloc} and Lemma~\ref{lem:ncc-ccc}),
the element $(\tilde{w}_{U,V})_{(U,V)\in \mathfrak{F}_B}$ also
satisfies them. Moreover there exists an element
$\tilde{d}_{U,V}$ (resp.\ $\tilde{d}_U$) in $\piw{\Gamma}_{(p)}^{\times}$ 
such that the congruence 
\begin{align} \label{eq:cong-w}
 \tilde{w}_{U,V} \equiv \tilde{d}_{U,V} \quad \pmod{J_{S,U,V}} \qquad (\text{resp.\ } & \tilde{w}_U \equiv \tilde{d}_U \quad \pmod{I'_{S,U}}) 
\end{align}
holds for each $(U,V)$ in $\mathfrak{F}_B$ except for $(G,[G,G])$ 
(resp.\ for each $U$ in $\mathfrak{F}_A^c$) by Proposition~\ref{prop:thetaloc} 
and Proposition~\ref{prop:congzeta}. 
We remark that these congruences are {\em not} sufficient to prove that
$(\tilde{w}_{U,V})_{(U,V)\in \mathfrak{F}_B}$ is contained in
$\widetilde{\Psi}$ (or equivalently in $\widetilde{\Psi}_c$).

\begin{rem} \label{rem:eliminate}
Unfortunately the congruences (\ref{eq:cong-w}) hold
not in $\iw{U/V}$ (resp.\ $\iw{U}$) but in $\iw{U/V}_S$ (resp.\
$\iw{U}_S$), but we may obtain the {\em integral} congruences
(\ref{eq:integral-cong}) later by ``eliminating $\tilde{d}_{U,V}$ and 
$\tilde{d}_U$.'' The author 
would like to appreciate Mahesh Kakde for pointing out the wrong
 arguments around these phenomena in the preliminary version of this article.
\end{rem}

\begin{thm}[strong congruences modulo \p-torsion] \label{thm:str-cong}
The congruence 
\begin{align*}
 \tilde{w}_{U,V} &\equiv \varphi(\tilde{w}_{\mathrm{ab}})^{(G:U)/p} &&\pmod{J_{U,V}} &
(\text{resp.\quad }  \tilde{w}_U &\equiv
 \varphi(\tilde{w}_{\mathrm{ab}})^{(G:U)/p} &&\pmod{I_U})
\end{align*}
holds for each $(U,V)$ in $\mathfrak{F}_B$ except for $(G,[G,G])$ 
$($resp.\ for each $U$ in $\mathfrak{F}_A^c)$.
\end{thm}

\begin{proof}
Recall that the non-negative integer $N$ is defined by $\sharp
 G^f=p^N$. We shall prove the claim by induction on $N$. 
We first assume that $G$ is abelian. Then 
the element $(\xi_{U,\{ e\}})_{(U, \{e\}) \in \mathfrak{F}_B}$ 
is in fact contained in the image 
of $\tilde{\theta}_S$ (use the existence of the \p-adic zeta pseudomeasure 
for $F_{\infty}/F$), and hence 
$(\xi_{U,\{e\}})_{(U,\{e\}) \in \mathfrak{F}_B}$ satisfies 
desired congruence condition and additional congruence condition. 
This implies that $(\tilde{w}_{U,\{e\}})_{(U,\{e\}) \in \mathfrak{F}_B}$ also 
satisfies them. In particular 
the cases where $N$ equals either $0$, $1$ or $2$ are done. 
Therefore we assume that $N$ is larger than $3$ and $G$ is non-commutative 
in the following argument.

 Now let $(U,V)$ (resp.\ $U$) be an
 element in  
$\mathfrak{F}_B$ (resp.\ $\mathfrak{F}_A^c$) such that $U$ contains the
 fixed central element $c$ chosen as in Section~\ref{ssc:choice}. 
Set $\bar{G}=G/\langle c \rangle$,
 $\bar{U}=U/\langle c \rangle$ and $\bar{V}=V\langle c \rangle / \langle
 c \rangle$ respectively. 
 Clearly the set of all such $(\bar{U}, \bar{V})$ is a Brauer family 
 $\bar{\mathfrak{F}}_B$ for $\bar{G}$, and the set of 
$\bar{U}_h=U_{h,c}/\langle c\rangle$ for all $h$ in $\mathfrak{H}$ 
 is an Artinian family $\bar{\mathfrak{F}}_A$ for $\bar{G}$. 
Note that the norm map $\Nr_{\iw{G}_S/\iw{U}_S}$ and 
the canonical homomorphism 
$K_1(\iw{G}_S)\rightarrow K_1(\iw{\bar{G}}_S)$ are compatible, and 
note also that the image of $\xi_{U,V}$ 
under the canonical quotient map 
$\iw{U/V}_S^{\times} \rightarrow \iw{\bar{U}/\bar{V}}_S^{\times}$ 
coincides with the \p-adic zeta pseudomeasure 
 $\xi_{\bar{U},\bar{V}}$ for $F_{\bar{V}}/F_{\bar{U}}$ 
(easily follows from its interpolation property). 
Hence we may apply the induction hypothesis to the image 
$\tilde{\bar{w}}_{\bar{U}, \bar{V}}$ of 
$\tilde{w}_{U,V}$ in $\piw{\bar{U}/\bar{V}}^{\times}$; 
in other words, the congruences 
\begin{align} \label{eq:cong-hyp}
 \tilde{\bar{w}}_{\bar{U}, \bar{V}} &\equiv
 \varphi(\tilde{\bar{w}}_{\mathrm{ab}})^{(\bar{G}:\bar{U})/p} &&
\pmod{J_{\bar{U}, \bar{V}}}, &
  \tilde{\bar{w}}_{\bar{U}} &\equiv \varphi(\tilde{\bar{w}}_{\mathrm{ab}})^{(\bar{G}:\bar{U})/p} &&\pmod{I_{\bar{U}}} 
\end{align}
hold for each $(\bar{U}, \bar{V})$ in $\bar{\mathfrak{F}}_B$ 
except for $(G,[G,G])$ and for each $\bar{U}$ in $\bar{\mathfrak{F}}_A$ 
if we define $J_{\bar{U}, \bar{V}}$ and $I_{\bar{U}}$ analogously to 
$J_{U,V}$ and $I_U$. On the other hand 
 we may readily verify that the natural surjection 
 $\iw{U/V}\rightarrow \iw{\bar{U}/\bar{V}}$ maps $J_{U,V}$ 
 to $J_{\bar{U}, \bar{V}}$ and $I_U$ to
 $I_{\bar{U}}$ respectively (use the definition of $J_{U,V}$ 
 and the explicit description of $I_U$). Let
 $I'_{\bar{U}}$ denote
 the image of the trace map
 $\Tr_{\Z_p[[\conj{N\bar{U}}]]/\Z_p[[\bar{U}]]}$ 
 for each $\bar{U}$ in $\bar{\mathfrak{F}}_A$, and let
 $J_{S,\bar{U},\bar{V}}$ (resp.\ $I'_{S,\bar{U}}$) denote
 the scalar extension 
 $J_{\bar{U}, \bar{V}}\otimes_{\iw{\Gamma}}\iw{\Gamma}_{(p)}$ 
 (resp.\ $I'_{\bar{U}}\otimes_{\iw{\Gamma}}\iw{\Gamma}_{(p)}$).
 Then we obtain the congruences 
\begin{align} \label{eq:cong-red}
\tilde{\bar{w}}_{\bar{U}, \bar{V}} &\equiv \tilde{d}_{U,V} &&
 \pmod{J_{S,\bar{U}, \bar{V}}}, &    \tilde{\bar{w}}_{\bar{U}} 
& \equiv \tilde{d}_U && \pmod{I_{S,\bar{U}}'} 
\end{align}
by applying the canonical surjection $\iw{U/V} \rightarrow
 \iw{\bar{U}/\bar{V}}$ to (\ref{eq:cong-w})
 (recall that $I_{S, \bar{U}}'$ contains
 $I_{S, \bar{U}}$). The congruences 
 (\ref{eq:cong-hyp}) and (\ref{eq:cong-red}) implies that 
 for $(U,V)$ in $\mathfrak{F}_B$ except for $(G,[G,G])$ 
 the element 
 $\varphi(\tilde{\bar{w}}_{\mathrm{ab}})^{-(\bar{G}:\bar{U})/p}\tilde{d}_{U,V}$
 is contained in $1+J_{S,\bar{U}, \bar{V}}^{\, \widetilde{\,}} \cap
 \piw{\Gamma}_{(p)}^{\times}$, which coincides 
 with $1+p\iw{\Gamma}_{(p)}^{\,\widetilde{\,}}$ 
 by definition. Furthermore for $U$ in $\mathfrak{F}_A^c$ the element 
 $\varphi(\tilde{\bar{w}}_{\mathrm{ab}})^{-(\bar{G}:\bar{U})/p}\tilde{d}_U$
 is contained in $(1+I'_{S, \bar{U}})^{\, \widetilde{\,}}
 \cap \piw{\Gamma}_{(p)}^{\times}$, which coincides with
 $1+p^{n_h-\epsilon}\iw{\Gamma}_{(p)}^{\, \widetilde{\,}}$ by the explicit
 description of $I'_{S,\bar{U}}$ (the integer $\epsilon$ is defined 
 as $2$ for (Case-1) and $1$ for (Case-2)). 
 Obviously the equations 
 $\varphi(\tilde{\bar{w}}_{\mathrm{ab}})=\varphi(\tilde{w}_{\mathrm{ab}})$ 
 and $(\bar{G}:\bar{U})=(G:U)$ hold by construction, 
 and therefore the congruences
\begin{equation} \label{eq:congd}
\begin{aligned}
\tilde{d}_{U,V} &\equiv \varphi(\tilde{w}_{\mathrm{ab}})^{(G:U)/p}
 && \pmod{\quad p\iw{\Gamma}_{(p)}}, \\ 
\tilde{d}_U &\equiv 
\varphi(\tilde{w}_{\mathrm{ab}})^{(G:U)/p} &&
 \pmod{p^{n_h-\epsilon}\iw{\Gamma}_{(p)}}.
\end{aligned}
\end{equation}
hold. Combining (\ref{eq:congd}) with (\ref{eq:cong-w}), we obtain the
 following congruences:\footnote{Since both $\tilde{w}_{U,V}$
 (resp.\ $\tilde{w}_U$) and $\varphi(\tilde{w}_{\mathrm{ab}})^{(G:U)/p}$
 are contained in $\piw{U/V}^{\times}$, the congruence
 (\ref{eq:integral-cong}) actually holds in $\iw{U/V}$ (resp.\ in
 $\iw{U}$) and we may remove the sub-index $S$ from the congruence. 
 This is the ``eliminating~$\tilde{d}$'' procedure mentioned 
 in Remark~\ref{rem:eliminate}.}
\begin{align} \label{eq:integral-cong}
 \tilde{w}_{U,V} &\equiv \varphi(\tilde{w}_{\mathrm{ab}})^{(G:U)/p}
 &&\pmod{J_{U,V}},  &
 \tilde{w}_U &\equiv \varphi(\tilde{w}_{\mathrm{ab}})^{(G:U)/p} &&\pmod{I'_U}.
\end{align}
The former congruence is no other than the desired one. The latter one
 for $U_c$ is also the desired one because $I'_{U_c}$ coincides
 with $I_{U_c}$ by definition. Now consider the congruence for
 $U_{h,c}$. It suffices to consider the case where $U_{h,c}$ is a proper 
subgroup of $G$ (see Remark~\ref{rem:G}). Since $\log
 (\varphi(\tilde{w}_{\mathrm{ab}})^{-p^{N-3}} \tilde{w}_{U_{h,c}})$ is
 contained in $I'_{U_{h,c}}$ by (\ref{eq:integral-cong}), it is
 explicitly described as 
\begin{align*}
\log \frac{\tilde{w}_{U_{h,c}}}{\varphi(\tilde{w}_{\mathrm{ab}})^{p^{N-3}}}
= \sum_{i=0}^{p-1} p^{n_h-\epsilon} a_i c^i +(\text{terms containing $h$})
\end{align*}
where each $a_i$ is an element in $\iw{\Gamma}$. Furthermore the equation 
\begin{align*}
\log \frac{\tilde{w}_{U_c}}{\varphi(\tilde{w}_{\mathrm{ab}})^{p^{N-2}}}
= \Tr_{\Z_p[[U_{h,c}]]/\Z_p[[U_c]]}(\log \frac{\tilde{w}_{U_{h,c}}}{\varphi(\tilde{w}_{\mathrm{ab}})^{p^{N-3}}})
= \sum_{i=0}^{p-1} p^{n_h-\epsilon+1} a_i c^i
\end{align*}
holds by (TCC). The first expression of the equation above 
is contained in $I_{U_c}=p^{N-1}\Z_p[[U_c]]$
 as we have already remarked, 
and hence there exists an element $b_i$ in $\iw{\Gamma}$ such that
 $p^{n_h-\epsilon+1}a_i$ coincides with $p^{N-1}b_i$ for each $i$. 
 Therefore we may conclude that the element $\log
 (\varphi(\tilde{w}_{\mathrm{ab}})^{-p^{N-3}} \tilde{w}_{U_{h,c}})$ is 
 contained in~$I_{U_{h,c}}$. This implies the desired
 congruence for $U_{h,c}$ because the logarithm induces an injection on $1+J_{U_{h,c}}^{\,
 \widetilde{\,}}$ (Proposition~\ref{prop:log-j}) 
 and an isomorphism between $1+I_{U_{h,c}}^{\, \widetilde{\,}}$ 
 and $I_{U_{h,c}}$ (Proposition~\ref{prop:log-i}).

Next let $(U,V)$ be an element in $\mathfrak{F}_B$ 
such that $U$ does not contain the fixed central element $c$. 
We claim that $U\times \langle c \rangle$ does not coincide with $G$; 
indeed if it does, the commutator subgroup of $G$ automatically 
coincides with $V$ which does not contain $c$. This is contradiction 
since we choose such $c$ as contained in $[G,G]$ in Section~\ref{ssc:choice}.
Now we apply the argument above to the pair 
 $(U\times \langle c \rangle, V)$ and obtain the congruence 
\begin{align*}
\tilde{w}_{U\times \langle c \rangle, V} \equiv \varphi(\tilde{w}_{\mathrm{ab}})^{(G:U)/p^2} \quad \pmod{J_{U\times \langle c \rangle, V}}
\end{align*}
(use the obvious equation 
$(G:U\times \langle c \rangle)=(G:U)/p$). By using (NCC) 
and the fact that $\varphi(\tilde{w}_{\mathrm{ab}})$ is contained in the centre of $\piw{U\times \langle c \rangle}^{\times}$, we have
\begin{align*}
\Nr_{\iw{U\times \langle c \rangle /V}/\iw{U/V}} (\varphi(\tilde{w}_{\mathrm{ab}})^{-(G:U)/p^2}\tilde{w}_{U\times \langle c \rangle, V})=\varphi(\tilde{w}_{\mathrm{ab}})^{-(G:U)/p}\tilde{w}_{U,V}.
\end{align*}
On the other hand the left hand side of the equation above 
is contained in $1+J_{U,V}^{\, \widetilde{\,}}$ 
by Corollary~\ref{cor:aug-norm}. The desired congruence thus holds for $(U,V)$.

Finally let $U_h$ be an element in $\mathfrak{F}_A$ and assume that $h$ does not coincide with $c$. By the same argument as above, we may conclude that 
$\varphi(\tilde{w}_{\mathrm{ab}})^{-p^{N-2}}\tilde{w}_{U_h}$ is
 contained in $1+J_{U_h}^{\, \widetilde{\,}}$. On the other hand the element
 $\varphi(\tilde{w}_{\mathrm{ab}})^{-p^{N-3}}\tilde{w}_{U_{h,c}}$ is
 contained in $1+I_{U_{h,c}}^{\, \widetilde{\,}}$ by the argument
 above. Now the compatibility lemma (Lemma~\ref{lem:lognorm}) 
 enables us to calculate as follows:
\begin{align*}
\Tr_{\Z_p[[U_{h,c}]]/\Z_p[[U_h]]} (\log \frac{\tilde{w}_{U_{h,c}}}{\varphi(\tilde{w}_{\mathrm{ab}})^{p^{N-3}}}) &=\log (\Nr_{\iw{U_{h,c}}/\iw{U_h}} (\frac{\tilde{w}_{U_{h,c}}}{\varphi(\tilde{w}_{\mathrm{ab}})^{p^{N-3}}}))\\  
&= \log \frac{\tilde{w}_{U_h}}{\varphi(\tilde{w}_{\mathrm{ab}})^{p^{N-2}}}.
\end{align*}
The $\Z_p$-module 
$\Tr_{\Z_p[[U_{h,c}]]/\Z_p[[U_h]]}(I_{U_{h,c}})$ is 
contained in $I_{U_h}$ by definition, 
and thus 
$\log (\varphi(\tilde{w}_{\mathrm{ab}})^{-p^{N-2}}\tilde{w}_{U_h})$ 
is also contained in $I_{U_h}$.  
The desired congruence now holds for $U_h$ 
because the logarithm induces an injection on $1+J_{U_h}^{\,
 \widetilde{\,}}$ (Proposition~\ref{prop:log-j})
 and an isomorphism between $1+I_{U_h}^{\,
 \widetilde{\,}}$ and $I_{U_h}$ (Proposition~\ref{prop:log-i}).
\end{proof}

By virtue of Theorem~\ref{thm:str-cong}, we may conclude that 
$(\tilde{w}_{U,V})_{(U,V)\in \mathfrak{F}_B}$ is an~element in 
$\widetilde{\Psi}_c$. 
Hence there exists a unique element $\tilde{w}$ in $\pK_1(\iw{G})$ such
that $\tilde{\theta}(\tilde{w})=(\tilde{w}_{U,V})_{(U,V)\in
\mathfrak{F}_B}$ holds (Proposition~\ref{prop:injection} and
Proposition~\ref{prop:surjection}). 
 Take an arbitrary lift of $\tilde{w}$ to $K_1(\iw{G})$ and
 set $\tilde{\xi}=f\tilde{w}$. Then by construction, 
 we may easily check that $\tilde{\xi}$ satisfies the following two properties:
\begin{enumerate}[$(\tilde{\xi}\text{-}1)$]
\item the equation $\partial(\tilde{\xi})=-[C_{F_{\infty}/F}]$ holds;
\item there exists an element
 $\tau_{U,V}$ in $\iw{U/V}^{\times}_{p\text{-tors}}$ 
      for each $(U,V)$ in~$\mathfrak{F}_B$ such that the equation
      $\theta_S(\tilde{\xi})=(\xi_{U,V}\tau_{U,V})_{(U,V)\in \mathfrak{F}_B}$ holds.
\end{enumerate}

By using $(\sharp)_A$ and $(\tilde{\xi}\text{-}2)$, we may
show that there exists a
\p-power root of unity $\zeta_{\rho, r}$ such that the equation 
$\tilde{\xi}(\rho\kappa^r) =\zeta_{\rho,r} L_{\Sigma}(1-r; F_{\infty}/F, \rho)$
holds for an arbitrary Artin representation $\rho$ of $G$ and an
arbitrary natural number $r$ divisible by~${p-1}$.
Roughly speaking, the element $\tilde{\xi}$ is 
the \p-adic zeta function ``modulo
\p-torsion'' for $F_{\infty}/F$ which interpolates special values 
of complex 
Artin $L$-functions {\em up to multiplication by a \p-power root of unity}. 

%%%%%%%%%%%%%%%%%%%%%%%%%%%%%%%%%%%%%%%%%%%%%%%%%%%%%%%%%
%
\subsection{Refinement of the $p$-torsion part}
%
%%%%%%%%%%%%%%%%%%%%%%%%%%%%%%%%%%%%%%%%%%%%%%%%%%%%%%%%%

We shall finally modify $\tilde{\xi}$ and reconstruct 
the \p-adic zeta function 
$\xi$ for $F_{\infty}/F$ without any ambiguity upon \p-torsion elements. 
The author strongly believes that 
our argument to remove ambiguity of the \p-torsion part 
is based upon essentially the same spirits 
as ``the torsion congruence method''  
used by J\"urgen Ritter and Alfred Weiss \cite{RW5}. We shall, however, adopt 
somewhat different formalism from theirs.

Let $\tilde{\xi}$ be the $p$-adic zeta function ``modulo \p-torsion'' 
for $F_{\infty}/F$ and set 
$\tau_{U,V}=\xi_{U,V}\theta_{S,U,V}(\tilde{\xi})^{-1}$ 
for each $(U,V)$ in $\mathfrak{F}_B$. 
Then $\tau_{U,V}$ is a \p-torsion element by definition. 
Moreover $\tau_{U,V}$ is an element in $\iw{U/V}^{\times}$ 
by the same argument as Burns' technique. 
Since the \p-torsion part of $K_1(\iw{G})$ is 
identified with $G^{f,\mathrm{ab}}\times SK_1(\Z_p[G^f])$ and that of 
$\iw{G^{\mathrm{ab}}}^{\times}$ is identified with $G^{f,\mathrm{ab}}$
respectively (see Section~\ref{ssc:intlog}), we may naturally regard
$\tau_{\mathrm{ab}}$ as an element in $K_1(\iw{G})_{p\text{-tors}}$. 
Set $\xi=\tau_{\mathrm{ab}} \tilde{\xi}$. 
Then $\theta_{S,\mathrm{ab}}(\xi)=\xi_{\mathrm{ab}}$ obviously holds 
by construction.

\begin{thm} \label{thm:p-tors}
The equation 
$\theta_{S,U,V}(\xi)=\xi_{U,V}$ holds for each $(U,V)$ in $\mathfrak{F}_B$.
\end{thm}

If the claim is verified, we may conclude that $\xi$ satisfies the
interpolation formula (\ref{eq:interp}) without any ambiguity by Brauer
induction (see
Section~\ref{sc:burns}). Therefore $\xi$ is no other than the ``true'' 
\p-adic zeta function for $F_{\infty}/F$.

We shall prove Theorem~\ref{thm:p-tors} by induction on $N$.
First assume that $G$ is abelian. Then the obvious equation 
$\xi=\theta_{S,\mathrm{ab}}(\xi)=\xi_{\mathrm{ab}}$ implies that 
$\xi$ is actually the \p-adic zeta function for $F_{\infty}/F$. 
In particular the cases in which $N$ equals either $0,1$ or $2$ 
are done. 

Now suppose that $N$ is larger than $3$ and $G$ is non-commutative. 
Let $c$ be a non-trivial central element in $G$ 
chosen as in Section~\ref{ssc:choice} and set $\bar{G}=G/\langle c \rangle$.
Let $\bar{\xi}$ be the image of $\xi$ 
under the canonical map 
$K_1(\iw{G}_S) \rightarrow K_1(\iw{\bar{G}}_S)$. 
Then the element $\bar{\xi}$ is the \p-adic zeta function 
``modulo \p-torsion'' for $F_{\langle c \rangle}/F$ 
by~construction. 
Furthermore the following diagram commutes since $c$ is contained 
in the commutator subgroup of $G$ (here $\bar{\theta}_{S,\mathrm{ab}}$ 
denotes the abelisation map for $\iw{\bar{G}}_S$): 
\begin{align*}
\xymatrix{
K_1(\iw{G}_S) \ar[r]^{\theta_{S,\mathrm{ab}}} \ar[d]_(0.45){\text{canonical}} & \iw{G^{\mathrm{ab}}}_S^{\times} \ar@{=}[d] \\
K_1(\iw{\bar{G}}_S) \ar[r]_{\bar{\theta}_{S, \mathrm{ab}}} & 
\iw{\bar{G}^{\mathrm{ab}}}_S^{\times}.
}
\end{align*}
This asserts that 
$\bar{\theta}_{S,\mathrm{ab}}(\bar{\xi})=\xi_{\mathrm{ab}}$ holds, 
and we may thus apply the induction hypothesis to $\bar{\xi}$; 
in other words we may assume that $\bar{\xi}$ is the ``true'' 
\p-adic zeta function for $F_{\langle c \rangle}/F$.
Now take an arbitrary pair $(U,V)$ in $\mathfrak{F}_B$.

\medskip
\noindent {\bfseries (Case-1).} Suppose that $c$ is contained in
$V$. Let $\bar{U}$ and $\bar{V}$ denote the quotient groups
$U/\langle c \rangle$ and $V/\langle c \rangle$ respectively. 
Let $\bar{\theta}_{S,\bar{U},\bar{V}}$ be 
the composition of the norm map 
$\Nr_{\iw{\bar{G}}_S/\iw{\bar{U}}_S}$ 
with the canonical homomorphism
$K_1(\iw{\bar{U}}_S)\rightarrow \iw{\bar{U}/\bar{V}}_S^{\times}$. 
Then it is clear that $U/V$ coincides with $\bar{U}/\bar{V}$ 
and the theta maps $\theta_{S,U,V}$ and 
$\bar{\theta}_{S,\bar{U}, \bar{V}}$ are compatible. Hence we obtain 
\begin{align*}
\theta_{S,U,V}(\xi)= \bar{\theta}_{S,\bar{U}, \bar{V}}(\bar{\xi})=\xi_{\bar{U},\bar{V}}=\xi_{U, V}
\end{align*}  
which is the desired result (the last equality follows from the fact
that 
$F_{\bar{V}}/F_{\bar{U}}$ is the completely same extension as $F_V/F_U$).

\medskip
\noindent {\bfseries (Case-2).} Suppose that $c$ is contained in $U$ but
not contained in $V$. Let $U'$ be an open subgroup of $G$ which contains $U$ 
as a subgroup of index $p$ (and $U$ is hence normal in $U'$). Let $V'$ denote
the commutator subgroup of $U'$. We claim that we may reduce to the case 
in which 
$\theta_{S,U',V'}(\xi)=\xi_{U',V'}$ holds; indeed the desired equation
holds if $V'$ contains $c$ by (Case-1). Assume that $V'$ does not contain $c$.
Then the pair $(U_1, V_1)=(U',V')$ also satisfies the condition of (Case-2), 
and recursively we may obtain a sequence of pairs 
$\{(U_i, V_i)\}_{i \in \Z_{\geq 0}}$ 
such that $(U_0, V_0)$ is equal to $(U,V)$ and 
$U_{i+1}$ contains $U_i$ as its normal
subgroup of index $p$ (each $V_i$ denotes the commutator subgroup of $U_i$). 
Therefore there exists a natural number $n$ such that
$V_n$ contains $c$ because we now assume that the commutator subgroup 
of $G$ contains $c$.

We now apply the following theorem:

\begin{thm} \label{thm:exi-u'}
The \p-adic zeta function $\xi_{U',V}$ for
 $F_V/F_{U'}$ exists uniquely as an element in $K_1(\iw{U'/V}_S)$.
\end{thm}

This is the special case of 
the deep results of J\"urgen Ritter and 
Alfred Weiss \cite{RW7}. 
For the convenience of the readers, we shall give the sketch of the proof in the
following subsection.

Now assume that Theorem~\ref{thm:exi-u'} is valid for a moment. 
Let $\can_{U',V}$ denote the canonical homomorphism 
$K_1(\iw{U'}_S) \rightarrow K_1(\iw{U'/V}_S)$. 
Note that the element $\can_{U',V} \circ \Nr_{\iw{G}_S/\iw{U'}_S}(\xi)$ 
is the \p-adic zeta function
``modulo \p-torsion'' for $F_V/F_{U'}$ by the interpolation property, 
and hence there exists an
element $\tau$ in~${U'}^f/{V'}^f$ 
such that  $\can_{U',V} \circ
\Nr_{\iw{G}_S/\iw{U'}_S}(\xi)$ coincides with  
$\tau \xi_{U',V}$ (here we remark that the \p-torsion part of
$K_1(\iw{U'/V})$ coincides with 
$({U'}^f/V^f)^{\mathrm{ab}}={U'}^f/{V'}^f$ because
$SK_1(\Z_p[{U'}^f/V^f])$ is trivial by \cite[Theorem~8.10]{Oliver}). 
Then easy calculation verifies that the equation
\begin{align*}
\xi_{U',V'} &=\theta_{S,U',V'} (\xi) =\can^V_{V'} \circ \can_{U',V}\circ 
 \Nr_{\iw{G}_S/\iw{U'}_S}(\xi) \\
 &= \can^V_{V'} (\tau \xi_{U',V}) =\tau \xi_{U',V'}
\end{align*}
holds where $\can^V_{V'} \colon K_1(\iw{U'/V}_S)
\rightarrow \iw{U'/V'}_S^{\times}$ denotes the canonical homomorphism. 
This implies that $\tau$ is trivial. On the other hand, the norm relation 
$\Nr_{\iw{U'/V}_S/\iw{U/V}_S}(\xi_{U',V})=\xi_{U,V}$
holds since
$\xi_{U',V}$ is the \p-adic zeta function for $F_V/F_{U'}$. 
Therefore we obtain the desired equation
\begin{align*}
\theta_{S,U,V}(\xi) &=\Nr_{\iw{U'/V}_S/\iw{U/V}_S} \circ \can_{U',V}
 \circ\Nr_{\iw{G}_S/\iw{U'}_S}(\xi) \\
&= \Nr_{\iw{U'/V}_S/\iw{U/V}_S}(\xi_{U',V}) = \xi_{U,V}.
\end{align*}

\noindent {\bfseries (Case-3).} Suppose that $c$ is contained in
neither $U$ nor $V$. In this case the pair 
$(U\times \langle c \rangle, V)$ satisfies the condition of (Case-2), 
and thus the equation $\theta_{S,U\times \langle c \rangle, V}(\xi)=
\xi_{U\times \langle c \rangle, V}$ holds. Then by using the commutative 
diagram
\begin{align*}
\xymatrix{
K_1(\iw{G}_S) \ar[rr]^{\theta_{S, U\times \langle c \rangle, V}}
 \ar[drr]_{\theta_{S,U,V}} & &\iw{U\times \langle c \rangle/V}_S^{\times}
 \ar[d]^{\Nr_{\iw{U\times \langle c \rangle/V}_S/\iw{U/V}_S}} \\
& &\iw{U/V}_S^{\times},
}
\end{align*}
we obtain
\begin{align*}
\theta_{S,U,V}(\xi)&=\Nr_{\iw{U\times \langle c \rangle/V}_S/\iw{U/V}_S}
 \circ \theta_{S, U\times \langle c \rangle, V}(\xi) \\
 &= 
 \Nr_{\iw{U\times \langle c \rangle/V}_S/\iw{U/V}_S}(\xi_{U\times
 \langle c \rangle, V}) =\xi_{U,V},
\end{align*}
which is the desired result.\footnote{We may derive the desired result 
for (Case-3) even if we only assume that $\theta_{S,U\times \langle c
\rangle, V} (\xi)=c^j\xi_{U\times
\langle c \rangle,V}$ holds for  certain $j$ 
(which we may verify by the arguments similar to (Case-1));
hence the essentially difficult part is just (Case-2). 
Note that $\Nr_{\iw{U\times \langle c
\rangle/V}_S/\iw{U/V}_S}(c^j)$ coincides with $(c^j)^p=1$ because
$c^j$ is contained in
the centre of $\iw{U\times \langle c
\rangle/V}_S^{\times}$.}

%%%%%%%%%%%%%%%%%%%%%%%%%%%%%%%%%%%%%%%%%%%%%%%%%%%%%%%%%%%%%%%
%
%
\subsection{Outline of the proof of Theorem~\ref{thm:exi-u'}} \label{ssc:proof}
%
%
%%%%%%%%%%%%%%%%%%%%%%%%%%%%%%%%%%%%%%%%%%%%%%%%%%%%%%%%%%%%%%%

In this subsection we shall give the rough sketch of the proof of 
the following theorem.

\begin{thm}[Ritter-Weiss] \label{thm:R-W}
Let $p$ be a positive odd prime number and $F$ a totally real number field.
Let $F_{\infty}$ be a totally real \p-adic Lie extension of $F$
 satisfying conditions $(F_{\infty}\text{-}1)$, $(F_{\infty}\text{-}2)$ and
 $(F_{\infty}\text{-}3)$ in Section~$\ref{ssc:review-iwa}$. 
Let $G$ denote the Galois group of $F_{\infty}/F$ and suppose that the
 following two conditions are satisfied$:$
\begin{enumerate}[$($\upshape i$)$]
\item the Galois group $G$ is isomorphic to the direct product 
      of a finite \p-group $G^f$
      and the commutative \p-adic Lie group $\Gamma$$;$
\item the finite part $G^f$ has an abelian subgroup $W^f$ of index $p$
      $($which is automatically normal in $G)$. 
\end{enumerate}
Then the \p-adic zeta function $\xi_{F_{\infty}/F}$ for $F_{\infty}/F$
 exists uniquely as an element in $K_1(\iw{G}_S)$.
\end{thm}

Theorem~\ref{thm:exi-u'} is the direct consequence of the claim above. 
Here we shall give the proof of this theorem which is based upon 
the method of Kazuya Kato in \cite{Kato2}. In \cite{RW7} 
J\"urgen Ritter and Alfred Weiss proved the more general claim 
in another manner (refer also to \cite{RW5, RW6}).

\begin{proof}[Sketch of the proof]
Set $W=W^f \times \Gamma$ and choose a generator $\lambda$ of the 
quotient group $G^f/W^f$. Then we obtain the splitting exact sequence
\begin{align*}
1\rightarrow W^f \rightarrow G^f \rightarrow \langle \lambda \rangle
 \rightarrow 1.
\end{align*}
Recall that Serre's  \p-adic zeta pseudomeasure 
$\xi_{\mathrm{ab}}$ (resp.\ $\xi_W$) for $F_{[G,G]}/F$ 
(resp.\ $F_{\infty}/F_W$) exists as a unique element in
 $\iw{G^{\mathrm{ab}}}_S^{\times}$ (resp.\ $\iw{W}_S^{\times}$).

\medskip
\noindent {\bfseries Step 1.} Construction of the Brauer family
 $\mathfrak{F}$. 

\nobreak
By identifying $W^f$ with a $d$-dimensional 
$\F_p$-vector space, 
the action of~$\lambda$ upon $W^f$ (which we regard as a {\em left}
 action) is described as the Jordan normal form 
$J_{\lambda}=\bigoplus_{i=1}^t J_i$ where each 
$J_i$ is the Jordan block of rank $m_i$ with eigenvalue~$1$. 
Then the summation of $\{m_i\}_{i=1}^t$ equals $d$ by definition. 
Moreover each $m_i$ is 
less than $p$ since the order of $J_{\lambda}-\id$ is just
 equal to $p$. Let $W^f=\bigoplus_{i=1}^t W_i^f$ 
be the corresponding generalised eigenspace decomposition of $W^f$.
Fix a Jordan basis $\{e_{i,j}\}_{j=1}^{m_i}$ of each
 $W_i$. Note that the abelisation of $G^f$ is identified with 
 $\bar{W}\times \langle \lambda \rangle$ where $\bar{W}$ is
 the quotient space of $W^f$ isomorphic to the subspace $W'$ of $W^f$ 
spanned by 
 $\{e_{i, m_i}\}_{i=1}^t$. 
Now let $\mathfrak{X}(W^f)$ denote the space of characters on~$W^f$. The 
cyclic group $\langle \lambda \rangle$ also acts upon 
$\mathfrak{X}(W^f)$ from the right by 
\begin{align*}
\mathfrak{X}(W^f) \times \langle \lambda \rangle \rightarrow
 \mathfrak{X}(W^f); \quad (\chi, \lambda) \mapsto
 (\chi*\lambda \colon w \mapsto \chi(\lambda w)).
\end{align*}
It is clear that the fixed subspace 
of $\mathfrak{X}(W^f)$ under this action is identified with 
$\mathfrak{X}(\bar{W})$ (use the Jordan basis above). 
By \cite[Th\'eor\`eme~17]{Serre1}, every irreducible representation of 
 $G^f=W^f \rtimes \langle \lambda \rangle$ is isomorphic to one of the 
 following types of induced representations:
\begin{itemize}
\item $\ind{G^f}{W^f}{\chi_1}$ where $\chi_1$ is an element in
 $\mathfrak{X}(W^f) \setminus \mathfrak{X}(\bar{W})$,

\item $\chi_2 \otimes \chi_2'$ where $\chi_2$ is an element in 
      $\mathfrak{X}(\bar{W})$ and $\chi_2'$
      is a character of $\langle \lambda \rangle$.
\end{itemize}
We may, however, each $\chi_2\otimes \chi_2'$ as a character of the 
abelisation $G^{f,\mathrm{ab}}$ of~$G^f$. 
Therefore an arbitrary irreducible representation of $G^f$ 
is obtained as an induced representation of a character of either $W^f$
 or $G^{f, \mathrm{ab}}$; in other words, the family 
 $\mathfrak{F}=\{ (G,[G,G]), (W, \{e \}) \}$ is a Brauer family for $G$.

\medskip
\noindent {\bfseries Step 2.} Additive theory.

\nobreak
Let $\theta_{\mathrm{ab}}^+ \colon \Z_p[[\conj{G}]]\rightarrow
 \Z_p[[G^{\mathrm{ab}}]]$ denote the abelisation homomorphism and 
 $\theta_{W}^+$ the trace homomorphism from 
 $\Z_p[[\conj{G}]]$ to $\Z_p[[W]]$. Let 
\begin{align*}
I_W=\langle \theta_W^+(w)=\sum_{k=0}^{p-1} \lambda^k w \mid w\in W^f \rangle_{\Z_p[[\Gamma]]} 
\end{align*}
be the image of $\theta_W^+$ in $\Z_p[[W]]$ and $\Phi$ the $\Z_p$-submodule of $\Z_p[[G^{\mathrm{ab}}]]\times \Z_p[[W]]$ 
consisting of all pairs $(y_{\mathrm{ab}}, y_W)$ satisfying the
 following two conditions:
\begin{itemize}
\item the equation
      $\Tr_{\Z_p[[G^{\mathrm{ab}}]]/\Z_p[[W/W_{\mathfrak{c}}]]}(y_{\mathrm{ab}})=\can_{W,W_{\mathfrak{c}}}(y_W)$
      holds;
\item the element $y_W$ is contained in $I_W$
\end{itemize}
where $W_{\mathfrak{c}}$ denotes the intersection of $W$ and $[G,G]$, 
and $\can_{W,W_{\mathfrak{c}}}$ denotes the canonical map
 $\Z_p[[W]] \rightarrow \Z_p[[W/W_{\mathfrak{c}}]]$. 
Then $\theta^+=(\theta_{\mathrm{ab}}^+, \theta_W^+)$ induces an
 isomorphism between $\Z_p[[\conj{G}]]$ and $\Phi$; indeed we may prove the
 injectivity of $\theta^+$ by the same argument as that in the proof of
 Proposition~\ref{prop:addtheta}. Moreover we may also prove that 
\begin{align*}
\theta^+_W|_{\Z_p[[\conj{W;G}]]} \colon \Z_p[[\conj{W;G}]] \rightarrow \Z_p[[W]]
\end{align*}
is injective where $\conj{W;G}$ denotes the subset of $\conj{G}$
 consisting of all conjugacy classes contained in $W$ 
(namely $I_W$ is a 
{\em free} $\Z_p[[\Gamma]]$-submodule of $\Z_p[[W]]$ spanned by
 $\{\theta^+_W([w]) \mid [w]\in \conj{W^f;G^f} \}$). 
Now let $(y_{\mathrm{ab}}, y_W)$ be an element in $\Phi$ described as
the following $\iw{\Gamma}$-linear combinations:
\begin{align*}
y_{\mathrm{ab}}=\sum_{\bar{w}\in \bar{W}}
 \sum_{j=0}^{p-1} a_{\bar{w},j} (\bar{w}, \lambda^j), \qquad 
y_W = \sum_{[w]\in \conj{W^f;G^f}}a_{[w]} \theta_W^+([w]).
\end{align*}
Then the element $y=\sum_{\bar{w}\in \bar{W}}
 \sum_{j=0}^{p-1} a_{\bar{w},j} [(w', \lambda^j)]+
 \sum_{[w]\in \conj{W^f;G^f}}  \! \!  a_{[w]} [w]$ satisfies
$\theta^+(y)=(y_{\mathrm{ab}}, y_W)$ where $w'$ is an element in $W'$ 
corresponding to $\bar{w}$ via the (non-canonical) isomorphism 
 $W' \xrightarrow{\sim} \bar{W}$. Clearly $\Phi$ contains the image
 of~$\theta^+$, and they thus 
 coincide.

\medskip
\noindent {\bfseries Step 3.} Logarithmic theory.

\nobreak
First note that $I_W$ is {\em an ideal} of the $\Z_p$-algebra 
 $\Z_p[[W]]^{\langle \lambda \rangle}$ ---the maximal subalgebra 
 of $\Z_p[[W]]$ fixed by the action of $\langle \lambda \rangle$---,
 and thus $I_W$ obviously contains $I_W^2$. Moreover $I_W$ is 
 contained in the augmentation kernel 
\begin{align*}
J_W=\ker (\Z_p[[W]] \xrightarrow{\aug_W} \Z_p[[\Gamma]]
 \rightarrow \F_p[[\Gamma]])
\end{align*}
because $\aug_W\circ \theta^+_W(w)=p$ holds for each $w$ in
 $W^f$. By the same
 argument as that in Proposition~\ref{prop:log-j}, we may prove that $1+J_W$ is
 a multiplicative subgroup of $\iw{W}^{\times}$ and the \p-adic
 logarithm induces a homomorphism
\begin{align} \label{eq:log-w}
\log \colon 1+J_W \rightarrow \iw{W}.
\end{align}

Next we shall prove that the restriction of the map 
(\ref{eq:log-w}) to the multiplicative subgroup $1+I_W$ is
 injective.\footnote{By the argument in the proof of
 Proposition~\ref{prop:log-j}, the element $x^p$ is especially contained
 in $p\iw{W}\cap I_W$ for $x$ in $I_W$. It implies that
 $(1+x)^{-1}=\sum_{j=0}^{\infty} (-x)^j$ converges \p-adically
 in $1+I_W$, and hence $1+I_W$ is a multiplicative subgroup of 
 $\piw{W}^{\times}$.} 
The kernel of the homomorphism (\ref{eq:log-w}) coincides 
 with $\mu_p(\iw{W})$ by Proposition~\ref{prop:log-j}, 
 which is isomorphic to 
 $\mu_p(\iw{\Gamma}) \times W^f$ by the theorem of Graham Higman
 \cite{Higman} and Charles Terence Clegg Wall \cite[Theorem~4.1]{Wall}. 
Note that $W^f$ is not contained in $1+I_W$; 
 if $x$ is an element in $I_W$, its coefficient of the unit of $W^f$ 
 is divisible by $p$ because $p$ is one of the free basis of 
 $I_W$ over $\iw{\Gamma}$. Therefore $w-1$ is 
 not contained in $I_W$ for each $w$ in $W^f$ 
 except for the unit. 
 Combining this fact with triviality of 
 $\mu_p(\iw{\Gamma})$, the homomorphism 
 $1+I_W \rightarrow \iw{W}$ induced by the \p-adic logarithm 
 is injective.

\medskip
\noindent {\bfseries Step 4.} Translation. 

\nobreak
Let $\theta_{\mathrm{ab}} \colon K_1(\iw{G})\rightarrow
 \iw{G^{\mathrm{ab}}}^{\times}$ denote the abelisation homomorphism and 
 $\theta_{W}$ the norm homomorphism
 $\Nr_{\iw{G}/\iw{W}}$. Let $\Psi$ be the subgroup of 
 $\iw{G^{\mathrm{ab}}}^{\times}\times \iw{W}^{\times}$ 
 consisting of all pairs $(\eta_{\mathrm{ab}}, \eta_W)$ satisfying the
 following two conditions:
\begin{itemize}
\item the equation
      $\Nr_{\iw{G^{\mathrm{ab}}}/\iw{W/W_{\mathfrak{c}}}}(\eta_{\mathrm{ab}})=\can_{W,W_{\mathfrak{c}}}(\eta_W)$
      holds;
\item the congruence 
$\eta_W \equiv \varphi(\eta_{\mathrm{ab}}) \, \mod{I_W}$ holds.
\end{itemize}
Then we may prove that $\theta=(\theta_{\mathrm{ab}}, \theta_W)$ induces
 a surjective homomorphism $K_1(\iw{G}) \rightarrow \Psi$ with kernel
 $SK_1(\Z_p[G^f])$ by the same arguments as those in 
 Sections \ref{ssc:contain}--\ref{ssc:isomorphy}; more precisely,
\begin{itemize}
\stepcounter{enumi}
\item[$-$] the congruence
      $\theta_W(\eta) \equiv \varphi(\theta_{\mathrm{ab}}(\eta)) \,
	   \mod{I_W}$  holds for  $\eta$ in~$K_1(\iw{G})$ 
	   by direct calculation 
      (refer to \cite[Section 5, Lemma~1.7]{Taylor});
      in other words, 
      $\Psi$ contains the image of $\theta$;
\item[$-$] for an arbitrary element $(\eta_{\mathrm{ab}}, \eta_W)$ in $\Psi$, 
      there exists a unique element 
      $y$ in $\Z_p[[\conj{G}]]$ which satisfies 
      \begin{align*}
       \theta^+(y)=(\Gamma_{G^{\mathrm{ab}}}(\eta_{\mathrm{ab}}), \log(
       \varphi(\eta_{\mathrm{ab}})^{-1}\eta_W))
      \end{align*}
	   (recall that $\theta^+$ is isomorphic). By applying 
	   theory upon integral
	   logarithmic homomorphisms, we may prove that 
	   there exists a unique element $\eta$ in $K_1(\iw{G})$ 
	   satisfying $\Gamma_G(\eta)=y$
	   up to multiplication by an element in $SK_1(\Z_p[G^f])$. 
	   The desired relation $\theta(\eta)=(\eta_{\mathrm{ab}},\eta_W)$ 
	   follows from the construction of $\eta$
	   (mimic the argument in the proof of 
	   Proposition~\ref{prop:surjection}).\footnote{Here we use the
	   injectivity of $\log \colon 1+I_W \rightarrow \iw{W}$.}
\end{itemize}
 Note that $SK_1(\Z_p[G^f])$ is trivial because $G^f$ has an abelian normal
 subgroup with cyclic quotient (see \cite[Theorem~8.10]{Oliver}),
 and therefore the induced homomorphism $\theta\colon K_1(\iw{G})\rightarrow
 \Psi$ is  an isomorphism.

\medskip
\noindent {\bfseries Step 5.} Localised version.

\nobreak
Let $\theta_{S,\mathrm{ab}} \colon K_1(\iw{G}_S)\rightarrow
 \iw{G^{\mathrm{ab}}}_S^{\times}$ be the abelisation homomorphism and 
 $\theta_{S,W}$ the norm homomorphism
 $\Nr_{\iw{G}_S/\iw{W}_S}$. 
 Set $I_{S,W}=I_W \otimes_{\iw{\Gamma}}\iw{\Gamma}_{(p)}$ and let 
 $\Psi_S$ be the subgroup of 
 $\iw{G^{\mathrm{ab}}}_S^{\times}\times \iw{W}^{\times}_S$ 
 consisting of all pairs $(\eta_{S,\mathrm{ab}}, \eta_{S,W})$ satisfying the
 following two conditions:
\begin{itemize}
\item the equation
      $\Nr_{\iw{G^{\mathrm{ab}}}_S/\iw{W/W_{\mathfrak{c}}}_S}(\eta_{S,\mathrm{ab}})=\can_{W,W_{\mathfrak{c}}}(\eta_{S,W})$
      holds;
\item the congruence $\eta_{S,W} \equiv \varphi(\eta_{S,\mathrm{ab}}) \,
      \mod{I_{S,W}}$ holds.
\end{itemize}
Then we may prove that $\Psi_S$ contains the image of
 $\theta_S=(\theta_{S,\mathrm{ab}}, \theta_{S,W})$ by mimicking 
 the argument in Step~4 
 (see also Section~\ref{sc:localise}). 
One of the most important points is the following fact: since Higman-Wall's
 theorem also holds for $R[W^f]$,\footnote{This is because $R$ is the \p-adic
 completion of  $\iw{\Gamma}_{(p)}$; 
 refer to \cite[the remark after Theorem~4.1]{Wall}.} 
 the intersection of $1+I_W^{\, \widehat{\,}}$ and 
 $\mu_p(R[W^f])=\mu_p(R) \times W^f$ is trivial 
 if we set $I_W^{\,
 \widehat{\,}}=I_W\otimes_{\iw{\Gamma}}R$ (use triviality of $\mu_p(R)$ 
 and mimic the argument in Step~3), and hence the induced
 homomorphism 
\begin{align*}
\log \colon 1+I_W^{\, \widehat{\,}} \rightarrow R[W^f]
\end{align*}
is also {\em injective}.

\medskip
\noindent {\bfseries Step 6.} Verification of the congruence. 

\nobreak
Under these settings, we may prove the claim if we verify the congruence
\begin{align*}
\xi_{W} &\equiv \varphi(\xi_{\mathrm{ab}}) & \mod{I_{S,W}}
\end{align*}
by virtue of Burns' technique (see Section~\ref{sc:burns}). 
It is, however, no other than a special case of the Ritter-Weiss'
 congruence \cite[Theorem]{RW6}.
\end{proof}

%%%%%%%%%%%%%%%%%%%%%%%%%%%%%%%%%%%%%%%%%%%%%%%%%%%%%%%%%%%%%%%

\end{document}